\documentclass[a4paper,11pt, reqno]{amsart}
\usepackage{epstopdf}
\usepackage{amssymb}
\usepackage{textcomp}
\usepackage{amsmath, amsthm, amscd, amsfonts}
\usepackage{graphicx, float}
\usepackage{mathrsfs}
\usepackage{bbm, bm, upgreek}

\usepackage{comment}
\usepackage{stmaryrd}
\usepackage[left= 2.8 cm, right= 2.8 cm, top= 3 cm, bottom=2.7 cm, foot= 1 cm]{geometry}

\usepackage{hyperref}
\hypersetup{colorlinks=true, citecolor=blue}
\usepackage{color}

\usepackage{tikz}

\newcommand{\mb}{\mathbb}
\newcommand{\mbm}{\mathbbm}
\newcommand{\mf}{\mathfrak}
\newcommand{\mc}{\mathcal}

\newcommand{\mbms}{\mathbbmss}
\newcommand{\ov}{\overline}

\newcommand{\wt}[1]{\widetilde{#1}}
\newcommand{\wh}[1]{\widehat{#1}}

\newcommand{\lb}{\llbracket}
\newcommand{\rb}{\rrbracket}

\newcommand{\bthm}{\begin{thm}}
\newcommand{\ethm}{\end{thm}}

\newcommand{\blem}{\begin{lemma}}
\newcommand{\elem}{\end{lemma}}

\newcommand{\bcor}{\begin{cor}}
\newcommand{\ecor}{\end{cor}}

\newcommand{\bprop}{\begin{prop}}
\newcommand{\eprop}{\end{prop}}

\newcommand{\brmk}{\begin{rem}}
\newcommand{\ermk}{\end{rem}}

\newcommand{\bpf}{\begin{proof}}
\newcommand{\epf}{\end{proof}}

\newcommand{\bdf}{\begin{defn}}
\newcommand{\edf}{\end{defn}}

\numberwithin{equation}{section}

\def\R{\mathbb{R}}

\def\cT{\mathcal{T}}

\def\mf{\mathfrak}

\def\mb{\mathbf}
\def\ol{\overline}

\def\vp{\varphi}

\def\w{\omega}

\def\vp{\varphi}

\def\xkm2{\overline{X}_{k-2}}

\makeatletter % `@' now normal "letter"
\@addtoreset{equation}{section}
\makeatother  % `@' is restored as "non-letter"

\numberwithin{equation}{section}
\newcommand{\beq}{\begin{equation}}
\newcommand{\eeq}{\end{equation}}

\title{Gauged Floer homology and spectral invariants}

\author[Wu]{Weiwei Wu}

\address{Centre de recherches math\'ematiques \\
      Universit\'e de Montr\'eal\\
      Pavillon Andr\'e-Aisenstadt\\
      2920 Chemin de la tour, Room 5357
      Montr\'eal (Qu\'ebec) H3T 1J4
			}
\email{wuweiwei@crm.umontreal.ca}

\author[Xu]{Guangbo Xu}
\address{
340 Rowland Hall\\
Department of Mathematics \\
University of California, Irvine\\
Irvine, CA 92697 USA
}
\email{guangbox@math.uci.edu}

\date{\today}

\begin{document}

\newtheorem{thm}{Theorem}[section]
\newtheorem{lemma}[thm]{Lemma}
\newtheorem{cor}[thm]{Corollary}
\newtheorem{prop}[thm]{Proposition}

\theoremstyle{definition}
\newtheorem{defn}[thm]{Definition}
\theoremstyle{remark}
\newtheorem{rem}[thm]{Remark}
\newtheorem{hyp}[thm]{Hypothesis}
\newtheorem{example}[thm]{Example}

\begin{abstract}
We define a version of spectral invariant in the vortex Floer theory for a $G$-Hamiltonian manifold $M$.  This defines potentially new (partial) symplectic quasi-morphism and quasi-states when $M//G$ is not semi-positive.  We also establish a relation between vortex Hamiltonian Floer homology \cite{Xu_VHF} and Woodward's quasimap Floer homology \cite{Woodward_toric} by constructing a closed-open string map between them.  This yields applications to study non-displaceability problems of subsets in $M//G$.
%In the first part of this paper, for certain classes of Hamiltonian %$G$-manifolds of which the symplectic reductions include compact toric %manifolds, we construct an analogue of the Oh-Schwarz' spectral %invariants in Hamiltonian Floer theory of the symplectic reductions, via %the vortex Floer homology constructed in \cite{Xu_VHF}. Our spectral %invariants satisfy all desired properties except for the normalization %and symplectic invariance.

%In the second part of this paper, we use the spectral numbers to %construct analogues of the (partial) symplectic quasi-states and %quasi-morphisms of Entov-Polterovich (\cite{Entov_Polterovich_1, %Entov_Polterovich_2, Entov_Polterovich_3}), and reproved the main %theorem of \cite{FOOO_spectral} (in the version without bulk %deformation) in the current setting.
\end{abstract}

\maketitle

\setcounter{tocdepth}{1}
\tableofcontents

\section{Introduction}

The main theme of this paper lies in the intersection of several different directions of interests in symplectic geometry. Classically, one of the fundamental questions driving the development of symplectic geometry is to understand the fixed point set of a Hamiltonian diffeomorphism, that is, the Arnold's conjecture. A relative version of Arnold's conjecture considers the intersection of a Lagrangian submanifold with its image under a Hamiltonian diffeomorphism. It is by now well-known that this line of problems is intertwined with Gromov-Witten invariants, in a way that the Hamiltonian Floer cohomology ring, generated by fixed points of Hamiltonians, is isomorphic to the quantum cohomology ring defined by counting holomorphic spheres \cite{PSS}. Although it is more complicated in general, but there is a rough correspondence in the open string (Lagrangian boundary) cases, see \cite{FOOO_Book}\cite{Biran_Cornea} for example.

Another important mainstream in the research of symplectic geometry owes to the discovery of a bi-invariant metric, Hofer's metric on ${\rm Ham}(M)$ for arbitrary symplectic manifolds \cite{Hofer_metric}\cite{LM95}. This in turn is closely tied to an extra action filtration structure on the Floer complex, which is usually referred to as the \textit{spectral theory}. In particular, Viterbo \cite{Vi92} constructed a spectral invariant for $R^{2n}$, which was later generalized by Schwarz and Oh, see for example, \cite{Schwarz_spectral}\cite{Oh_2005}.

We recall briefly the construction of the chain level theory here. For a generic time dependent Hamiltonian $H_t: X \to {\mb R}\ (t\in S^1)$, we can construct the Floer homology group $HF(X, H)$ (with some appropriate coefficient ring $\Lambda$). $HF(X, H)$ is the homology of certain chain complex $CF(X, H)$ of $\Lambda$-modules with some distinguished generators; there is an action functional ${\mc A}_H$ (which is sensitive in $H$) defined on those generators. $HF(X, H)$ is canonically isomorphic to the singular homology of the underlying manifold $X$ in the same coefficient ring. Then, any homology class $a\in H(X; \Lambda)\simeq HF(X, H)$ is represented by a Floer chain
\begin{align*}
{\mf X}:= \sum_{{\mf x}} a_{{\mf x}} {\mf x} \in CF(X, H).
\end{align*}
Then we define the spectral number $c(a, H)$ formally as
\begin{align*}
c(a, H) = \inf \big\{ {\mf v}({\mf X})\ |\ [{\mf X}] = a \in HF(X, H)\big\},\ {\mf v}({\mf X}) = \max\big\{ {\mc A}_H({\mf x})\ |\ a_{{\mf x}} \neq 0 \big\}.
\end{align*}
One can show that $c(a, H)$ is independent of various choices made in order to define the Floer chain complex, and it only depends on the Hamiltonian path $\wt\psi_H \in \wt{\rm Ham}(M)$.  An important property is that the spectral invariant lower bounds the Hofer norm of a Hamiltonian diffeomorphism.

A remarkable connection was built between the filtered Floer theory and the displaceability problem of subsets of symplectic manifolds by Entov and Polterovich in their series of papers, see for example \cite{Entov_Polterovich_1, Entov_Polterovich_2, Entov_Polterovich_3}.  In a nutshell, they combine the ideas from dynamical systems, function theory as well as symplectic geometry to construct certain functionals on ${\rm Ham} (M)$, called the {\it (partial) symplectic quasi-states} and {\it Calabi quasimorphisms}. They used these objects to study symplectic intersections and introduced the notion of {\it heavy} and {\it superheavy} subsets of symplectic manifolds. Later such constructions are extended by \cite{FOOO_spectral}, \cite{Ush08} to include the contributions of the {\it big} quantum cohomology in general symplectic manifolds.

The primary purpose of this paper is to explore a similar picture in the vortex context.  The \textit{vortex Floer theory}, as the last strand of the braid, is defined when there is a compact Lie group $G$ acting on the symplectic manifold $M$.  The solutions of vortex equations exhibited similar structures to those of Floer equations, thus the relation between vortex theory and Hamiltonian/Lagrangian Floer theory has attracted much interests since.  In the Gromov-Witten context, a \textit{vortex equation} was studied by Mundet and Cieliebak-Gaio-Salamon \cite{Mundet_thesis, Cieliebak_Gaio_Salamon_2000, Mundet_2003, Cieliebak_Gaio_Mundet_Salamon_2002} to define a \textit{Hamiltonian Gromov-Witten invariant}. In \cite{Cieliebak_Gaio_Salamon_2000} a version of vortex Hamiltonian Floer theory (closed string) was proposed, which was only realized very recently by the second named author \cite{Xu_VHF, Xu_Floer_2, Xu_Floer_3}. The open string theory counterpart of symplectic vortex equations were considered first in \cite{Frauenfelder_thesis} to obtain new non-displaceability results.  Recently Chris Woodward defined a quasimap Floer cohomology using an adiabatic limit version of the equation in \cite{Woodward_toric}, and have achieved great advance in understanding the non-displaceability of the toric fibers.

In this paper, we try to understand two aspects of the vortex Floer homology theory constructed by the second named author \cite{Xu_VHF}: (1) its filtration structure, and (2) justify the fact that the vortex Hamiltonian Floer theory is an appropriate closed string theory versus Woodward's quasi-map Floer theory.  Combining these two aspects of considerations, we are able to study the non-displaceability problems and quasi-morphism(state) theory in the vortex context.

As the first step, we define the spectral invariants in vortex context.  The result can be vaguely stated as follows.
\begin{thm}\label{t:spectral}
Let $(M, \omega, \mu)$ be an aspherical, equivariantly convex, Hamiltonian $G$-manifold with compact smooth symplectic quotient (see conditions {\bf (H1)}--{\bf (H3)} at the beginning of Section \ref{section2}). Then there exists a function
\begin{align*}
c: VHF(M, \mu) \times \wt{\rm Ham} ( \ov{M}) \to {\mb R}
\end{align*}
satisfying usual axioms of the spectral invariants except for the normalization and symplectic invariance (replaced by Hamiltonian invariance).
\end{thm}

The proof of the above theorem is not a straightforward adaption from the Hamiltonian Floer case: among others, one of the common difficulty in recovering usual properties of spectral invariants in ordinary Floer theory lies in that, \textit{a priori}, we have vortex Floer trajectories which do not descend to the symplectic quotient.  An adiabatic limit technique is hence developed to handle the problem.  Details of the proof was provided in Section \ref{s:Proof}.

The spectral invariant we define is reduced to the usual spectral invariant when $G=\{id\}$ and $M$ is compact aspherical.  It is not clear to the authors at the present stage whether the absence of normalization and symplectic invariance properties are only technical in general.  In particular, the normalization will follow from a vortex version of PSS isomorphism (in principle the PSS isomorphism can always be constructed using the virtual technique but it is not yet known if a construction without virtual technique is possible at the time of writing).  However, a key observation is that this weaker version of vortex spectral invariant defines a genuine partial symplectic quasi-morphism (quasi-state) with all formal properties, hence making sense of the notion of heaviness. This will be explained in Section \ref{s:QMQS}.

There are two remarks we should make at this point.  It is notable that we have constrained us to the weaker notion of partial quasi-morphisms (quasi-states) instead of the genuine quasi-morphisms (quasi-states), and only heaviness instead of superheaviness.  The reason is that, while the proof of properties of superheavy sets mostly follows formally from Entov and Polterovich's original approach, we do not have a good example where vortex Floer cohomology can be computed and confirmed to have a field idempotent element other than the Fano cases, in which the vortex Floer theory is identified with the usual Hamiltonian Floer theory.

However, one should note that our partial quasi-morphisms/states alludes a relation to Borman's reduction on quasi-morphism/states on $M$ to $\ov{M}$, see \cite{Borman}.  There is no known relation between $HF(M, H)$ and $VHF(M, H)$, so it would be interesting if one could still establish instead a reduction relation on the quasi-morphism/state level, which might also shed lights on the relation between idempotents of the two rings.

The second part of the paper investigates the relation between $VFH(M)$ and Woodward's quasi-map Floer theory with focus on examples of toric manifolds.  Our main result is:

\bthm\label{t:clopen}  For a $G$-manifold $M$ as in Theorem \ref{t:spectral}, $L\subset M$ a compact $G$-invariant Lagrangian submanifold contained in $\mu^{-1}(0)$ and a brane structure $b\in H^1(L,\Lambda_0)$, there is a closed-open map
\begin{align*}
\mbms{co}: VHF(M) \to QHF(L^b)
\end{align*}
which is a $\Lambda$-linear map sending the identity $\mbms{1}_{\mbms H}$ to the identity $\mbms{1}_{L^b}$.
\ethm

This seems an appropriate correspondence of the closed-open map in the usual Floer theory (see \cite{FOOO_Book}). The reason that vortex Hamiltonian Floer theory is connected with Woodward's quasimap Lagrangian Floer theory is, when consider the symplectic vortex equation on the half cylinder $[0, +\infty) \times S^1$ with boundary lying in a $G$-invariant Lagrangian submanifold, the boundary bubbles are exactly ``quasidisks'' (disk vortices with zero area form).

A concrete consequence of the existence of the closed-open map is

\bcor\label{c:weakArnold} Any generic Hamiltonian diffeomorphism on a toric manifold has at least one fixed point.
\ecor

%\bpf It is known that the Hori-Vafa superpotential has at least one %critical point, which means there is a toric fiber with non-zero %quasi-map Floer cohomology (Theorem \ref{thma8}).  From Theorem %\ref{t:clopen}, this implies the non-triviality of $VFH(M)$, yielding %the assertion.

%\epf
\begin{rem}
We would like to emphasize that the proof of the above corollary does not involve virtual techniques, and the toric manifold in the corollary can have arbitrary negative curves.
This result was firstly proved as \cite[Theorem 8.3]{Cieliebak_Gaio_Mundet_Salamon_2002} under the assumption of {\bf (H1)}, {\bf (H3)}, while allowing the symplectic quotient to be an orbifold. Their argument is similar to that of Gromov \cite{Gromov_1985}, by degenerating the symplectic vortex equation on a sphere with Hamiltonian perturbations. In the orbifold setting, it is more convenient to use the notion of ``$G$-relative 1-periodic orbits'' in $\mu^{-1}(0)$ instead of using 1-periodic orbits in the symplectic quotient. One should note also that the expected (but unproved) PSS isomorphism for vortex Floer theory implies the above result.
\end{rem}

%We would like to emphasize that the proof of the above corollary does %not involve virtual techniques, and the toric manifold in the corollary %can have arbitrary negative curves, which seems new in the literature.
As another application of the closed-open map, following \cite{Biran_Cornea_09}, \cite{Entov_Polterovich_3} and \cite{FOOO_spectral}, we conclude that all critical fibers of Hori-Vafa superpotential on a toric manifold are heavy with respect to the partial quasi-morphism (quasi-state) defined by the identity in $VFH(\ov M)$ in Section \ref{s:QMQS}.

\bthm\label{t:heavy} All Hori-Vafa critical fibers in a toric manifolds are heavy.
\ethm

We also considered a vortex version of Entov-Polterovich's theorem, which says the product of two heavy sets is again heavy, see Theorem \ref{thm69}.  Note that this includes the case when one of the product factor is heavy in Entov-Polterovich's original sense, by setting $G_1=\{id\}$, which is a ``hybrid case".  In particular this implies heavy sets are stably non-displaceable.

 %We show that the product of a Hori-Vafa critical fiber and any heavy %set in the regular Floer theorectic sense, is again heavy with respect %to some partial symplectic quasi-state, hence stably non-displaceable.

\subsubsection*{Acknowledgements}

We would like to thank Kenji Fukaya, Yong-Geun Oh for stimulating discussions.  The second author would also like to thank Chris Woodward for many detailed explanations on quasi-map Floer homology.

\section{Vortex Floer Homology}\label{section2}

In this section we review the definition of vortex Floer homology from \cite{Xu_VHF} for the convenience of the reader. Besides fixing notations, we unify the definition of the continuation map and the pair-of-pants product under the framework of symplectic vortex equations over surfaces with cylindrical ends. The reader may refer to \cite{Xu_VHF} for more details of the definition of the vortex Floer chain complex and to \cite{Xu_Floer_3} for details of the ring structure.

We start by giving an incomplete list of our notations and conventions in this paper.

\begin{itemize}
\item $G$ is a compact connected Lie group.  $(M, \omega, \mu)$ is a Hamiltonian $G$-manifold of dimension $m$. For any $\xi \in {\mf g}$, the vector field ${\mc X}_\xi$ defined by
%${\mf x}_+ \rb; H, J)$ is one-dimensional only when
\begin{align*}
{\mc X}_\xi(x) = \left. \frac{d}{dt} \right|_{t= 0} e^{t \xi} x.
\end{align*}
The moment map $\mu: M \to {\mf g}^*$ is defined as:
\begin{align*}
d ( \mu(\xi) ) = \omega({\mc X}_\xi, \cdot) \in \Omega^1(M).
\end{align*}

\item The following will be our standing assumptions for $(M, \omega, \mu)$: (cf. the main assumptions of \cite{Cieliebak_Gaio_Mundet_Salamon_2002}).

{\bf (H1)} $(M, \omega)$ is aspherical. Namely, for any smooth map $f: S^2 \to M$, $\displaystyle \int_{S^2} f^*\omega = 0$.

{\bf (H2)} $\mu$ is a proper map, $0\in {\mf g}^*$ is a regular value of $\mu$ and the restriction of the $G$-action to $\mu^{-1}(0)$ is free.

{\bf (H3)} (cf. \cite[Definition 2.6]{Cieliebak_Gaio_Mundet_Salamon_2002}) If $M$ is noncompact, then there exists a pair $({\mf f}, {\mf J})$, where ${\mf f}: M \to [0, +\infty)$ is a $G$-invariant proper function, ${\mf J}$ is an almost complex structure on $M$, such that there exists a constant ${\mf c}_0 >0$ and
\begin{align*}
{\mf f}(x) \geq {\mf c}_0 \Longrightarrow \langle \nabla_\xi \nabla {\mf f}(x), \xi \rangle + \langle \nabla_{{\mf J}\xi} \nabla {\mf f}(x), {\mf J} \xi \rangle \geq 0,\ d{\mf f}(x) \cdot {\mf J} {\mc X}_{\mu(x)} \geq 0,\ \forall \xi \in T_x M.	
\end{align*}
Here $\nabla$ is the Levi-Civita connection of the metric defined by $\omega$ and ${\mf J}$.

\item $\ov{M}:= \mu^{-1}(0) / G$ is the symplectic quotient of dimension $\ov m$, which is always assumed to be smooth (see {\bf (H2)} below). $\ov{M}$ admits a canonically induced symplectic form $\ov{\omega} \in \Omega^2(\ov{M})$.
    %Let $\ov{m}$ be the complex dimension of $\ov{M}$.

\item We consider $G$-invariant Hamiltonian functions $H = (H_t)_{t\in S^1}$ where for each $t\in S^1$, $H_t: M \to {\mb R}$. We also consider their descendants $\ov{H} = (\ov{H}_t)_{t\in S^1}$ by $G$-quotients, where for each $t\in S^1$, $\ov{H}_t: \ov{M} \to {\mb R}$. The Hamiltonian vector field of $H_t$ is defined by $\omega(Y_{H_t}, \cdot ) = d H_t$.

\item The set of $G$-invariant Hamiltonian functions and diffeomorphisms are denoted respectively as
\begin{align*}
{\mc H}_G(M)\text{ and }{\rm Ham}_G(M),
\end{align*}
and similarly for the set of non-equivariant Hamiltonian functions and diffeomorphism:
\begin{align*}
{\mc H}(\ov{M})\text{ and } {\rm Ham}(\ov M),
\end{align*}
while the universal cover of the Hamiltonian diffeomorphism groups are denoted as
\begin{align*}
\wt{\rm Ham}(\ov M)\text{ and }\wt{\rm Ham}_G(M).
\end{align*}

\item From time to time we impose the following constraints on $H = (H_t)_{t\in S^1}\in {\mc H}_G(M)$:

\vspace{0.3cm}

{\bf (H4)} The Hamiltonian diffeomorphism on $\ov{M}$ induced from $\ov{H}$ is nondegenerate.

\vspace{0.3cm}

\noindent We then denote different classes of Hamiltonian functions as:
\begin{align*}
{\mc H}^*(\ov{M}) := \Big\{ (\ov{H}_t) \in {\mc H}(M)\ |(\ov{H}_t)\text{ satisfies {\bf(H4)}} \Big\},
\end{align*}
\begin{align*}
{\mc H}(\ov{M})_0 := \Big\{ (\ov{H}_t) \in {\mc H}(M)\ |\  \forall t\in S^1,\ \int_{\ov{M}} \ov{H}_t (\ov\omega)^m = 0. \Big\},
\end{align*}
\begin{align*}
 {\mc H}^*(\ov{M})_0 & := {\mc H}^*(\ov{M}) \cap {\mc H}(\ov{M})_0.
\end{align*}
\begin{align*}
{\mc H}_G^*(M):= \Big\{ H \in {\mc H}_G(M)\ |\ \ov{H} \in {\mc H}^*(\ov{M}) \Big\}.
\end{align*}
\begin{align*}
{\mc H}_G^*(M)_0:= \Big\{ H \in {\mc H}_G(M)\ |\ \ov{H} \in {\mc H}^*(\ov{M})_0 \Big\}.
\end{align*}
\end{itemize}

\subsection{Equivariant topology}

%To define the Novikov ring and the vortex Floer chain complex $VCF_*(M, %\mu; H)$ we need to review some basic equivariant topology.

We briefly recall notions which will appear in the paper from basic equivariant theory.

\subsubsection*{Equivariant spherical classes}
%The Borel construction of $M$ acted by $G$ is $M_G:= EG\times_G M$, where %$EG \to BG$ is a universal $G$-bundle over the classifying space $BG$. Then %the equivariant (co)homology of $M$ is defined to be the ordinary %(co)homology of $M_G$, denoted by $H_*^G(M)$ for homology and $H^*_G(M)$ %for cohomology.

For any manifold $N$, we denote by $S_2(N)$ to be the image of the Hurwitz map $\pi_2(N) \to H_2(N; {\mb Z})$.  Let $M_G=EG\times_G M$, we define $S_2^G(M):= S_2(M_G)$. We can view $S_2(M)$ as a subset of $S_2^G(M)$. Geometrically, it is convenient to represent a generator of $S_2^G(M)$ by the following object: a smooth principal $G$-bundle $P \to S^2$ and a smooth section $\phi: S^2\to P\times_G M$. We denote the class of the pair $(P, \phi)$ to be $[P, \phi]\in S_2^G(M)$.

\subsubsection*{Equivariant symplectic form and equivariant Chern numbers}

The equivariant cohomology of $M$ can be computed using the equivariant de Rham complex $( \Omega^*(M)^G, d^G )$. In $\Omega^2(M)^G$, there is a distinguished closed form $\omega^G = \omega - \mu$, called the equivariant symplectic form, which represents an equivariant cohomology class.

If $[P, u] \in S_2^G(M)$, then the pairing $\langle [\omega^G], [P, u] \rangle \in {\mb R}$ can be computed in the following way. Choose any smooth connection $A$ on $P$. Then there exists an associated closed 2-form $\omega_A $ on $P\times_G M$, called the {\bf minimal coupling form}. If we trivialize $P$ locally over a subset $U\subset S^2$, such that $A= d + \alpha$, $\alpha \in \Omega^1(U, {\mf g})$ with respect to this trivialization, then $\omega_A$ can be written as $\omega_A= \pi^* \omega - d( \mu\cdot \alpha)\in \Omega^2( U \times M)$. We have
\begin{align*}
\langle [\omega^G], [P, u] \rangle = \int_{S^2} u^* \omega_A.
\end{align*}
On the other hand, any $G$-invariant almost complex structure $J$ on $M$ makes $TM$ an equivariant complex vector bundle. So we have the equivariant first Chern class $c_1^G:= c_1^G(TM) \in H^2_G(M; {\mb Z})$. We define
\begin{align*}
N_2^G(M) = {\rm Ker} [\omega^G] \cap {\rm Ker}  c_1^G \subset S_2^G(M),\ \Gamma:= S_2^G(M)/ N_2^G(M).
\end{align*}

%\subsubsection*{Kirwan maps}

%The cohomological Kirwan map is a {\it surjective} morphism
%\begin{align*}
%\kappa: H^*( M; {\mb R}) \to H^*( \ov{M}; {\mb R}).
%\end{align*}
%Here we take ${\mb R}$-coefficients for simplicity. It is easy to check that
%\begin{align*}
%\kappa( [\omega^G]) = [\ov{\omega}] \in H_2(\ov{M}; {\mb R}),\ \kappa(c_1^G) = c_1(T\ov{M}) \in H_2( \ov{M}; {\mb R}).
%\end{align*}
%$[\omega^G]$ and $c_1^G(TM)$ can be viewed as elements of the dual space of $S_2^G(M)$. We define

\subsection{The space of loops and the action functional}\label{s:actionFunctional}

Let $\wt{\mc L}$ be the space of smooth contractible parametrized loops in $M \times {\mf g}$ and a general element of $\wt{\mc L}$ is denoted by
\begin{align*}
\wt{x}:= (x, f): S^1 \to M \times {\mf g}.
\end{align*}

Let $\wt{\mf L}$ be a covering space of $\wt{\mc L}$, consisting of triples ${\mf x}:= ( x, f, [w] )$ where $\wt{x}= (x, f) \in \wt{\mc L}$ and $[w]$ is an equivalence class of smooth extensions of $x$ to the disk ${\mb D}$, where two extensions $w$ and $w'$ are equivalent if the class of $w_1 \# (-w_2)$ is zero in $S_2(M)$.

Denote by $LG:= C^\infty( S^1, G)$ the smooth free loop group of $G$. Take any point $x_0 \in M$, we have the homomorphism $l(x_0): \pi_1(G) \to \pi_1(M, x_0)$ induced from the map sending a loop $t\mapsto \gamma(t) \in G$ to a loop $t \mapsto \gamma(t) x_0 \in M$. Its kernel ${\rm Ker} l(x_0)$ is independent of the choice of $x_0$. We define
\begin{align*}
L_0 G\subset L_M  G:= \big\{ \gamma: S^1 \to G\ |\ [\gamma] \in {\rm Ker} l(x_0) \subset \pi_1(G) \big\}\subset LG,
\end{align*}
%Note that the above definition is independent of the choice of $x_0$.
where $L_0 G$ is the subgroup of contractible loops.

The right action of $L_M G$ on $\wt{\mc L}$ can be explicitly written as
\begin{align*}
h^*(x, f)(t) =  \big( h(t)^{-1} x(t), {\rm Ad}_{h(t)}^{-1}  (f(t))  + h(t)^{-1} \partial_t h(t) \big),\text{ for any }h(t)\in L_M G.
\end{align*}

This action doesn't lift to $\wt{\mf L}$ naturally; only its restriction to $L_0 G$ does. To define this lift, one assigns a canonical choice of lift of capping for each element in $\wt{\mf L}$ as follows.

For a contractible loop $h: S^1 \to G$, extend $h$ arbitrarily to $h: {\mb D} \to G$. The homotopy class of extensions is unique because $\pi_2(G)= 0$ for any connected compact Lie group (\cite{Cartan_1936}). Then the class of $( h^{-1} x, h^* f, [ h^{-1} w] )$ in $\wt{\mf L}$ is independent of the extension.

It is easy to see that the covering map $\wt{\mf L}\to \wt{\mc L}$ is equivariant with respect to the inclusion $L_0 G \to L_M G$. Hence it induces a covering
\begin{align*}
\wt{\mf L} / L_0 G \to \wt{\mc L} / L_M G.
\end{align*}
An element of $\wt{\mf L}/ L_0 G$ represented by ${\mf x} \in \wt{\mf L}$ is denoted by $[{\mf x}]$.

%\subsubsection*{The deck transformation}

One may now define an action by ``connected sum" of $S_2^G(M)$ on $\wt{\mf L} / L_0 G$. For the precise definition, we refer the readers to \cite[Section 2.3]{Xu_VHF}.  One checks that

\begin{lemma}[\cite{Xu_VHF}, Lemma 2.4]
%The action (\ref{equation22}) is well-defined (i.e., independent of the %representatives and choices) and
$S^G_2(M)$ is the group of deck transformations of the covering $\wt{\mf L}/ L_0 G  \to \wt{\mc L}/ L_M G$.
\end{lemma}
Since $N_2^G(M) \subset S_2^G(M)$, by this lemma, we form ${\mf L}:= ( \wt{\mf L}/ L_0 G )/ N_2^G(M)$, which is again a covering of ${\mc L}:= \wt{\mc L}/ L_M G$ and $\Gamma$ is the group of deck transformations. We use $\lb {\mf x} \rb$ to denote an element in ${\mf L}$ represented by ${\mf x}\in \wt{\mf L}$.

\subsubsection*{The action functional}

Take $H = (H_t) \in {\mc H}_G(M)$. We define a 1-form $\wt{\mc B}_H$ on $\wt{\mc L}$ by
\begin{align*}
 T_{(x, f)} \wt{\mc L}\ni (\xi, h)\mapsto       \int_{S^1}\Big( \omega ( \dot{x}(t)+ {\mc X}_f-  Y_{H_t},\xi(t) ) + \langle \mu(x(t)), h(t)\rangle \Big) dt \in {\mb R}.
\end{align*}
Its pull-back to $\wt{\mf L}$ is the differential of the following action functional on $\wt{\mf L}$:
\begin{align*}
\wt{\mc A}_H(x, f, [w]):= -\int_B w^* \omega+ \int_{S^1}  \left( \mu(x(t))\cdot f(t) - H_t(x(t))\right) dt.
\end{align*}
The zero set of $\wt{\mc B}_H$ is
\begin{align*}
{\rm Zero} \wt{\mc B}_H:= \Big\{ (x, f) \in \wt{\mc L}\ |\ \mu(x(t))\equiv 0, \ \dot{x}(t) + {\mc X}_{f(t)} (x(t))- Y_{H_t}(x(t))=0\Big\}.
\end{align*}
The critical point set of $\wt{\mc A}_H$ its preimage under the covering $\wt{\mf L} \to \wt{\mc L}$.

By \cite[Lemma 2.5]{Xu_VHF}, $\wt{\mc A}_H$ is $L_0G$-invariant and $\wt{\mc B}_H$ is $L_M G$-invariant. So we have the induced action functional $\wt{\mc A}_H: \wt{\mf L}/ L_0 G \to {\mb R}$. Moreover, by \cite[Lemma 2.6]{Xu_VHF}, for any $[{\mf x}] \in \wt{\mf L} / L_0 G$ and $B \in S_2^G(M)$, we have
\begin{align*}
\wt{\mc A}_H ( B\# [\mf x] ) = \wt{\mc A}_H ( [\mf x] ) -  \langle [ \omega^G ], B \rangle.
\end{align*}
Therefore  $\wt{\mc A}_H$ descends to a well-defined function ${\mc A}_H: {\mf L} \to {\mb R}$. The vortex Floer theory is formally the Morse theory of the pair $\left( {\mf L}, {\mc A}_H \right)$.

\subsection{Gradient flow, symplectic vortex equation and the moduli spaces}

Let ${\mc P}$ be the set of pairs $(H, J)$ where $H \in {\mc H}_G^*(M)$ and $J = (J_t)$ is an $S^1$-family of almost complex structures $J$ on $M$, such that
\begin{align*}
{\mf f}(x) \geq {\mf c}_0 \Longrightarrow J_t(x) = {\mf J}(x)
\end{align*}
where $({\mf f}, {\mf J})$ is the convex structure in {\bf (H3)}. On the other hand, we fix a biinvariant metric on ${\mf g}$, which induces an identification ${\mf g}\simeq {\mf g}^*$. Let $\lambda>0$ be a constant. Formally, the equation for the negative gradient flow of ${\mc A}_H$ is the following equation for pairs $(u, \Psi): \Theta \to M \times {\mf g}$
\begin{align}\label{equation21}
\frac{\partial u }{\partial s}  + J_t \Big( \frac{\partial u}{\partial t} + {\mc X}_\Psi(u)- Y_{H_t}(u)  \Big) = 0,\ \frac{\partial \Psi}{\partial s} + \lambda^2 \mu(u) = 0.
\end{align}
This equation is invariant under the $LG$-action given on pairs $(u, \Psi)$ given by
\begin{align*}
g^* (u, \Psi)(s, t) = \left( g(t)^{-1} u(s, t), g(t)^{-1} \Psi(s, t) g(t)  + \partial_t \log g(t) \right).
\end{align*}
The energy of a solution ${\mf w}=(u, \Psi)$ is defined by
\begin{align}\label{equation22}
\mc{YMH}_\lambda ({\mf w}):= \Big\| \frac{\partial u}{\partial s} \Big\|_{L^2(\Theta)}^2 + \lambda^{-2} \Big\| \frac{\partial \Psi}{\partial s } \Big\|_{L^2(\Theta)}^2
\end{align}

\subsubsection*{Moduli spaces of flow lines}

It is more convenient to consider the symplectic vortex equation over $\Theta$ in general gauge, i.e., the following equation on triples $(u, \Phi, \Psi): \Theta \to M \times {\mf g} \times {\mf g}$
\begin{align}\label{equation23}
\displaystyle \frac{\partial u}{\partial s} + {\mc X}_{\Phi}(u) + J \Big( \frac{\partial u}{\partial t} + {\mc X}_{\Psi}(u)-Y_{H_t(u)} \Big) = 0,\ \frac{\partial \Psi}{\partial s} - \frac{\partial \Psi}{\partial t} + [\Phi, \Psi] + \lambda^2 \mu(u) = 0.
\end{align}
We showed in \cite[Section 3]{Xu_VHF} that, every finite energy solution \eqref{equation23} whose image in $M$ has compact closure is gauge equivalent to a solution ${\bm w} = (u, \Phi, \Psi)$ (which in fact can be taken to be in \textit{temporal gauge}, i.e., $\Phi \equiv 0$) and there exists a pair $\wt{x}_\pm = (x_\pm, \xi_\pm) \in {\rm Zero}\wt{\mc B}_H$ such that
\begin{align}\label{equation24}
\lim_{s \to \pm \infty} \Phi(s, \cdot) = 0,\ \lim_{s\to \pm\infty} (u(s, \cdot), \Psi(s, \cdot)) = \wt{x}_\pm.
\end{align}
We simply write $\displaystyle \lim_{z \to \pm\infty} {\bm w} = \wt{x}_\pm$. We also have the energy identity (cf. Theorem \ref{thm25})
\begin{align}\label{equation25}
\mc{YMH}_\lambda({\mf w}) = {\mc A}_H( {\mf x}_-) - {\mc A}_H({\mf x}_+).
\end{align}

For ${\mf x}_\pm \in {\rm Crit} \wt{\mc A}_H$ which project to $\wt{x}_\pm \in {\rm Zero} \wt{\mc B}_H$ (forgetting the cappings), we can consider solutions which ``connect'' them, which form a moduli space
\begin{align*}
{\mc M}_\lambda ( {\mf x}_-, {\mf x}_+; H, J) = \big\{ {\bm w} =(u, \Phi, \Psi)\ {\rm solves\ }\eqref{equation23}\ |\ \lim_{s \to \pm\infty} {\bm w} = \wt{x}_\pm,\ {\mf x}_- \# {\bm w} = {\mf x}_+\big\}/ {\mc G}_0.
\end{align*}
Here ${\mc G}_0$ is the space of smooth gauge transformations on $\Theta$ which are asymptotic to the identity at the circles at infinity. If ${\mf x}_\pm' = h_\pm {\mf x}_\pm$ for $h_\pm \in L_0 G$, then one find gauge transformation $h: \Theta \to G$ which is asymptotic to $h_\pm$ at the circles at infinity, and $h$ induces an identification between ${\mc M}_\lambda({\mf x}_-, {\mf x}_+; H, J)$ and ${\mc M}_\lambda({\mf x}_-', {\mf x}_+'; H, J)$. The moduli space in common with respect to this type of identifications is denoted by ${\mc M}_\lambda( [{\mf x}_-], [{\mf x}_+]; H, J)$, where $[{\mf x}_\pm] \in {\rm Crit} \wt{\mc A}_H/ L_0 G \subset \wt{\mf L}/L_0 G$ is the orbit of ${\mf x}_\pm$. Lastly, if $A \in \Gamma$, then there is an obvious identification between ${\mc M}_\lambda( [{\mf x}_-], [{\mf x}_+]; H, J)$ and ${\mc M}_\lambda( A \#[{\mf x}_-], A\# [{\mf x}_+]; H, J )$. Then the moduli space in common with respect to this ``capping shifting'' identification is denoted by
\begin{align*}
{\mc M}_\lambda ( \lb {\mf x}_-\rb, \lb {\mf x}_+ \rb; H, J ).
\end{align*}

\subsubsection*{Transversality}

In the case that $M$ is symplectic aspherical, we can perturb the $S^1$-family of almost complex structures $J$ to achieve transversality of the moduli space. We give a short account of the method to achieve transverality and refer the reader to \cite[Section 6]{Xu_VHF} for more details. For our convenience we introduce the following concept which was not present in \cite{Xu_VHF}.

\begin{defn}\label{defn22}
Fix $\ov{H} \in {\mc H}^*(\ov{M})$ and $\lambda>0$. A pair $(H, J)\in {\mc P}$ is called a regular pair (relative to $\ov{H}$ and $\lambda$), if $H$ is a $G$-invariant lift of $\ov{H}$ and the linearization of \eqref{equation23} along each bounded solution is surjective. The set of regular pairs (relative to $\ov{H}$ and $\lambda$) is denoted by ${\mc P}_{\ov{H}, \lambda}^{reg}$ and the union of ${\mc P}^{reg}_{\ov{H}, \lambda}$ for all $\ov{H}\in {\mc H}^*(\ov{M})$ is denoted by ${\mc P}^{reg}_\lambda$.
\end{defn}

In \cite[Section 6]{Xu_VHF}, the second named author used the following scheme to find $\lambda$-regular pairs. Among all smooth $S^1$-families of almost complex structures and all $G$-invariant lifts of $\ov{H}$, the second named author defined the notion of admissible families of almost complex structures (\cite[Definition 6.2]{Xu_VHF}) and admissible lifts (\cite[Definition 6.5]{Xu_VHF}) (both notions are defined relative to $\ov{H}$). Let $\wt{\mc J}_{\ov{H}}$ be the space of all smooth admissible families (relative to $\ov{H}$). Moreover, the following was proved.
\begin{thm}\label{thm23}\cite[Lemma 6.6, Theorem 6.8]{Xu_VHF} Let $\ov{H} \in {\mc H}^*(\ov{M})$ and $\lambda>0$.
\begin{enumerate}
\item There exists an admissible lift $H \in {\mc H}_G^*(M)$.

\item For each such lift $H$ and each $\lambda>0$, there exists a subset $\wt{\mc J}_{H, \lambda}^{reg} \subset \wt{\mc J}_{\ov{H}}$ of second category such that for any $J \in \wt{\mc J}_{H, \lambda}^{reg}$, $(H, J) \in {\mc P}_{\ov{H}, \lambda}^{reg}$, $(H, J) \in {\mc P}_{\ov{H}, \lambda}^{reg}$.

\item For any $(H, J) \in {\mc P}_{\ov{H}, \lambda}^{reg}$, any pair $\lb {\mf x}_\pm \rb \in {\rm Crit} {\mc A}_H$, ${\mc M}_\lambda ( \lb {\mf x}_- \rb, \lb {\mf x}_+ \rb; H, J)$ is a smooth manifold of dimension ${\sf cz} \lb {\mf x}_- \rb - {\sf cz} \lb {\mf x}_+ \rb$.
\end{enumerate}
\end{thm}

\begin{rem}
The statement of the above theorem differs slightly from the original ones in \cite{Xu_VHF}. This is because when we consider vortex Floer theory for product manifolds in Subsection \ref{subsection63}, we will take regular pairs $(H, J)$ of product type, which may not fall into the class of ``admissible'' ones in \cite{Xu_VHF}. Nevertheless, only the property of regular pairs are relevant to our construction.
\end{rem}

On the other hand, since the target $M$ is aspherical (Hypothesis {\bf (H1)}), one can compatify the moduli spaces only by adding broken flow lines. Meanwhile, one can orient the moduli spaces consistently (in the sense of \cite{Floer_Hofer_Orientation}) as the case of ordinary Morse or Hamiltonian Floer theory.

\subsection{The vortex Floer homology}

Let $R$ be either ${\mb Z}$ or ${\mb Q}$. The universal downward Novikov ring is
\begin{align}\label{equation26}
\Lambda:= \Lambda^\downarrow:= \left\{ \left. \sum_{i=1}^\infty a_i q^{w_i}\ \right| \ a_i \in R,\ w_i \in {\mb R}\ \lim_{i \to \infty} w_i = -\infty.\right\}.
\end{align}
The free $\Lambda$-module generated by ${\rm Crit}{\mc A}_H\subset {\mf L}$ is denoted by $\widehat{VCF}(M; H)$. We define an equivalence relation on $\widehat{VCF}(M; H)$ by
\begin{align*}
\lb {\mf x} \rb \sim q^c \lb {\mf x}' \rb \Longleftrightarrow \Gamma \lb {\mf x} \rb = \Gamma \lb {\mf x}' \rb,\ {\mc A}_H\lb {\mf x}' \rb + c = {\mc A}_H \lb  {\mf x} \rb.
\end{align*}
Denote by $VCF(M; H)$ the quotient $\Lambda$-module by the above equivalence relation.

In the notations for moduli spaces we temporarily omit the dependence on $J$, $H$ and $\lambda$. Let $\wh{\mc M}_\lambda ( \lb{\mf x}_- \rb, \lb {\mf x}_+ \rb)$ be the quotient of ${\mc M}_\lambda ( \lb{\mf x}_- \rb, \lb{\mf x}_+ \rb)$ by the ${\mb R}$-action of translation. Then for $(H, J) \in {\mc P}^{reg}_\lambda$, ${\rm dim} \wh{\mc M}_\lambda ( \lb{\mf x}_- \rb, \lb{\mf x}_+ \rb) = {\sf cz} \lb{\mf x}_-\rb - {\sf cz}\lb {\mf x}_+ \rb -1$ and one can define the function $m_{J, \lambda}^G:  {\rm Crit} {\mc A}_H \times {\rm Crit} {\mc A}_H  \to R$ by
\begin{align}\label{equation27}
m_{J,\lambda}^G ( \lb {\mf x}\rb, \lb{\mf y}\rb) = \left\{ \begin{array}{cc} 0,\ &\ {\sf cz}\lb{\mf x}\rb - {\sf cz} \lb{\mf y} \rb \neq 1,\\
                                        \# \widehat{\mc M}_\lambda ( \lb{\mf x}\rb, \lb {\mf y} \rb),\ &\ {\sf cz} \lb {\mf x}\rb - {\sf cz} \lb {\mf y} \rb = 1.
																				 \end{array} \right.
\end{align}
Here the number of $\wh{\mc M}_\lambda( \lb {\mf x} \rb, \lb{\mf y}\rb)$ is the algebraic count. The boundary operator $\delta_{J, \lambda}: VCF(M; H) \to VCF(M; H)$ is defined by the linear extension of
\begin{align*}
\delta_{J, \lambda} ( \lb{\mf x} \rb) = \sum_{ \lb{\mf y} \rb \in {\rm Crit} {\mc A}_H} m_{J, \lambda}^G ( \lb {\mf x} \rb, \lb {\mf y} \rb) \lb{\mf y}\rb.
\end{align*}
By a compactness argument the above sum is well-defined in $VCF(M; H)$ and by a gluing argument proved in \cite{Xu_VHF}, $(\delta_{J, \lambda})^2 = 0$. The {\bf vortex Floer homology} of the Hamiltonian $G$-manifold $(M, \mu)$ and the data $H, J, \lambda$ is defined as
\begin{align*}
VHF ( M; H, J, \lambda ) = H ( VCF(M; H), \delta_{J, \lambda} ).
\end{align*}

\subsection{Symplectic vortex equation on general Riemann surfaces}

To study further properties of vortex Floer cohomology, we must turn to the vortex equation on more general Riemann surfaces.  Indeed \eqref{equation23} fits into this general framework when specialized to $\Sigma=S^1 \times\R$.  In this section we briefly recall this general theory together with a calculation of the Yang-Mills-Higgs energy with cylindrical ends, which is a mild generalization of \cite[Section 3]{Xu_VHF} and \cite[Proposition 2.2]{Cieliebak_Gaio_Mundet_Salamon_2002} but critical to our applications.

%The vortex Floer theory is based on the analysis of the symplectic vortex %equation over Riemann surfaces, which was introduced by Mundet %\cite{Mundet_thesis, Mundet_2003} and Cieliebak-Gaio-Salamon %\cite{Cieliebak_Gaio_Salamon_2000}. \eqref{equation23} is just a special %case of this equation over the infinite cylinder, which we will see in a %moment. We include here a short review of the symplectic vortex equation for %those readers who are unfamiliar with this equation.	

\subsubsection*{Symplectic vortex equation}

Let $\Sigma$ be a Riemann surface and $P \to \Sigma$ be a smooth $G$-bundle. Let $\pi: Y:= P\times_G M\to \Sigma$ be the associated fibre bundle. Choose a family of $G$-invariant, $\omega$-compatible almost complex structures $J = (J_z)_{z\in \Sigma}$ on $M$. Then $J$ lifts to a complex structure on the vertical tangent bundle $T^\bot Y \to Y$, which is still denoted by $J$. On the other hand, the moment map $\mu$ lifts to a map $\mu: Y \to \big({\rm ad} P)^*:= P\times_{{\rm ad}^*} {\mf g}^*$, where ${\rm ad}^*$ is the co-adjoint action. We choose an ${\rm ad}$-invariant metric on ${\mf g}$ inducing an identification ${\mf g}^* \simeq {\mf g}$; hence $\mu$ is viewed as a map from $Y$ to ${\rm ad}P$.

Consider the space of smooth $G$-connections ${\mc A}(P)$ on $P$ and the space of smooth sections ${\mc S}(Y)$ of $Y\to \Sigma$. For each pair $(A, u) \in {\mc A}(P) \times {\mc S}(Y)$, we have the covariant derivative $d_A u \in \Gamma \big( \Sigma, {\rm Hom}(T\Sigma, u^* T^\bot Y) \big)$; $\ov\partial_A u$ is the $(0,1)$-part of $d_A u$ with respect to the complex structure $j_\Sigma$ on $\Sigma$ and $J$ on $T^\bot Y$. On the other hand, $F_A \in \Omega^2({\rm ad}P)$ is the curvature 2-form of $A$. Choose a volume form $\nu\in \Omega^2(\Sigma)$ and denote by $* F_A \in \Gamma({\rm ad}P)$ the contraction of $F_A$ against $\nu$. The symplectic vortex equation is the following system on pairs $(A, u)$
\begin{align}\label{equation28}
\ov\partial_A u = 0,\ * F_A + \mu\circ u = 0.
\end{align}

The group of gauge transformations ${\mc G}(P)$ consists of smooth maps $g: P \to G$ satisfying $g(ph) = h^{-1} g(p) h$ for $p\in P$ and $h \in G$. It induces a left-action on $P$ by $g\cdot p= p g(p)$ and a right action on pairs $(A, u)$ by reparametrization.

\subsubsection*{Hamiltonian perturbation}

In ordinary Floer theory, one uses Hamiltonian functions to perturb the $J$-holomorphic curve equation (see \cite[Section 8]{McDuff_Salamon_2004} as a standard reference). In our situation, we perturb similarly, by $G$-invariant Hamiltonians. A ${\rm Ham}^G(M)$-connection on $Y$ is a 1-form $\nabla \in \Omega^1(\Sigma, C^\infty_c(M)^G)$. If $z = s+ {\bm i} t$ is a local coordinate on $U\subset \Sigma$, then $\nabla$ can be written as
\begin{align*}
\nabla = K_\nabla^{(1)} (s, t) ds + K^{(2)}_\nabla (s, t) dt.
\end{align*}
Its curvature is
\begin{align*}
R_\nabla =\Big(  \frac{\partial K^{(2)}_\nabla}{\partial s} - \frac{\partial K^{(1)}_\nabla}{\partial t} + \{ K^{(1)}_\nabla, K^{(2)}_\nabla\}\Big) ds dt \in \Omega^2(U, C^\infty_c(M)^G).
\end{align*}
Let ${\mc Y}_{K_\nabla^{(1)}}$ and ${\mc Y}_{K_\nabla^{(2)}}$ be the Hamiltonian vector fields of $K_\nabla^{(1)}$ and $K_\nabla^{(2)}$, both depending on $(s, t)$. By $G$-invariance of $K^{(1)}_\nabla$ and $K^{(2)}_\nabla$, ${\mc Y}_{K_\nabla^{(1)}}$ and ${\mc Y}_{L_\nabla^{(2)}}$ lift to sections of $T^\bot Y|_U$. Then $\nabla$ deforms the operator $d_A$ to $d_{A, \nabla}$, which is
\begin{align*}
d_{A, \nabla} u  = d_A u + ds \otimes {\mc Y}_{K_\nabla}(u) + dt \otimes {\mc Y}_{L_\nabla}(u) \in \Gamma(\Sigma, {\rm Hom}(T\Sigma, u^* T^\bot Y) ).
\end{align*}
The $\nabla$-perturbed symplectic vortex equation reads
\begin{align}\label{equation29}
\ov\partial_{A, \nabla} u = 0,\  F_A + \mu(u)= 0,
\end{align}
where $\ov\partial_{A, \nabla}$ is the $(0, 1)$-part of $d_{A, \nabla}$. If we trivialize $P|_U$, then $A|_U = d+ \Phi ds + \Psi dt$ where $\Phi$ and $\Psi$ are ${\mf g}$-valued functions on $U$, and $\nu = \sigma ds dt$. Then \eqref{equation29} takes the form
\begin{align}\label{equation210}
\left\{ \begin{array}{ccc} \displaystyle  \frac{\partial u}{\partial s} + {\mc X}_\Phi + {\mc Y}_{K^{(1)}_\nabla} + J_{s, t} \Big( \frac{\partial u}{\partial t} + {\mc X}_\Psi(u) + {\mc Y}_{K^{(2)}_\nabla} \Big) & = & 0,\\[0.3cm]
\displaystyle \frac{\partial \Psi}{\partial s} - \frac{\partial \Psi}{\partial t} + [\Phi, \Psi] + \mu(u) \sigma(s, t) & = & 0.
\end{array}\right.
\end{align}

The energy of pairs $(A, u)$ with respect to the perturbed equation \eqref{equation22} can be generalized to the case of general perturbed symplectic vortex equation \eqref{equation28}, which is the {\bf Yang-Mills-Higgs functional}, given by
\begin{align*}
\mc{YMH}_\nabla (A, u) = \frac{1}{2} \Big( \big\| d_{A, \nabla} u \big\|_{L^2(\Sigma)}^2 + \big\| F_A \big\|_{L^2(\Sigma)}^2 + \big\| \mu(u) \big\|_{L^2(\Sigma)}^2 \Big).
\end{align*}
Here all the norms are taken with respect to the domain metric determined by $j_\Sigma$ and $\nu$, the target metric determined by $\omega$ and $J$, and the metric on the Lie algebra.

\subsubsection*{Punctured surface}

To get better control in actual applications, one needs to impose additional constraints in the case of non-compact Riemann surfaces.
%Now we consider \eqref{equation29} on surfaces with cylindrical ends.

Suppose $\Sigma$ is obtained by puncturing a compact Riemann surface $\ov\Sigma$ at finitely many points $z_1, \ldots, z_k$. Choose mutually disjoint cylindrical ends $U_j$ around $z_j$ for each puncture $z_j$ with identification $U_j \simeq \Theta_- = (-\infty, 0]\times S^1$. Denote $D_j = U_j \cup \{z_j\} \subset \ov\Sigma$. Specify a trivialization of $P|_{U_j}$, which induces an extension $\ov{P}\to \ov{\Sigma}$ of $P$. With respect to the specified trivialization, the restriction of any connection $A \in {\mc A}(P)$ to $U_j$ is written as $A|_{U_j} = d + \Phi_j ds + \Psi_j dt$ and a section $u\in \Gamma(Y)$ is identified with a map $u_j: U_j \to M$. If $u_j(s, t)$ converges to a smooth contractible loop $x_j(t)$ as $s \to +\infty$, and $w_j: {\mb D} \to M$ is a capping to $x_j$, then the ``connected sum''
\begin{align*}
(w_1, \ldots, w_k) \# u
\end{align*}
gives well-defined homotopy class  of continuous sections of $\ov{Y} = \ov{P}\times_G M$, and hence represents a homology class in $H_2^G(M)$.

We consider the case of \eqref{equation29} where all the auxiliary structures are of cylindrical type. Choose a volume form $\nu$ such that $\nu|_{U_j} = (\lambda^{(j)})^2 ds dt$. Choose regular pairs $(H^{(j)}, J^{(j)}) \in {\mc P}^{reg}_{\lambda^{(j)}}$ and consider a ${\rm Ham}^G(M)$-connection $\nabla$ over $\Sigma$, which is flat over $U_j$ and in a gauge such that
\begin{align*}
\nabla|_{U_j} = - H^{(j)}_t dt,\ H^{(j)} = (H^{(j)}_t) \in {\mc H}_G^*(M).
\end{align*}	
Choose a family of almost complex structures $J = (J_z)_{z\in \Sigma}$ which is {\bf exponentially close} to $J^{(j)}$ on $U_j$, i.e.,
\begin{align*}
\big| e^{|s|} ( J_{s, t} - J^{(j)}_t) \big|_{C^l(U_j)} < +\infty,\ j=0, 1, \ldots,\ j=1, \ldots, k.
\end{align*}
Then we have the following theorem on bounded solutions to \eqref{equation29}.
\begin{thm}\label{thm25}
Let ${\mf w} = (A, u)$ be a solution to \eqref{equation29} on a surface $\Sigma$ with cylindrical ends $U_j$ ($j=1, \ldots, k$), such that $\mc{YMH}_\nabla({\mf w})< +\infty$ and $\sup_{\Sigma} |\mu(u)| < +\infty$. Then the following holds.
\begin{enumerate}

\item There is a gauge transformation $g \in {\mc G}(P)$ and $\wt{x}_j = (x_j, f_j) \in {\rm Zero} \wt{\mc B}_{H^{(j)}}$, such that over $U_j$, with respect to the specified trivialization of $P|_{U_j}$, $g^* A = d + \Phi_j(s, t) + \Psi_j(s, t) dt$ and
\begin{align*}
\lim_{s \to -\infty} \Phi_j(s, t) = 0,\ \lim_{s \to -\infty} (g^{-1} u, \Psi_j) = (x_j, f_j)\in {\rm Zero}\wt{\mc B}_{H^{(j)}}.
\end{align*}

\item  We have a two form $\omega_{A, \nabla}\in \Omega^2(Y)$, called the {\bf minimal coupling form}, so that for any choice of capping ${\mf x}_j = (x_j, f_j, [w_j]) \in {\rm Crit} \wt{\mc A}_{H^{(j)}}$ to each $\wt{x}_j$,
\begin{align}\label{equation211}
\begin{split}
\mc{YMH}_\nabla ({\mf w}) &= \int_{\Sigma} u^* \omega_{A, \nabla} + R_\nabla \\
&=\langle [\omega^G], [(w_1, \cdots, w_k) \# u] \rangle +  \sum_{j=1}^k \wt{\mc A}_{H^{(j)}} ( {\mf x}_j ) +  \int_{\Sigma} R_{\nabla}.
\end{split}
\end{align}

%\noindent Here
\end{enumerate}
\end{thm}
\begin{proof}
The first part follows from the proof of the asymptotic behavior of bounded solutions to \eqref{equation23} in \cite[Section 3]{Xu_VHF}. The second part is essentially the same as \cite[Proposition 2.2]{Cieliebak_Gaio_Mundet_Salamon_2002}. Since it will be used in this paper, we include a detailed proof.

To prove \eqref{equation211}, we write down the energy density explicitly. Let $z = s+ {\bm i} t$ be a local coordinate on $U\subset \Sigma$ and $\nu|_U = \sigma ds dt$. Denote
\begin{align*}
v_s = \frac{\partial u}{\partial s} + {\mc X}_\Phi + {\mc Y}_{K_\nabla^{(1)}},\ v_t= \frac{\partial u}{\partial t} + {\mc X}_\Psi + {\mc Y}_{K_\nabla^{(2)}}.
\end{align*}
Then since $(A, u)$ solves \eqref{equation210}, we have
\begin{align*}
\begin{split}
&\ \frac{1}{2} \Big( \big| d_{A, \nabla} u\big|^2 + \big| F_A \big|^2 + \big| \mu(u) \big|^2 \Big)\sigma\\
= &\  \frac{1}{2} \big| v_s + J v_t \big|^2 -  \langle v_s, J v_t \rangle  + \frac{1}{2} \big| F_A + \mu(u) \nu \big|^2 \sigma - \mu(u) \cdot F_A\\
= &\ \omega\Big( \frac{\partial u}{\partial s} + {\mc X}_\Phi + {\mc Y}_{K_\nabla^{(1)}}, \frac{\partial u}{\partial t} + {\mc X}_\Psi + {\mc Y}_{K_\nabla^{(2)}} \Big) -\mu(u) \cdot \Big( \frac{\partial \Psi}{\partial s} - \frac{\partial \Phi}{\partial t} + [\Phi, \Psi] \Big)\\
= &\ \omega\Big( \frac{\partial u}{\partial s}, \frac{\partial u}{\partial t} \Big) + d \Big( \mu\cdot \Phi + K_\nabla^{(1)}\Big) \cdot \Big( \frac{\partial u}{\partial t} \Big) - d \Big( \mu \cdot \Psi + K_\nabla^{(2)} \Big) \cdot \Big( \frac{\partial u}{\partial s} \Big)\\
&\ + \omega\Big( {\mc X}_{\Phi} + {\mc Y}_{K_\nabla^{(1)}}, {\mc X}_{\Psi} + {\mc Y}_{K_\nabla^{(2)}} \Big) -\mu(u) \cdot \Big( \frac{\partial \Psi}{\partial s} - \frac{\partial \Phi}{\partial t} + [\Phi, \Psi] \Big)\\
= &\ \omega\Big( \frac{\partial u}{\partial s}, \frac{\partial u}{\partial t} \Big) + \frac{\partial}{\partial t} \Big( \mu(u) \cdot \Phi + K_\nabla^{(1)}(u) \Big) - \frac{\partial }{\partial s} \Big( \mu(u) \cdot \Psi + K_\nabla^{(2)}(u) \Big) + R_\nabla(\frac{\partial}{\partial s},\frac{\partial}{\partial t})\\
= &\ u^* \omega_{A, \nabla}\Big( \frac{\partial}{\partial s}, \frac{\partial}{\partial t} \Big) + R_\nabla \Big( \frac{\partial}{\partial s}, \frac{\partial}{\partial t} \Big) 	.
\end{split}
\end{align*}
Here $\omega_{A, \nabla}\in \Omega^2(Y)$ can be written in local coordinates as $\omega_{A, \nabla} = \omega - d\big( \mu \cdot (\Phi ds + \Psi dt) + K_\nabla^{(1)} ds + K_\nabla^{(2)} dt \big)$, which is independent of choice of local coordinates and trivializations thus well-defined. Therefore
\begin{align}\label{equation212}
\mc{YMH}_\nabla(A, u) = \int_{\Sigma} u^* \omega_{A, \nabla} + R_\nabla.
\end{align}
To evaluate the first term, for each $s>1$, choose cut-off functions $\rho_s: \Sigma \to [0, 1]$ which is equal to 1 away from the cylindrical end $U_j(s-1):= (-\infty, -s+1]\times S^1 \subset U_j$, and equal to 0 over each $U_j(s)$. Denote
\begin{align*}
\Sigma_s = \Sigma \setminus \cup_{j=1}^k U_j(s),\ u_s= u|_{\Sigma_s}.
\end{align*}
Then for $s$ large, we can connect $w_j$ with $u_s$ to obtain a family of smooth sections $\ov{u}_s \in \Gamma(\ov{Y})$, whose restriction to the disk $D_{j, s} = U_j(s) \cup \{z_j\}$ is denoted by $w_{j, s}: D_{j, s} \to M$. Moreover, $\ov{u}_s$ is represents the class of $(w_1, \ldots, w_k) \# u$ and
\begin{align*}
\lim_{s \to +\infty} \int_{D_{j, s}} (w_{j, s})^* \omega = \int_{\mb D} (w_j)^* \omega.
\end{align*}
On the other hand, $\nabla_s = \rho_s \nabla$ is another ${\rm Ham}^G(M)$-connection which extends to $\ov\Sigma$, and using the specified trivialization of $P|_{U_j}$,
\begin{align*}
A_s = A - (1-\rho_s) \sum_{j=1}^k (\Phi_j ds + \Psi_j dt)
\end{align*}
is a smooth connection on $P$ which extends to $\ov{P} \to \ov\Sigma$. Then
\begin{align}\label{equation213}
\begin{split}
&\ \int_{\Sigma_s} (u_s)^* \omega_{A_s, \nabla_s} \\
= &\ \int_{\Sigma_s} (u_s)^* \omega_{A_s, \nabla_s} + \sum_{j=1}^k \int_{U_j} u^* (\omega_{A, \nabla} - \omega_{A_s, \nabla_s})\\
                                     = &\ \int_{\ov\Sigma}(\ov{u}_s)^* \omega_{A_s, \nabla_s} - \sum_{j=1}^k \int_{D_{j, s}} (w_{j, s})^* \omega \\
																		 &\ + \sum_{j=1}^k \int_{[-s+1, -s]\times S^1} - u^* d \Big( (1-\rho_s) \mu\cdot ( \Phi_j ds + \Psi_j dt) - H_t^{(j)} dt \Big)\\
																		 = &\ \int_{\ov\Sigma} (\ov{u}_s)^* \omega_{A_s, \nabla_s} - \sum_{j=1}^k \int_{D_{j, s}} (w_{j, s})^* \omega + \sum_{j=1}^k \int_{S^1}\Big( \mu(u(s, t)) \Phi_j -  H_t^{(j)}(u(s, t)) \Big)dt																		 
\end{split}
\end{align}
The first term in the last line is equal to the pairing between $[\omega^G]$ and the homology class represented by $(w_1, \ldots, w_k) \# u$. Let $s \to +\infty$. The left-hand-side of \eqref{equation213} then converges to the integral of $u^* \omega_{A, \nabla}$ on $\Sigma$ and the right-hand-side converges to the sum of $\wt{\mc A}_{H^{(j)}}({\mf x}_j)$. Then by \eqref{equation212}, \eqref{equation211} is proved.\end{proof}

\begin{rem}

One should compare the general vortex equation in this section with those appeared in the definition of vortex Floer homology as follows.

If $\Sigma = \Theta$ and $(H_t)$ is an $S^1$-family of $G$-invariant Hamiltonians, then one has the flat connection $\nabla = - H_t dt$. So \eqref{equation21} and \eqref{equation23} are the special form of \eqref{equation210} for the trivial $G$-bundle over $\Theta$, with volume form $\lambda^2 ds dt$, and \eqref{equation21} is the form in
%a special gauge (i.e., $\Phi \equiv 0$, which is usually called the
temporal gauge, and the parameter $\lambda$ can also be viewed as a rescaling factor of the metric on ${\mf g}$. \eqref{equation25} is hence a special case of \eqref{equation211}.

\end{rem}

\subsection{Continuation maps}

We now continue our discussion on vortex Floer homology using results from last section.  Let $\lambda^\pm > 0$ and $(H^\pm, J^\pm) \in {\mc P}^{reg}_{\lambda^\pm}$. Choose a smooth family $\wt{H} = ( H_{s, t} )_{(s, t) \in \Theta}$ of Hamiltonians with $H_{s, t} \in {\mc H}_G(M)$, such that
\begin{align*}
\pm s \gg 0 \Longrightarrow H_{s, t} = H^\pm_t.
\end{align*}
$\wt{H}$ defines a ${\rm Ham}^G(M)$-connection on $\Theta$, i.e., $\nabla = - H_{s, t}dt$. Choose a smooth function $\wt\lambda: {\mb R} \to (0, +\infty)$ such that
\begin{align*}
\pm s \gg 0 \Longrightarrow \wt\lambda(s) = \lambda^\pm.
\end{align*}
This gives a volume form $\wt\lambda (s)^2 ds dt$ on $\Theta$. On the other hand, consider the space $\wt{\mc J}(J^-, J^+)$ consisting of families of almost complex structures $\wt{J} = (J_{s, t})_{(s, t) \in \Theta}$ which are exponentially close to $J^\pm$ on $\Theta_\pm$. Then for the auxiliary data $\wt{H}$, $\wt{J}$, $\wt\lambda$ on the cylinder $\Theta$, consider the $\nabla$-perturbed equation \eqref{equation29}, which, in the cylindrical coordinates, reads
\begin{align}\label{equation214}
\left\{ \begin{array}{ccc}
\displaystyle \frac{\partial u }{ \partial s} + {\mc X}_{\Phi}(u) + J_{s, t} \Big(\frac{\partial u }{\partial t}  + {\mc X}_{\Psi}(u) - Y_{H_{s, t}}(u) \Big) & = & 0;\\[0.3 cm]
\displaystyle \frac{\partial \Psi}{\partial s} - \frac{\partial \Phi}{\partial t} + [\Phi, \Psi]  + \wt\lambda(s)^2 \mu(u) & = & 0.
\end{array}\right.
\end{align}
Then by Theorem \ref{thm25}, for any bounded solution ${\mf w}$ to \eqref{equation214}, there exists a pair $\wt{x}^\pm \in {\rm Zero} \wt{\mc B}_{H^\pm}$ such that ${\mf w}$ is gauge equivalent to some ${\mf w}'$ with
\begin{align*}
\lim_{z\to \pm \infty}{\mf w}' = \wt{x}^\pm.
\end{align*}
Then similarly to the case of moduli spaces of vortex Floer connecting orbits, for any pair $\lb {\mf x}^\pm \rb \in {\rm Crit} {\mc A}_{H^\pm}$, we consider the moduli space ${\mc N} ( \lb  {\mf x}^- \rb, \lb {\mf x}^+ \rb; \wt{H}, \wt{J}, \wt\lambda )$ of gauge equivalence classes of solutions to \eqref{equation214} which connect $\lb {\mf x}^- \rb$ and $\lb {\mf x}^+ \rb$. As a special case of the second part of Theorem \ref{thm25}, we have
\begin{prop}\cite[Proposition 7.5]{Xu_VHF}\label{prop27}
For any $[ {\mf w} ] \in {\mc N} ( \lb {\mf x}^- \rb, \lb {\mf x}^+ \rb; \wt{H}, \wt{J}, \wt\lambda )$, we have
\begin{align}\label{equation215}
\mc{YMH}_{\wt{H}} \big( [{\mf w}] \big) = {\mc A}_{H^-} \lb {\mf x}^- \rb  - {\mc A}_{H^+} \lb {\mf x}^+ \rb - \int_\Theta \frac{\partial H_{s, t}}{\partial s} (u) ds dt.
\end{align}
\end{prop}

By \cite[Proposition 7.6]{Xu_VHF}, there is a subset $\wt{\mc J}^{reg}_{\wt{H}, \wt\lambda}(J^-, J^+)\subset \wt{\mc J}(J^-, J^+)$ of second category, such that for each $\wt{J} \in  \wt{\mc J}^{reg}_{\wt{H}, \wt\lambda}(J^-, J^+)$, the moduli space ${\mc N} (\lb {\mf x}^- \rb, \lb {\mf x}^+ \rb; \wt{H}, \wt{J}, \wt\lambda)$ is an oriented smooth manifold of dimension ${\sf cz} \lb {\mf x}^- \rb - {\sf cz} \lb {\mf x}^+ \rb$. So we can use the oriented counting to define the chain level continuation map
\begin{align*}
\mbms{cont}: VCF( M; H^- ) \to VCF ( M; H^+),
\end{align*}
which is a chain homotopy equivalence. For a different set of homotopy $(\wt{H}', \wt{J}', \wt\lambda')$ having the same asymptotic conditions as $(\wt{H}, \wt{J}, \wt\lambda)$, the induced continuation map $\mbms{cont}'$ is homotopic to $\mbms{cont}$. Therefore we have

\begin{thm}
The vortex Floer homology $VHF_\lambda (M; H, J)$ of $M$ is canonically independent of the choice of $\lambda>0$ and the regular pair $(H, J) \in {\mc P}^{reg}_\lambda$.
\end{thm}
The common $\Lambda$-module which is canonically isomorphic to each $VHF_\lambda (M; H, J)$ (whenever it is defined) is denoted by $VHF(M)$ and is called the {\bf vortex Floer homology} of $M$.

\subsection{Ring structure}\label{subsection27}

As in the ordinary Hamiltonian Floer theory, the quantum ring structure on $VHF(M)$ can be defined by the pair-of-pants product. Its construction is given by specializing the the analysis of \eqref{equation29} to the pair-of-pants. The details are given in \cite{Xu_Floer_3} but the basis analysis are well established: for the analysis of breaking and gluing the vortex Floer trajectories from cylindrical ends, see \cite{Xu_VHF} and \cite{Mundet_thesis}; for the independence of auxiliary data, see \cite{Mundet_thesis, Mundet_2003, Cieliebak_Gaio_Mundet_Salamon_2002}. In particular, the transversality argument is again similar to that in \cite{Xu_VHF} due to the absence of sphere bubbles.

%Here we review necessary details of the ring structure, which will be used in %establishing certain properties of our spectral numbers.

\subsubsection*{Pair-of-pants product}

The pair-of-pants multiplications is a $\Lambda$-bilinear map
\begin{align}\label{equation216}
VHF (M) \otimes VHF (M) \to VHF (M).
\end{align}
Let $\Sigma$ be the pair-of-pants, which is biholomorphic to the three punctured sphere $S^2 \setminus \{z_0, z_1, z_\infty\}$. We take a cylindrical area form $\nu\in \Omega^2(\Sigma)$, so that there are three (disjoint) cylindrical ends $U_0, U_1, U_\infty \subset \Sigma$. We regard the first two as ``incoming'', which are isometric to $\Theta_-= (-\infty, 0]\times S^1$, and $U_\infty$ as ``outgoing'', which is isometric to $\Theta_+ = [0, +\infty) \times S^1$. Choose a volume form $\nu$ whose restriction to $U_\sigma$ is equal to $(\lambda^\sigma)^2 ds dt$.

For $\sigma = 0, 1, \infty$, choose regular pairs $(H^\sigma, J^\sigma) \in {\mc P}^{reg}_{\lambda^\sigma}$. Then we have the chain complexes
\begin{align*}\Big(VCF(M; H^\sigma), \delta_{J^\sigma, \lambda^\sigma} \Big)
\end{align*}
Choose a ${\rm Ham}^G(M)$-connection $\nabla\in \Omega^1(\Sigma, C^\infty(M)^G)$ such that
\begin{align}\label{equation217}
\nabla|_{U_\sigma} = - H^\sigma_t dt.
\end{align}
Choose a family of almost complex structures $J = (J_z)_{z\in \Sigma}$ which are exponentially close to $J^\sigma$ on $U_\sigma$ for $\sigma = 0, 1, \infty$.

Now consider the $\nabla$-perturbed symplectic vortex equation \eqref{equation29} on the trivial $G$-bundle over $\Sigma$, with the auxiliary data $\nu$, $\nabla$ and $J$. Each bounded solution $(A, u)$ has  good asymptotic behavior by Theorem \ref{thm25}. Then, for any triple ${\mf x}_\sigma\in {\rm Crit} \wt{\mc A}_{H^\sigma}$ which projects to $\wt{x}_\sigma \in {\rm Zero} \wt{\mc B}_{H^\sigma}$, we consider
\begin{align*}
\begin{split}
\wt{\mc M}\big( x_0, x_1; x_\infty \big) = &\ \Big\{ (A, u)\ {\rm solves\ } \eqref{equation29} \ |\ \lim_{z \to z_\sigma} u = \wt{x}_\sigma \Big\},\\
\wt{\mc M}\big( {\mf x}_0, {\mf x}_1; {\mf x}_\infty \big) = &\ \Big\{ (A, u) \in \wt{\mc M}\big( {\mf x}_0, {\mf x}_1; {\mf x}_\infty\big)\ |\ ({\mf x}_0, {\mf x}_1, \ov{\mf x}_\infty) \# [u] = 0 \Big\}.
\end{split}
\end{align*}
Here $\ov{\mf x}_\infty$ is the reverse of ${\mf x}_\infty$. Let ${\mc G}_\Sigma$ be the space of smooth gauge transformations $g: \Sigma \to G$ which are asymptotic to the identity at each $z_\sigma$. ${\mc G}_\Sigma$ acts on $\wt{\mc M}({\mf x}_0, {\mf x}_1; {\mf x}_\infty)$ and denote ${\mc M}({\mf x}_0, {\mf x}_1; {\mf x}_\infty) = \wt{\mc M}({\mf x}_0, {\mf x}_1; {\mf x}_\infty)/ G_\Sigma$.

If $[ {\mf x}_\sigma] = [{\mf x}_{\sigma}']\in \wt{\mf L}/ L_0 G$, then there is a canonical identification between ${\mc M}({\mf x}_0, {\mf x}_1; {\mf x}_\infty)$ and ${\mc M}({\mf x}'_0, {\mf x}'_1; {\mf x}'_\infty)$. We denote the common object by ${\mc M}([{\mf x}_0], [{\mf x}_1]; [{\mf x}_\infty])$. Moreover, if $B_0, B_1 \in S_2^G(M)$, then $B_0 \#[{\mf x}_0]$ and $B_1 \# [{\mf x}_1]$ are represented by $h_0(x_0)$ and $h_1(x_1)$ respectively, where $h_0, h_1\in L_M G$. Then we can extend $h_0, h_1$ to a global map $h: \Sigma \to G$ whose restrictions on $U_0$, $U_1$, $U_\infty$ are $h_0, h_1, h_0 h_1$ respectively. The gauge transformation by $h$ provides an identification
\begin{align*}
{\mc M}\big( [{\mf x}_0], [{\mf x}_1]; [{\mf x}_\infty] \big) \simeq {\mc M} \big( B_0 \# [{\mf x}_0], B_1 \# [{\mf x}_1]; (B_0 + B_1) \# [{\mf x}_\infty] \big).
\end{align*}
Then we define
\begin{align*}
{\mc M}\big( \lb {\mf x}_0 \rb, \lb{\mf x}_1\rb; \lb{\mf x}_\infty \rb \big) := \Big( \bigcup_{[\eta_\sigma] \in {\rm Crit} \wt{\mc A}_{H^\sigma}/ L_0 G \atop \lb \eta_\sigma \rb = \lb {\mf x}_\sigma \rb}   {\mc M}\big( [\eta_0], [\eta_1]; [\eta_\infty] \big)  \Big) / \left( N_2^G(M) \times N_2^G(M) \right).
\end{align*}

Under the assumptions (H1)-(H4), it is easy to achieve transversality by choosing a generic family of almost complex structures which are exponentially close to $J^\sigma$, similar to the case of defining the continuation map. Hence we may always assume that the above moduli space is a smooth manifold. It has a consistent orientation, for essentially the same reason as the case of the usual Hamiltonian Floer theory (cf. \cite{PSS}). For any choice of $\lb {\mf x}_\sigma\rb$, ($\sigma = 0, 1, \infty$), the moduli space is also compact up to the breaking of vortex Floer connecting orbits at cylindrical ends, which is a mild generalization from \cite{Xu_VHF}. Therefore, it is a finite set when ${\sf cz}\lb {\mf x}_0\rb + {\sf cz} \lb {\mf x}_1\rb = {\sf cz}\lb {\mf x}_\infty \rb + \ov{m}$. So its counting gives a chain map
\begin{align*}
VCF(H^0, J^0) \otimes VCF(H^1, J^1) \to VCF (H^\infty, J^\infty),
\end{align*}
which induces the pair-of-pants product \eqref{equation216}.

It is routine to check that the pair-of-pants product \eqref{equation216} is independent of various choices by a homotopy argument, due to the absence of bubbling.  This is essentially a special case of the analysis for general Riemann surfaces in \cite{Mundet_thesis}.  In particular, it doesn't depend on the choice of ${\rm Ham}^G(M)$-connections on $\Sigma$ that satisfy \eqref{equation217}. This fact will be crucial in proving the triangle inequality for the spectral numbers.

\subsubsection*{Identity element}

The multiplicative identity of $VHF(M)$ can be described as follows. Let the domain be ${\mb C}$ and let $U = \{ z\in {\mb C}\ |\ |z|\geq 1\}$ be the cylindrical end, identified with $[0, +\infty)\times S^1$ via the exponential map. Choose a cut-off function $\rho(z)$ supported in $U$ and $\rho(z) = 1$ for $|z|\geq 2$. Choose a volume form $\nu\in \Omega^2({\mb C})$ which is equal to $\lambda^2 ds dt$ on $U$ for some $\lambda>0$.

Now choose a regular pair $(H, J) \in {\mc P}^{reg}_{\lambda}$. Then $\nabla = -\rho(s, t) H_t dt$ is a ${\rm Ham}^G(M)$-connection on ${\mb C}$. Choose a family $\wt{J} = (J_z)_{z\in {\mb C}}$ which is exponentially closed to $J$ on $U$. Then these auxiliary data gives the following
%$\nabla$-perturbed symplectic vortex equation
special instance of \eqref{equation210} on the trivial $G$-bundle over ${\mb C}$ on pairs $(A, u) \in \Omega^1(\mb C, \mf g)\times C^\infty({\mb C}, M)$
\begin{align}\label{equation218}
\ov\partial_{A, \nabla} u = 0,\ F_A + \mu(u) \nu = 0.
\end{align}
The group of gauge transformations $g: {\mb C} \to G$ acts on pairs $(A, u)$ which makes the above equation invariant.

By Theorem \ref{thm25}, any bounded solution to \eqref{equation218} is gauge equivalent to a solution ${\mf w} = (A, u)$ which is asymptotic to a loop $\wt{x}\in {\rm Zero}\wt{\mc B}_H$. The homotopy class of $u$ automatically determines a capping ${\mf x}$ to $\wt{x}$, whose class $\lb {\mf x} \rb \in {\rm Crit} {\mc A}_H \subset {\mf L}$ only depends on the gauge equivalence class of ${\mf w}$. Therefore we can consider equivalence classes of solutions to \eqref{equation218} which represent $\lb {\mf x} \rb$, whose moduli space is denoted by
\begin{align*}
{\mc C}(\lb {\mf x} \rb):= {\mc C}(\lb {\mf x} \rb; \nabla, \wt{J}, \nu),
\end{align*}
where ${\mc C}$ stands for ``cap''. By an easy index calculation, we see that this moduli space has virtual dimension
\begin{align*}
{\rm dim}{\mc C}(\lb {\mf x} \rb) = \ov{m} - {\sf cz}\lb {\mf x}\rb.
\end{align*}
Moreover, by perturbing the family of the almost complex structure $J_z$, we can make every such moduli space transverse and has its dimension equal to the virtual dimension.

On the other hand, each ${\mc C}( \lb {\mf x} \rb)$ is compact up to the breakings of vortex Floer trajectories at the infinity side of ${\mb C}$ due to the absence of bubbling again. Such broken objects exist in strictly lower dimensions. Therefore, when $\ov{m} = {\sf cz} \lb {\mf x} \rb$, ${\mc C}( \lb {\mf x} \rb)$ is a finite set. On the other hand, we can assign an orientation on each ${\mc C}( \lb {\mf x} \rb)$ consistent with the orientation of moduli spaces of vortex Floer trajectories. Then we have the algebraic counting $\# {\mc C}( \lb {\mf x} \rb) \in {\mb Z}$. This counting is independent of various choices, such as the $S^1$-family of almost complex structures, the cut-off function $\rho$, etc..

Now we define a formal sum
\begin{align*}
{\bm 1}_H:= \sum_{\ov{m} = {\sf cz} \lb {\mf x} \rb} \# {\mc C}( \lb {\mf x} \rb) \lb {\mf x} \rb.
\end{align*}
By a compactness argument, ${\bm 1}_H\in VCF(M; H)$. Moreover, $\delta_{J, \lambda} {\bm 1}_H = 0$, because if
\begin{align*}
\delta_{J, \lambda}{\bm 1}_H= \sum_{{\sf cz} \lb {\mf y} \rb = \ov{m} -1} a_{\lb {\mf y} \rb} \lb {\mf y}\rb
\end{align*}
then each $a_{\lb {\mf y} \rb}$ counts the boundary of the 1-dimensional moduli space ${\mc C}( \lb {\mf y} \rb)$, which should be zero. Therefore ${\bm 1}_H$ is a vortex Floer cycle. A homotopy argument including certain gluing construction shows that the class of ${\bm 1}_H$ is independent of the choice of $\nu$, $\nabla$ and $J$which have the necessary asymptotic conditions. Therefore, ${\bm 1}_H$ represents a well-defined class
\begin{align}\label{equation219}
{\mbms 1}_{\mbms H} \in VHF(M).
\end{align}

Lastly, using the standard method, one may show that ${\mbms 1}_{\mbms H}$ is a multiplicative identify in $VHF(M)$.  The required gluing analysis was carried out in \cite{Xu_VHF} and \cite{Cieliebak_Gaio_Mundet_Salamon_2002}. The existence of multiplicative identity (or more generally idempotent elements) in the quantum ring is crucial in constructing pre-quasimorphisms and partial quasi-states.

\section{The spectral invariants}\label{section3}

\subsection{Usher's abstract formulation}\label{subsection31}

Let $R$ be a commutative ring. We recall (an ungraded version of) the definition of Usher \cite{Ush08, Usher_2010} of an abstract filtered Floer-Novikov complex.
\begin{defn}\label{defn31}
A {\bf graded filtered Floer-Novikov complex} ${\mf c}$ over $R$ is a tuple
\begin{align*}
{\mf c}:= \big( \underline{\Gamma}, \underline{\mc L}; {\mc A}, \omega; m \big)
\end{align*}
where
\begin{enumerate}
\item {\bf (Loops)} $\underline\Gamma$ is a finitely generated abelian group and $\underline{\mc L}$ is a $\underline\Gamma$-torsor such that $\underline{\mc L}/ \underline\Gamma$ is finite.

\item {\bf (Action functional)} The `action functional'' ${\mc A}: \underline{\mc L} \to {\mb R}$ and the ``period homomorphism'' $\omega: \underline\Gamma \to {\mb R}$ satisfy, for $B \in \underline\Gamma$ and $x \in \underline{\mc L}$
\begin{align*}
{\mc A}(B\# x) = {\mc A}(x)-\omega(B).
\end{align*}

\item {\bf (Flow line counting)} $m: \underline{\mc L} \times \underline{\mc L} \to R$ is a function satisfying, for $x, y \in \underline{\mc L}$ and $B\in \underline\Gamma$
\begin{align*}
{\mc A}(x) \leq {\mc A}(y) \Longrightarrow  m(x, y) = 0,\ m( B\# x, B\# y ) =  m( x, y).
\end{align*}

\item {\bf (Differential)} Let $\Lambda$ be the downward Novikov ring defined by \eqref{equation26} and let $\widehat{\mf c}$ be the free $\Lambda$-module generated by $\underline{\mc L}$. Let ${\mf c}$ be the quotient of $\widehat{\mf c}$ modulo the relation
\begin{align*}
x \simeq q^w x' \Longleftrightarrow \underline\Gamma x = \underline\Gamma x',\ {\mc A} (x') = {\mc A}(x) - w.
\end{align*}
We require that the formal series $\delta(x) = \sum_{y \in \underline{\mc L}} m(x, y) y$ lies in ${\mf c}$ and hence extends to a $\Lambda$-linear map. We require that $\delta^2 = 0$.
\end{enumerate}
\end{defn}

%\subsection{Filtered Morse-Smale-Witten complex of the symplectic quotient}

%Take a generic Morse function $f: \ov{M} \to {\mb R}$. Consider the induced Riemannian metric $\ov{g}$ on $\ov{M}$ and assume that $(\ov{f}, \ov{g})$ is Morse-Smale. Therefore we have the Morse-Smale-Witten complex associated to the pair $(\ov{f}, \ov{g})$, denoted by
%\begin{align}
%\left( CM_* \left( \ov{f}; {\mb Z} \right), \ov\delta_{\ov{g}} \right).
%\end{align}
%We simply tensor it with the Novikov ring $\Lambda_{\mb Z}$ and extend $\ov\delta_{\ov{g}}$ linearly, obtaining the chain complex
%\begin{align}
%\left( CM_* \left( \ov{f}; \Lambda_{\mb Z} \right), \ov\delta_{\ov{g}} \right).
%\end{align}
%Its homology is isomorphic to the singular homology $H_*( \ov{M}; \Lambda_{\mb Z})$.

%Now we define a filtration on $CM_* \left( \ov{f}; \Lambda_{\mb Z} \right)$. For any $\kappa\in {\mb R}$, define
%\begin{align}
%CM_*^\kappa \left( \ov{f}; \Lambda_{\mb Z} \right) = \left\{ {\mf S} = \sum a_{{\mf L}} {\mf L} \in CM_* (\ov{f}; \Lambda_{\mb Z}) \ |\ a_{\mf L} \neq 0 \Longrightarrow {\mf L} + {\mf v}(a_{\mf L}) \leq \kappa \right\}.
%\end{align}
%Then this is a subcomplex simply because the negative gradient flow decreases the function value.

\subsection{The filtration on vortex Floer complex}

By Definition \ref{defn31}, the construction of the complex $(VCF(M; H), \delta_{J, \lambda})$ and the energy identity \eqref{equation25}, the following fact is obvious.

\blem
Using the notations of Section \ref{section2}, for $\lambda > 0$ and $(H, J) \in {\mc P}_\lambda^{reg}$, the collection $(\Gamma,{\mc L};{\mc A}_H,\w^G,m^G_{J, \lambda})$ (where $m^G_{J, \lambda}$ is defined by \eqref{equation27}) is an example of Usher's abstract filtered Floer-Novikov complex over ${\mb Z}$.
\elem

For any filtered Novikov-Floer complex, one can define a natural action filtration. Here we recall the construction for the concrete complex $VCF(M; H)$.

On the Novikov ring $\Lambda$ there is a valuation ${\mf v}_q: \Lambda \to {\mb R}$ given by
\begin{align}\label{equation30}
{\mf v}_q \Big( \sum_{i =1}^\infty a_i q^{w_i} \Big) = \sup\{ w_i\ |\  a_i \neq 0\}.
\end{align}
It induces a valuation on the $VCF(H)$, given by
\begin{align*}
{\mf v}_q \Big( \sum_{\lb {\mf x} \rb} a_{\lb {\mf x} \rb} \lb {\mf x} \rb \Big) = \max\Big\{  {\mf v}_q ( a_{\lb {\mf x} \rb}) + {\mc A}_H \lb {\mf x} \rb  \Big\}.
\end{align*}
We define a filtration on $VCF(H)$ by
\begin{align*}
VCF_\tau ( H ):= \big\{ {\mf X} \in VCF(H) \ | \  {\mf v}_q ({\mf X})  \leq \tau \big\}.
\end{align*}
Then (3) of Definition \ref{defn31} (or the energy identity \eqref{equation25}), we have
\begin{lemma}
If $(H, J)\in {\mc P}^{reg}_\lambda$, then $\delta_{J, \lambda} \left( VCF_\tau (H) \right) \subset VCF_\tau (H)$.
\end{lemma}
Therefore we have the induced homology groups and induced maps
\begin{align*}
VHF_\tau (H, J, \lambda) = \frac{ {\rm Ker} \big( \delta_{J, \lambda}: VCF_\tau(H) \to VCF_\tau(H) \big) }{{\rm Im} \big( \delta_{J, \lambda}: VCF_\tau(H) \to VCF_\tau(H) \big)}.
\end{align*}
\begin{align*}
\begin{split}
\iota_\tau: VHF_\tau (H, J, \lambda)  \to VHF (H, J, \lambda) \simeq VHF(M);\\
\iota_{\tau_1, \tau_2}: VHF_{\tau_1} (H, J, \lambda) \to VHF_{\tau_2} (H, J, \lambda),\ {\rm if}\ \tau_1 \leq \tau_2.
\end{split}
\end{align*}
which are all consistent.

There are still a few facts we need to recall before we introduce the spectral invariants. Following \cite{FOOO_spectral}, we denote
\begin{align}\label{equation31}
G(M, \mu) := \left\{ \left\langle \left[ \omega^G  \right], A \right\rangle \ |\ A \in S_2^G(M) \right\} \subset {\mb R}.
\end{align}
The action spectrum of $H$ is
\begin{align}\label{equation31X}
{\rm Spec} H = \big\{ {\mc A}_H \lb {\mf x} \rb\ |\ \lb {\mf x} \rb \in {\rm Crit} {\mc A}_H  \big\}.
\end{align}
Indeed, ${\rm Spec} H$ is a $G(M, \mu)$-torsor. We can define the corresponding action spectrum of the reduction $\ov{H}$ on the symplectic quotient, ${\rm Spec} \ov{H}$; and $G(\ov{M}):= \left\{ \alpha \cap \ov{\omega} \ |\ \alpha \in \pi_2(\ov{M}) \right\}$. Then ${\rm Spec} \ov{H}$ is a $G(\ov{M})$-torsor and we have natural inclusions ${\rm Spec} \ov{H} \subset {\rm Spec} H$ and $G(\ov{M}) \subset G(M, \mu)$ such that the former is equivariant with respect to the latter.

\begin{lemma}\label{lemma34}{\rm (}cf. \cite[Lemma 2.2]{Oh_1}{\rm )}
For each $H \in {\mc H}_G(M)$, ${\rm Spec} H$ is a measure zero subset of ${\mb R}$.
\end{lemma}

\begin{lemma}\label{lemma35}
${\rm Spec}H$ only depends on the isotopy class of $\ov{H}$, i.e., the Hamiltonian path $\wt\psi_{\ov{H}} \in \wt{\rm Ham}(\ov{M})$.
\end{lemma}

\bpf Since $Spec(\ov H)\subset Spec(H)$ is a non-empty measure zero subset depending only on the isotopy class of $\ov H$, the conclusion follows.
\epf

\subsection{The spectral invariants}

We now may define the \textit{spectral invariant} in the vortex context as the spectral number in Usher's abstract context. %\cite{Oh_2005}\cite{Schwarz_spectral}\cite{Vi92}. we define the %spectral invariants
Explicitly, this means the functional
$c^v : VHF(M) \times {\mc H}_G^*(M) \to {\mb R}\cup \{-\infty\}$ given by
\begin{align*}
c^v(a, H) =  \inf \big\{ \tau \in {\mb R}\ | \ a \in \iota_\tau ( VHF_\tau (H, J, \lambda) ) \big\}.
\end{align*}

The following main proposition establishes the key properties for $c^v$, with which one should compare with the case of usual Hamiltonian Floer theory as in \cite{Oh_2005}\cite{Schwarz_spectral}. In particular, the spectral invariants are defined for all smooth Hamiltonian functions (see item (4) below).
%\\\\
%\noindent\textbf{A note on notation:} Except in Subsection %\ref{subsection63} where we need to distinguish the vortex spectral %invariants and the usual spectral invariants, we will drop the %superscript in our notation and simply denote our vortex spectral %invariant as $c(a,H)$.

%{\it a priori} it depends on the parameter $\lambda>0$ and on the %almost complex structure $J \in \wt{\mc J}^{reg}_{H, \lambda}$; the %dependence will disappear in a moment. Similar to the case of %ordinary Hamiltonian Floer theory, we will establish the following %list of important formal properties of vortex spectral invariants.

\begin{prop}\label{prop36}
Let $H \in {\mc H}_G^*(M)$ and $a \in VHF(M) \setminus \{0\}$.
\begin{enumerate}
 \item {\bf (Finiteness and spectrality)} If $(H, J) \in {\mc P}^{reg}_\lambda$, then $c(a, H; J, \lambda)$ is finite and $c(a, H; J, \lambda) \in {\rm Spec}H$.  Moreover, $c(a,H+C)=c(a,H)+ C$ for $C\in\mb R$.

 \item {\bf (Independence of $J$ and $\lambda$)} $c(a, H; J, \lambda)$ doesn't depend on the choice of the constant $\lambda>0$ and $J$ for which $(H, J)\in {\mc P}^{reg}_\lambda$. We abbreviate $c(a, H)= c(a, H; J, \lambda)$.

 \item {\bf (Independence of the lifting)} If the restriction of $H^\alpha$ and $H^\beta$ on $\mu^{-1}(0)$ are equal, then $c(a, H^\alpha) = c(a, H^\beta)$. So $c(a, H)$ only depends on $\ov{H} \in {\mc H}^*(\ov{M})$ and we denote $c(a, \ov{H}) = c(a, H)$.

 \item {\bf (Lipschitz continuity)} For $\ov{H}^\alpha, \ov{H}^\beta \in {\mc H}^*(\ov{M})$, we have
\begin{align*}
\int_{S^1} \min_{\ov{M}} \big( \ov{H}^\beta - \ov{H}^\alpha \big) dt \leq c\big( a, \ov{H}^\alpha \big) - c \big( a, \ov{H}^\beta \big) \leq \int_{S^1} \max_{\ov{M}} \big( \ov{H}^\beta - \ov{H}^\alpha \big) dt.
\end{align*}
This allows us to define $c \big( a, \ov{H} \big)$ for all $\ov{H}\in {\mc H}(\ov{M})$.

\item {\bf (Monotonicity)} If $\ov{H}^\alpha \leq \ov{H}^\beta$ on $\ov{M}$, then $c \big( a,\ov{H}^\alpha \big) \geq c \big( a,\ov{H}^\beta \big)$.

\item {\bf (Isotopy invariance)} If $\ov{H}, \ov{K}\in {\mc H}(\ov{M})_0$ and $\wt\phi^1_{\ov{H}} = \wt\phi^1_{\ov{K}} \in \wt{\rm Ham}(\ov{M})$, then $c ( a,\ov{H} )=c ( a,\ov{K} )$. By isotopy invariance, we can define $c\big( a, \wt\phi \big)$ for $\wt\phi\in  \wt{\rm Ham}(\ov{M})$ as $c\big( a, \ov{H}\big)$ for any $\ov{H} \in {\mc H}(\ov{M})_0$ with $\wt\phi_{\ov{H}} = \wt\phi$.

\item {\bf (Hamiltonian invariance)} For $\psi \in {\rm Ham}(\ov{M})$ and $\wt\phi \in \wt{\rm Ham}(\ov{M})$,
\begin{align*}
c\big( a,\psi^{-1} \wt\phi \psi \big)=c\big( a, \wt\phi \big).
\end{align*}

\item {\bf (Shifting property)} For $\lambda \in \Lambda$ and $\wt\phi \in \wt{\rm Ham}(\ov{M})$,
\begin{align*}
c\big( \lambda a, \wt\phi \big) = c \big( a, \wt\phi\big) + {\mf v}_q (\lambda).
\end{align*}

\item {\bf (Triangle inequality)} For $a_0, a_1 \in VHF(M)$ and $\wt\phi_0, \wt\phi_1 \in \wt{\rm Ham}(\ov{M})$,
\begin{align}\label{equation32}
c \big( a_0 * a_1,\wt\phi_0 \wt\phi_1 \big)\leq c \big( a_0,\wt\phi_0 \big) + c \big( a_1, \wt\phi_1 \big).
\end{align}

\item {\bf (Weak normalization)} $c(a,0)$ is finite and depends only on $a\in VHF(M)$.

\end{enumerate}
\end{prop}

\begin{rem}\hfill
\begin{itemize}
\item
%From the above remark, one could define the spectral invariants for %constant functions, namely, $c(a, C)$.

By the time of writing, it is not clear to the authors whether $c(a,0)=0$, which is called the \textit{normalization property} of spectral theory in the Hamiltonian Floer context.  Nonetheless, item (10) will be enough for supplementing the normalization in defining partial quasi-states.

\item The ordinary spectral invariants of $\ov{M}$ satisfy the symplectic invariance property. That means, for any symplectomorphism $\psi: \ov{M} \to \ov{M}$, $c( \psi^* a, \psi^* H ) = c (a, H )$. Here $\psi^* H_t(x) = H_t (\psi(x))$ is the pull-back Hamiltonian. This is an immediate implication of the symplectic invariance of Floer's equation. In contrast, the vortex equation is only invariant under $G$-invariant symplectomorphisms of $M$ (which descends to symplectomorphisms of $\ov{M}$). It is not known if every symplectomorphism of $\ov{M}$ can be lifted to a $G$-invariant one on $M$.

While this unsatisfactory feature will not undermine our construction, it leads to the the conceptual consequence that, even if one could define superheaviness of a subset as in \cite{Entov_Polterovich_3}, it does not follow that they cannot be disjoined by arbitrary symplectomorphisms as in the Hamiltonian Floer context.
\end{itemize}
\end{rem}

\subsection{Proof of Proposition \ref{prop36}}\label{s:Proof}

\subsubsection{Finiteness and spectrality} This follows from \cite[Theorem 1.3, Theorem 1.4]{Ush08}, which applies for all filtered Floer-Novikov chain complex.

\subsubsection{Independence of $J$ and $\lambda$} This fact follows from the estimate on action functional in Proposition \ref{prop27}. Let $\lambda^\alpha, \lambda^\beta >0$ and $J^\alpha$, $J^\beta$ be $S^1$-families of almost complex structures such that $(H, J^\alpha) \in {\mc P}^{reg}_{\lambda^\alpha}$, $(H,J^\beta) \in {\mc P}^{reg}_{\lambda^\beta}$. Choose homotopies $\wt{J}$ and $\wt\lambda$ interpolating between $(J^\alpha, \lambda^\alpha)$ and $(J^\beta, \lambda^\beta)$ as when we define the continuation maps, while $H$ remains independent of $s$. If $a \in VHF(M)\setminus \{0\}$, then for any $\epsilon>0$, $a$ is represented by a nonzero cycle
\begin{align*}
{\mf X}^\alpha = \sum a_{ \lb {\mf x} \rb} \lb {\mf x} \rb \in VCF(M; H)
\end{align*} such that ${\mf v}_q ({\mf X}) \leq c(a, H, J^\alpha, \lambda^\alpha)  +  \epsilon$. Suppose $\Phi: VCF(M; H) \to VCF(M; H)$ is the chain level continuation map induced from $(\wt{J}, \wt\lambda)$. Then $\Phi({\mf X}) = \sum a_{ \lb {\mf x} \rb} \Phi( \lb{\mf x} \rb)\in VCF(M; H)$ is a Floer cycle representing $a \in VHF(M)$. Then by Proposition \ref{prop27},
\begin{align*}
{\mf v}_q(\Phi({\mf X})) \leq {\mf v}_q({\mf X}) \leq c ( a, H; J^\alpha, \lambda^\alpha ) + \epsilon.
\end{align*}
Hence $c ( a, H; J^\beta, \lambda^\beta) \leq c ( a, H; J^\alpha,  \lambda^\alpha )+ \epsilon$. Reversing the continuation map, we see that $c(a, H; J^\alpha, \lambda^\alpha) = c(a, H; J^\beta, \lambda^\beta)$.

\subsubsection{Independence of the lifting} Again this follows from the continuation principle, but we have to use the adiabatic limit argument to ``screen out'' the contribution from the difference between $H^\alpha$ and $H^\beta$ away from $\mu^{-1}(0)$.

Let $\lambda^\alpha, \lambda^\beta>0$ and $(H^\alpha, J^\alpha)\in {\mc P}^{reg}_{\lambda^\alpha}$, $(H^\beta, J^\beta) \in {\mc P}^{reg}_{\lambda^\beta}$. Suppose
\begin{align*}
H^\alpha|_{\mu^{-1}(0) \times S^1} = H^\beta|_{\mu^{-1}(0)\times S^1}.
\end{align*}
Choose a non-decreasing cut-off function $\rho: {\mb R} \to [0, 1]$ supported on $[0, +\infty)$ and equal to 1 on $[1, +\infty)$. Define the homotopy $\wt{H}$ between $H^\alpha$ and $H^\beta$ by $H_{s, t} = (1- \rho(s)) H^\alpha_t + \rho(s) H^\beta_t$. Then because $H^\alpha|_{\mu^{-1(0)}} = H^\beta|_{\mu^{-1}(0)}$, both of which are compactly supported, and because $\mu$ is proper, there exists a constant $C>0$ such that
\begin{align}\label{equation33}
\Big| \frac{\partial H_{s, t}}{\partial s} (x) \Big| \leq C |\mu(x)|.
\end{align}
For {\it any} $\lambda_0>0$, choose a homotopy $\wt\lambda(s)$ between $\lambda^\alpha$ and $\lambda^\beta$ such that
\begin{align*}
\inf_{s\in [0,1]} \wt\lambda(s) \geq \lambda_0.
\end{align*}
Choose a generic homotopy $\wt{J} = (J_{s, t})_{(s, t)\in \Theta}$ which is exponentially close to $J^\alpha$ (resp. $J^\beta$) on $\Theta_-$ (resp. $\Theta_+$). Then the chain level continuation map associated to the homotopy $(\wt{H}, \wt{J}, \wt\lambda)$
\begin{align*}
\Phi_{\wt\lambda} : VCF(M; H^\alpha) \to VCF(M; H^\beta)
\end{align*}
in which we would like to emphasize the role of $\wt\lambda$.

Suppose there is $\epsilon>0$ such that
\begin{align}\label{equation34}
c(a, H^\beta) - c(a, H^\alpha) \geq 2\epsilon.
\end{align}
There exists a Floer cycle ${\mf X}^\alpha = \sum a_i \lb {\mf x}_i \rb \in VCF(M; H^\alpha)$ representing $a$ such that
\begin{align}\label{equation35}
{\mf v}_q({\mf X}^\alpha) \leq c(a, H^\alpha) + \epsilon.
\end{align}
Since $\Phi_{\wt\lambda} ({\mf X}^\alpha)$ represents $a$, we have ${\mf v}_q(\Phi_{\wt\lambda} ({\mf X}^\alpha)) \geq c(a, H^\beta)$, which implies that there exists $ \lb {\mf y} \rb \in {\rm Crit} {\mc A}_{H^\beta}$ which {\it contributes} to $\Phi_{\wt\lambda} ( {\mf X}^\alpha )$ and
\begin{align}\label{equation36}
{\mc A}_{H^\beta}\lb {\mf y} \rb \geq c(a, H^\beta).
\end{align}
By the definition of continuation map, there exists $\lb {\mf x} \rb \in {\rm Crit} {\mc A}_{H^\alpha}$ which contributes to ${\mf X}^\alpha_\lambda $ and the moduli space ${\mc N} ( \lb{\mf x} \rb,  \lb {\mf y} \rb; \wt{H}, \wt{J}, \wt\lambda )$ is nonempty. Hence
\begin{align}\label{equation37}
{\mc A}_{H^\alpha} \lb {\mf x} \rb \leq {\mf v}_q({\mf X}^\alpha).
\end{align}
Take a solution ${\mf w} = (u, \Phi, \Psi)$ representing an element in the above moduli space. Then \eqref{equation215} and \eqref{equation34}--\eqref{equation37} imply
\begin{align}\label{equation38}
\int_{\Theta} \frac{\partial H_{s, t} }{\partial s}(u) ds dt \leq {\mc A}_{H^\alpha} \lb {\mf x} \rb - {\mc A}_{H^\beta} \lb  {\mf y} \rb \leq {\mf v}_q({\mf X}^\alpha) - c(a, H^\beta) \leq - \epsilon.
\end{align}
On the other hand, since $\partial_s H$ has compact support in $\Theta \times M$, there is a constant $C(\wt{H})$, independent of $\lambda$, such that
\begin{align}\label{equation39}
\mc{YMH}_{\wt{H}}({\mf w}) \leq -\epsilon - \int_\Theta \frac{ \partial H_{s, t}}{\partial s}(u) ds dt\leq \int_\Theta \Big| \frac{\partial H_{s, t} }{\partial s} (u) \Big| ds dt \leq C(\wt{H}).
\end{align}
Then by \eqref{equation33}, \eqref{equation38} and \eqref{equation39}, and the definition of energy,
\begin{align*}
\begin{split}
\epsilon \leq &\  - \int_\Theta  \frac{\partial H_{s, t} }{\partial s}(u) ds dt  \leq  \int_{\Theta} \Big| \frac{\partial H_{s, t}}{\partial s}(u) \Big| ds dt \\
\leq &\   \int_{[0,1]\times S^1} C \big|\mu(u) \big| ds dt  \leq C  \frac{ \sqrt{2\pi}}{\inf_{s\in [0,1]}\wt\lambda(s) } \left( \int_{\Theta} \wt\lambda(s)^2 \big| \mu(u) \big|^2 ds dt \right)^{1/2} \\
\leq &\ C \frac{\sqrt{ 2\pi}}{ \lambda_0} \sqrt{ \mc{YMH}_{\wt{H}} ({\mf w}) } \leq C \frac{\sqrt{ 2\pi}}{ \lambda_0} \sqrt{ C(\wt{H})}.
\end{split}
\end{align*}
This is not true if we take $\lambda_0$ sufficiently large. Hence $c(a, H^\alpha) = c(a, H^\beta)$.

\subsubsection{Lipschitz continuity} This also follows from a continuation and adiabatic limit argument. Suppose that there exists $\epsilon>0$ such that
\begin{align}\label{equation310}
c\big( a, \ov{H}^\beta \big) > c \big( a, \ov{H}^\alpha \big) - \int_{S^1} \min_{\ov{M}}  \big( \ov{H}^\beta - \ov{H}^\alpha \big) dt + 3\epsilon.
\end{align}

Now we make use of the more delicate notion of ``admissible'' lifts of Hamiltonians on $\ov{M}$, a notion which only depends on the infinitesimal behavior of the lifts along $\mu^{-1}(0)$ (for more details, see \cite[Definition 6.5]{Xu_VHF}). Then we can choose admissible lifts $H^\alpha, H^\beta\in {\mc H}_G(M)$ of $\ov{H}^\alpha$, $\ov{H}^\beta$ and assuming for all $t \in S^1$,
\begin{align}\label{equation311}
\begin{split}
\max_M \big( H^\beta_t - H^\alpha_t \big) \leq \max_{\ov{M}} \big( \ov{H}^\beta_t - \ov{H}^\alpha_t \big) + \epsilon,\\
\min_M \big( H^\beta_t-H^\alpha_t \big) \geq \min_{\ov{M}} \big( \ov{H}^\beta_t - \ov{H}^\alpha_t \big) - \epsilon.
\end{split}
\end{align}
Choose the homotopy $\wt{H} = H_{s, t}= (1-\rho(s)) H^\alpha + \rho(s) H^\beta$, where $\rho: {\mb R} \to [0, 1]$ is a smooth cut-off function which equals to $0$ for $s \leq -1$ and equals to $1$ for $s \geq 1$. For every $\lambda>0$, take $J^\alpha \in \wt{\mc J}^{reg}_{H^\alpha, \lambda}$, $J^\beta \in \wt{\mc J}^{reg}_{H^\beta, \lambda}$ (see Theorem \ref{thm23}) so that $(H^\alpha, J^\alpha), (H^\beta, J^\beta) \in {\mc P}^{reg}_\lambda$. Choose a generic homotopy $\wt{J} = (J_{s, t})_{(s, t) \in \Theta}$ which is exponentially close to $J^\alpha$ (resp. $J^\beta$) on $\Theta_-$ (resp. $\Theta_+$).

By the definition of the spectral numbers, there exists a Floer cycle ${\mf X}^\alpha \in VCF ( M; H^\alpha)$ representing $a\in VHF(M)$ such that ${\mf v}_q({\mf X}^\alpha) \leq c ( a, \ov{H}^\alpha ) + \epsilon$. Let $\Phi: VCF (M; H^\alpha) \to VCF (M; H^\beta)$ be the chain level continuation map induced from the homotopy $(\wt{H}, \wt{J}, \lambda)$. Hence there exist $\lb {\mf y} \rb \in {\rm Crit} {\mc A}_{H^\beta}$ contributing to $\Phi({\mf X}^\alpha)$, $\lb {\mf x} \rb \in {\rm Crit} {\mc A}_{H^\alpha}$ contributing to ${\mf X}^\alpha$ such that ${\mc N} ( \lb {\mf x} \rb, \lb{\mf y} \rb; \wt{H}, \wt{J}, \lambda \big)  \neq \emptyset$ and
\begin{align}\label{equation312}
{\mc A}_{H^\alpha} \lb{\mf x} \rb \leq {\mf v}_q({\mf X}^\alpha) \leq c(a, \ov{H}^\alpha) + \epsilon,\ {\mc A}_{H^\beta} \lb {\mf y} \rb = {\mf v}_q(\Phi({\mf X}^\alpha)) \geq c(a, \ov{H}^\beta).
\end{align}
For any solution ${\mf w}$ to \eqref{equation214} representing an element in this moduli space, by \eqref{equation215} and \eqref{equation310}--\eqref{equation312},
\begin{align*}
\begin{split}
\mc{YMH}_{\wt{H}}  ({\mf w}) = &\ {\mc A}_{H^\alpha}\lb{\mf x} \rb - {\mc A}_{H^\beta} \lb  {\mf y} \rb - \int_\Theta \frac{\partial H_{s, t}}{\partial s} (u) ds dt \\
\leq &\ c(a, \ov{H}^\alpha) + \epsilon - c(a, \ov{H}^\beta)  - \int_\Theta \frac{\partial H_{s, t}}{\partial s} (u) ds dt \\
	\leq &\ -2 \epsilon + \int_{S^1} \min_{\ov{M}} \big( \ov{H}^\alpha - \ov{H}^\beta \big) dt  + \int_\Theta \rho'(s) \big( H^\alpha(u) - H^\beta(u) \big) ds dt\\
	\leq &\ - 2\epsilon + \int_{S^1} \min_{\ov{M}} \big( \ov{H}^\alpha -\ov{H}^\beta \big) dt - \int_{S^1} \min_M  \big( H^\alpha - H^\beta \big) dt \leq -\epsilon
	\end{split}
	\end{align*}	
which is absurd. Therefore \eqref{equation310} is false and we must have
\begin{align*}
c (a, H^\beta ) \leq c ( a, H^\alpha ) - \int_{S^1} \min_{\ov{M}} \big( \ov{H}^\beta- \ov{H}^\alpha \big) dt.
\end{align*}
The other inequality follows by a similar argument, with the direction of the continuation map reversed.

\subsubsection{Monotonicity} It follows trivially from the Lipschitz continuity.

\subsubsection{Isotopy invariance} Suppose we have $\ov{H}^\alpha, \ov{H}^\beta \in {\mc H}^*(\ov{M})_0$ such that $\wt\psi:= \wt\psi^\alpha:= \wt\psi_{\ov{H}^\alpha} = \wt\psi_{\ov{H}^\beta}=: \wt\psi^\beta \in \wt{\rm Ham}(\ov{M})$. Then there exists a continuous path $\ov{H}^s \in {\mc H}^*(\ov{M})_0, s\in [0,1]$ such that $\ov{H}^0 = \ov{H}^\alpha$ and $\ov{H}^1= \ov{H}^\beta$ such that for each $s$, $\wt\psi_{\ov{H}^s} = \wt\psi$. Then for every $a \in VHF(M)$, the function
\begin{align*}
s\mapsto c\left( a, \ov{H}^s \right)
\end{align*}
is continuous and takes value in a fixed (by Lemma \ref{lemma35}) measure zero set (by Lemma \ref{lemma34}) of ${\mb R}$. Such a function can only be a constant.

\subsubsection{Hamiltonian invariance} Any Hamiltonian diffeomorphism $\psi \in {\rm Ham}(\ov{M})$ can be lifted to a Hamiltonian diffeomorphism of $M$ which commutes with the $G$-action. Then the Hamiltonian invariance follows from the invariance of the symplectic vortex equation under the action of symplectomorphisms that commute with the $G$-action.

\subsubsection{Shifting property} It is elementary to see that
\begin{align*}
c(\lambda a,H)\leq c(a,H)+ {\mf v}_q (\lambda)\leq c(\lambda a,H)+ {\mf v}_q (\lambda)+ {\mf v}_q (\lambda^{-1})=c(\lambda a,H).
\end{align*}
Here the last inequality follows from the fact that ${\mf v}_q (\lambda) + {\mf v}_q (\lambda^{-1})=0$.

\subsubsection{Triangle inequality}\label{triangle}

There are several ways to prove the triangle inequality for spectral invariants. The approaches of \cite{Schwarz_spectral} and \cite{FOOO_spectral} are based on explicit choice of the Hamiltonian connections on the pair of pants. Here we take the approach of \cite{Oh_2005}, which was based on Entov's work \cite{Entov_2001} on the relations between K-area and Hamiltonian diffeomorphism groups. Consider the infinite-dimensional Lie group ${\rm Ham}^G(M)$ consisting of compactly supported Hamiltonian diffeomorphisms of $(M, \omega)$ which commute with the $G$-action. This is a subgroup of ${\rm Ham}(M)$ of compactly-supported Hamiltonian diffeomorphisms. The usual Hofer metric on ${\rm Ham}(M)$ restricts to a Finsler bi-invariant pseudo-metric on ${\rm Ham}^G(M)$ (see \cite[page 95]{Entov_2001} for the definition).

Choose $\lambda>0$. Let $\wt\phi_0, \wt\phi_1 \in \wt{\rm Ham}^G(M)$ be generated by Hamiltonians $H^0, H^1 \in {\mc H}_G(M)$ which belong to regular pairs $(H^0, J^0)$ and $(H^1, J^1)$ relative to $(\ov{H}^0, \lambda)$ and $(\ov{H}^1, \lambda)$ respectively. Let $H^\infty \in {\mc H}_G^*(M)$ be given by $H^\infty_t(x) = H^0_t(x) + H^1_t\big(( \wt\phi_0^t)^{-1}x \big)$, with induced Hamiltonian isotopy $\wt\phi_\infty \in \wt{\rm Ham}^G(M)$. We choose $\wt\phi_0$, $\wt\phi_1$ generically so that the Hamiltonian isotopy on $\ov{M}$ induced from $\wt\phi_\infty$ is nondegenerate.

Let $\Sigma_0 = S^2 \setminus D_0 \cup D_1 \cup D_\infty$ be the pair-of-pants with three boundary components. Choose an area form $\Omega$ on $\Sigma_0$ with total area 1. We consider ${\rm Ham}^G(M)$-connections on $\Sigma_0$ which are flat near the boundary of $\Sigma_0$ and whose restrictions to the boundary circles $\partial D_0$, $\partial D_1$, $\partial D_\infty$ induce parallel transports which are in the same conjugacy classes of $\wt\phi_0$, $\wt\phi_1$ and $\wt\psi_\infty$, respectively. We denote the set of such ${\rm Ham}^G(M)$-connections on $\Sigma_0$ by ${\rm Conn}^G(\wt\phi_0, \wt\phi_1, \wt\phi_\infty)$. Any such connection $\nabla$ is in particular a ${\rm Ham}(M)$-connection, so we can apply Entov's result.

\begin{thm}\cite[Theorem 3.6.1]{Entov_2001}
For any $\epsilon>0$, there exists $\nabla_\epsilon\in {\rm Conn}^G(\wt\phi_0, \wt\phi_1, \wt\phi_\infty)$ on $\Sigma_0$ such that
\begin{align}\label{equation313}
\sup_{\wt{M}_0} \big| R_{\nabla_\epsilon} \big| \leq \epsilon.
\end{align}
\end{thm}
\begin{proof}
First, by Entov's theorem, there is a ${\rm Ham}(M)$-connection satisfying this condition. Averaging over $G$ with respect to the Haar measure, we obtain a ${\rm Ham}^G(M)$-connection satisfying the same estimate.
\end{proof}

Now choose a volume form $\nu$ on $\Sigma = \Sigma_0 \cup U_0\cup U_1\cup U_\infty \simeq S^2 \setminus \{0, 1, \infty\}$ where the notations are the same as in Subsection \ref{subsection27}, such that $\nu|_{U^\sigma} = \lambda^2 ds dt$. We can extend $\nabla$ to $\wt{M} = M \times \Sigma$ which is flat over each $U_\sigma$. Up to gauge transformation we may take $\nabla|_{U_\sigma} = H^\sigma dt$. Choose a family of almost complex structures $\wt{J} = (J_z)_{z\in \Sigma}$ which are exponentially close to $J^\sigma$ over $U_\sigma$ $(\sigma = 0, 1, \infty)$, where $(H^\sigma, J^\sigma)$ is a regular pair relative to $(\ov{H}^\sigma, \lambda)$. Then the data $(\nabla_\epsilon, \wt{J})$ and $\nu$ induces a chain level map
\begin{align*}
\Phi_{\nabla_\epsilon, \wt{J}}: VCF(H^0, J^0) \otimes VCF(H^1, J^1) \to VCF(	 H^\infty, J^\infty),
\end{align*}
which is defined by counting gauge equivalence classes of solutions to \eqref{equation313} (with $\nabla$ replaced by $\nabla_\epsilon$). On homology level it coincides with the pair-of-pants product.

Now we prove the triangle inequality. For any $a_0, a_1 \in VHF(M)$, if $a_0 * a_1 = 0$, then $c\big( a_0 * a_1, \wt\phi_0 \wt\phi_1 \big) = -\infty$ so \eqref{equation32} holds. If $a_0 * a_1 \neq 0$, then for any $\epsilon>0$, $a_0$, $a_1$ are represented by Floer cycles ${\mf X}_\sigma \in VCF(M; H^\sigma, J^\sigma, \lambda )$ ($\sigma = 0, 1$) with ${\mf v}_q({\mf X}_\sigma ) \leq c(a_\sigma, \wt\phi_\sigma) + \epsilon$. Any ${\mf x}_\sigma \in {\rm Crit} {\mc A}_{H^\sigma}$ which contributes to ${\mf X}_\sigma$ satisfies
\begin{align}\label{equation314}
{\mc A}_{H^\sigma}({\mf x}_\sigma) \leq c(a_\sigma, \wt\phi_\sigma) + \epsilon,\ \sigma = 0, 1.
\end{align}
Moreover, the cycle ${\mf X}_\infty:= \Phi_{\nabla_\epsilon, J}({\mf X}_0, {\mf X}_1)$ represents $a_0 * a_1$. Choose $\lb {\mf x}_\infty \rb\in {\rm Crit} {\mc A}_{H^\infty}$ which contributes to ${\mf X}_\infty$ with ${\mf v}_q({\mf X}_\infty) = {\mc A}_{H^\infty}({\mf x}_\infty)$. Then the corresponding moduli space contains at least one element, of which we choose a representative $(A, u)$. By Theorem \ref{thm25}, \eqref{equation313}, \eqref{equation314}, we have
\begin{align*}
\begin{split}
{\mf v}_q( {\mf X}_\infty) = {\mc A}_{H^\infty}( {\mf x}_\infty ) = &\ {\mc A}_{H^0}(  {\mf x}_0 ) + {\mc A}_{H^1} ( {\mf x}_1 ) - \mc{YMH}_{\nabla_\epsilon} (A, u) - \int_{\Sigma_0} R_{\nabla_\epsilon}(u)\\
																				 \leq &\ c(a_0, \wt\phi_0) + c(a_1, \wt\phi_1) + 2 \epsilon + \epsilon \int_{\Sigma_0} \nu.
\end{split}
\end{align*}
Since $\epsilon$ is arbitrarily small, \eqref{equation32} follows.

\subsubsection{Weak normalization property} This follows from item (1)--(3) of this proposition.

\section{Quasi-morphisms, quasi-states and heaviness}\label{s:QMQS}

We first recall the basic concepts of partial quasi-states. Note that the definition of partial quasi-states are not completely consistent in the literature. Here our axioms and signs follow the convention of \cite[Definition 13.3]{FOOO_spectral} except that we add {\it additivity with constants} and replace the {\it symplectic invariance} by {\it Hamiltonian invariance}. Compared to \cite{Entov_Polterovich_2} and \cite{Entov_Polterovich_3}: \cite{Entov_Polterovich_2} does not include \textit{additivity with constants} or \textit{triangle inequality}; \cite{Entov_Polterovich_3} has the additional {\it characteristic exponent property} which is irrelevant here.

Let $(M, \omega)$ be a compact symplectic manifold and $C(M)$ be the Banach algebra of real valued continuous functions.

\begin{defn}\label{defn41}
A functional $\upzeta: C(M) \to {\mb R}$ is called a {\bf partial quasi-state} on $M$ if it satisfies the following conditions.
\begin{enumerate}
\item {\bf (Lipschitz continuity)} $|\upzeta(F_1) - \upzeta(F_2)| \leq \| F_1 - F_2\|_{C^0(M)}$ for $F_1, F_2\in C(M)$.

\item {\bf (Semi-homogeneity)} $\upzeta(\lambda F) = \lambda\upzeta(F)$ when $\lambda\in\R_{\ge0}$.

\item {\bf (Monotonicity)} $\upzeta(F) \leq \upzeta(G)$ whenever $F\leq G$.

\item {\bf (Normalization)} $\upzeta(1) = 1$.

\item {\bf (Partial Additivity)} If $F_1,F_2\in C^\infty(M)$, $\{F_1,F_2\}=0$ and $supp(F_2)$ is displaceable, then $\upzeta(F_1+F_2)=\upzeta(F_1)$.

\item {\bf (Hamiltonian invariance)} $\upzeta(F) = \upzeta(F \circ \phi)$ for $\phi \in {\rm Ham}(M)$.

\item{\bf (Additivity with constants)} $\upzeta(F+\alpha) = \upzeta(F)+\alpha$ when $\alpha\in\R$.

\item {\bf (Vanishing)} $\upzeta(F) = 0$ when ${\rm supp}(F)$ is Hamiltonian displaceable.

\item {\bf (Triangle inequality)} $\upzeta(F + G) \geq \upzeta(F) + \upzeta(G)$ for $F, G\in C^\infty(M)$ and $\{ F, G \} = 0$.
\end{enumerate}
\end{defn}

We remark that the normalization property (4) can be deduced from (2) and (7).

Now we recall the concept of pre-quasimorphism, following \cite[Definition 13.6]{FOOO_spectral}, where we replace the {\it symplectic invariance} by {\it Hamiltonian invariance}.

\begin{defn}\label{defn42}
We call a map $\upmu: \wt{\rm Ham}(M) \to {\mb R}$ a {\bf pre-quasimorphism}, if the following conditions are satisfied.
\begin{enumerate}
\item {\bf (Lipschitz continuity)} There exists a constant $C>0$ such that  for $\wt\phi, \wt\psi \in \wt{\rm Ham}(M)$, $\big|\upmu(\wt\psi)- \upmu (\wt\psi) \big| \leq C \big\| \wt\psi \wt\phi^{-1} \big\|$, where $\big\| \cdot \big\|$ is the Hofer norm.

\item {\bf (Semi-homogeneity)} For $\wt\phi\in \wt{\rm Ham}(M)$ and $n \geq 1$, we have $\upmu(\wt\phi^n) = n \upmu(\wt\phi)$.

\item {\bf (Controlled quasi-additivity)} If $U \subset M$ is displaceable, then there exists a constant $K>0$ depending on $U$ such that for $\wt\phi, \wt\psi\in \wt{\rm Ham}(M)$,
\begin{align*}
\big| \upmu(\wt\psi\wt\phi) - \upmu(\wt\phi) - \upmu(\wt\psi) \big| \leq K \min \big( \big\| \wt\phi\big\|_U, \big\| \wt\psi \big\|_U \big).
\end{align*}
Here $\big\| \wt\phi \big\|_U$ is the Banyaga's fragmentation norm (see \cite{Banyaga_1978}).

\item {\bf (Hamiltonian invariance)} $\upmu(\wt\phi) = \upmu(\psi\wt\phi	 \psi^{-1})$ for $\wt\phi \in \wt{\rm Ham}(M)$, $\phi\in {\rm Ham}(M)$.

\item {\bf (Calabi property)} If $U \subset M$ is displaceable, then the restriction of $\upmu$ to $\wt{\rm Ham}(U)$ coincides with the Calabi homomorphism ${\rm Cal}_U$.
\end{enumerate}
\end{defn}
Here in (6), the {\bf Calabi functional} ${\rm Cal}_U: \wt{\rm Ham}(U) \to {\mb R}$ is defined for all Hamiltonian diffeomorphisms supported in $U$, given by
\begin{align*}
{\rm Cal}_U(\wt\phi_H) = \int_0^1 dt \int_M H_t \omega^n,
\end{align*}
where $\wt\phi_H$ is the Hamiltonian isotopy generated by $H_t$.

\subsection{Partial quasi-states and pre-quasimorphisms from vortex Floer theory}

Let $e \in VHF(M)$ be an idempotent element and assume that $e \neq 0$ in $VHF(M)$. For $\wt\phi\in \wt{\rm Ham}(\ov{M})$, define $c(e, \wt\phi) = c(e, \ov{H})$ for any $\ov{H}\in {\mc H}(\ov{M})$ such that $\wt\phi(1) = \phi_{\ov{H}}^1$. By the isotopy invariance of the spectral numbers, $c(e, \wt\phi)$ is well-defined. On the other hand, we can also evaluate $c(e, F)$ for $F \in C(\ov{M})$.

Define $\upmu_e: \wt{\rm Ham}(\ov{M}) \to {\mb R}\cup \{\pm \infty\}$ and $\upzeta_e: C(\ov{M}) \to {\mb R} \cup \{\pm \infty\}$ by
\begin{align}\label{equation41}
\upmu_e (\wt\phi) = {\rm vol}(\ov{M}) \lim_{n \to \infty} \frac{c (e, \wt\psi^n)}{n},\ \upzeta_e (F) = - \lim_{n \to \infty} \frac{ c(e, nF)}{n}.
\end{align}

\begin{thm}\label{thm43}
$\upmu_e$ and $\upzeta_e$ take finite values. Moreover, $\upmu_e$ is a pre-quasimorphism on $\wt{\rm Ham}(\ov{M})$ and $\upzeta_e$ is a partial quasi-state on $\ov{M}$.
\end{thm}

The proof of the above theorem follows almost word by word as the proof of \cite[Theorem 4.1]{Entov_Polterovich_2}, which we will not reproduce.  The reason is that the proof relies almost entirely on the formal properties of spectral invariants.  There are, however, two points worth mentioning in our situation.

Firstly, the only non-formal part of the proof involves a trick of spectrum shift due to Ostrover \cite{Ostrover_2003}. In particular, one needs that for a displaceable open subset $U \subset \ov{M}$,
\begin{align}
\label{equation42}c(a,f\phi)=c(a,f) + \frac{{\rm Cal}_U(\phi)}{vol(\ov M, \ov\omega)},
\end{align}
for $a\in VHF(M)$, $f\in\wt{\rm Ham}(\ol M)$ and $\phi\in\wt{\rm Ham}(U)$. Note that Ostrover's proof in fact shows the following:

\blem\label{lemma44} In the situation above, there exists a path of Hamiltonian diffeomorphism $\phi_s\in \wt{\rm Ham}(U)$ such that
\begin{align}
\label{equation43} {\rm Spec}^{\ov{M}} (f\circ\phi_s)= {\rm Spec}^{\ov{M}} (f) + \frac{s {\rm Cal}_U(\phi)}{vol( \ov{M}, \ov\omega)}.
\end{align}
\elem
Here ${\rm Spec}^{\ov{M}}$ denotes the spectrum of the ordinary Hamiltonian Floer theory of $\ov{M}$. Recall that ${\rm Crit}{\mc A}_H$ of our equivariant action functional on ${\mf L}$ is a finitely generated $\Gamma$-torsor, and there is a canonical inclusion ${\rm Crit} {\mc A}_{\ov{H}} \to {\rm Crit} {\mc A}_H$ such that $\Gamma {\rm Crit} {\mc A}_{\ov{H}} = {\rm Crit} {\mc A}_H$. This fact and Lemma \ref{lemma44} imply that \eqref{equation43} holds if replacing ${\rm Spec}^{\ov{M}}$ by the vortex Floer spectrum \eqref{equation31X}. Then \eqref{equation42} follows.

Secondly, our spectral invariants miss the normalization property.  The only place this was used is when proving $\upzeta(0)=0$, but since we know $c(a,0)$ is finite, the property follows already.  Especially the proof of the vanishing property will encounter a similar cancellation due to homogenization thus does not raise any concerns.

\begin{example} The basic example for partial symplectic quasi-states on any symplectic manifold is given by the one defined by the identity element of quantum cohomology \cite{Entov_Polterovich_3}.

In the Hamiltonian $G$-manifold setting, one could define a partial symplectic quasi-state using the identity ${\mbms 1}_{\mbms H}$ in vortex Floer cohomology, of which the existence was proved in \ref{subsection27}, assuming that ${\mbms 1}_{\mbms H}\neq 0$.  By the time of writing, the authors do not know if the two partial quasi-states coincide.  Moreover, it is not clear if such partial quasi-states are always non-trivial.  We will establish the non-triviality of these partial quasi-states on all toric symplectic manifolds in Section \ref{s:app}.

\end{example}
%For completeness we include a detailed proof in the appendix. The %argument is very much formal and it is desirable to have an axiomatic %approach, based on the basic relation such as displaceability and the %basic properties of $\wt{\rm Ham}(M, \omega)$ such as fragmentation.

\subsection{Heaviness}

We use the same notions as in \cite{FOOO_spectral}. Let $\ov{M}$ be our symplectic quotient.

\begin{defn}
Let $\ov{H} = (\ov{H}_t)_{t\in [0,1]}$ be a time-dependent Hamiltonian on $\ov{M}$ and $Y \subset \ov{M}$ be a closed subset. We define
\begin{align*}
E_\infty^\pm( \ov{H}; Y) = \sup \big\{ \pm \ov{H} (t, p)\ |\ (t, p) \in [0,1]\times Y \big\}.
\end{align*}
%\begin{align*}
%E_\infty(H; Y) = E_\infty^-(H; Y) + E_\infty^+(H; Y).
%\end{align*}
For $\wt\phi \in \wt{\rm Ham}(\ov{M})$, define
\begin{align*}
e_\infty^\pm( \wt\phi; Y) = \inf \Big\{ E_\infty^\pm(\ov{H}; Y)\ |\ \wt\phi = [\phi_{\ov{H}} ] \in \wt{\rm Ham}(\ov{M}),\ \ov{H}_t \in {\mc H}(\ov{M})_0 \Big\}. %\ e_\infty(\wt\phi; Y) = \inf_{\wt\phi = [\phi_H]} \{ E_\infty(H; Y)\}.
\end{align*}
\end{defn}

\begin{defn}
Let $\upzeta: C(M) \to {\mb R}$ be a functional. A closed subset $Y \subset M$ is called {\bf $\upzeta$-heavy} if
\begin{align*}
\upzeta(H) \leq \sup_Y H,\ \forall H \in C(M).
\end{align*}
%It is called {\bf $\upzeta$-superheavy} if
%\begin{align*}
%\upzeta(H) \geq \inf_Y H,\ \forall H \in C(M).
%\end{align*}
\end{defn}

\begin{defn}
Let $\upmu: \wt{\rm Ham}(M)\to {\mb R}$ be a function. A closed subset $Y \subset M$ is called {\bf $\upmu$-heavy} if
\begin{align*}
\upmu( \wt\phi) \geq - {\rm vol}(M) e_\infty^+(\wt\phi; Y),\ \forall \wt\phi \in \wt{\rm Ham}(M).
\end{align*}
%It is called {\bf $\upmu$-superheavy} if
%\begin{align*}
%\upmu(\wt\phi) \leq {\rm vol}(M) e_\infty^-( \wt\phi; Y),\ \forall \wt\phi \in \wt{\rm Ham}(M).
%\end{align*}
\end{defn}

\begin{rem}
The observation in \cite{Entov_Polterovich_3} is that a $\upzeta$-heavy set of arbitrary symplectic partial quasi-state $\zeta$ is \textbf{non-displaceable}.  This uses only the formal properties of partial quasi-states hence carries over to our case.  Entov and Polterovich then showed a stronger result of \textbf{stably non-displaceability} of heavy sets, which will be a consequence of our result in Section \ref{subsection63}.
\end{rem}

%In the vortex theory, one then defines symplectic quasi-states and %pre-quasimorphisms as in \cite{Entov_Polterovich_2}:

%\bthm

%\ethm

%\subsection{Partial quasi-states via spectral numbers}

\section{The closed-open map}\label{s:COMap}

In this section  we establish the closed-open string map between the vortex Hamiltonian Floer cohomology and Woodward's quasimap Floer cohomology.  Also, we prove a key property to our applications, that the closed-open map sends the identity to the identity.  In fact, the closed-open map is a ring map, and the proof follows the same line as presented but we will not make use of it here.  On the filtration level, we will establish the necessary energy estimate for the applications in Section \ref{s:app}.

We shall remind the reader that there are many versions of closed-open string maps in the literature \cite{FOOO_Book}, \cite{Biran_Cornea}.  Since we will follow Woodward's pearly definition of quasimap Floer cohomology, our version of closed-open map will take the form of \cite{Biran_Cornea}, while our exposition follows mostly \cite{FOOO_spectral}.

%we consider the relation between spectral invariants and %Lagrangian Floer theory, following the idea of %\cite{FOOO_spectral}. The counterpart of the Lagrangian %Floer theory used in \cite{FOOO_spectral} is the %quasimap Floer theory introduced by Woodward in %\cite{Woodward_toric}. We remark that the term %%``quasimap'' follows the name used by Givental %\cite{Givental_96}, which means that certain maps into %$M$ are viewed as ``quasi'' maps into the symplectic %quotient $\ov{M}$.

\subsection{$G$-Lagrangian branes}\label{subsection51}

Let $(M, \omega, \mu)$ be a Hamiltonian $G$-manifold, satisfying {\bf (H1)--(H3)}. A {\bf $G$-Lagrangian} of $M$ is a smooth embedded, {\it connected} $G$-invariant Lagrangian submanifold $L$ contained in $\mu^{-1}(0)$. $G$-Lagrangians are in one-to-one correspondence with Lagrangians in the symplectic quotient $\ov{M}$.

Let $D_2(M, L) \subset H_2(M, L)$ be the image of the Hurwicz map $\pi_2(M, L) \to H_2(M, L)$. Then $S_2(M)$ acts on $D_2(M, L)$ in the usual way. Moreover, it extends to an $S_2^G(M)$-action as follows. We view $H_2(M, L)$ as a subset of $H_2^G(M, L) = H_2(M_G, L_G)$ since $(M, L)$ is a $G$-equivariant pair, Then every $\beta\in D_2(M, L)$ can be viewed as a homotopy class of maps $u: ({\mb D}, \partial {\mb D}) \to (M_G, L_G)$. Then the map
\begin{align*}
S_2^G(M) \to H_2^G(M) = H_2(M_G) \to H_2(M_G, L_G),
\end{align*}
induces an $S_2^G(M)$-action on $D_2(M, L)$, which is denoted by $B\# \beta$ for $B \in S_2^G(M)$ and $\beta\in D_2(M, L)$. We denote
\begin{align*}
\wt\Gamma:= D_2(M, L)/ N_2^G(M)
\end{align*}
which has the residual $\Gamma$-action (recall $\Gamma = S_2^G(M)/ N_2^G(M)$). Every element $\lb \beta \rb \in \wt\Gamma$ has a well-defined {\bf Maslov index} ${\sf ml} \lb \beta \rb \in {\mb Z}$.

Let $J_0$ be a $G$-invariant almost complex structure on $M$ which satisfies the convexity condition {\bf (H3)}. For each $\lb \beta \rb \in \wt\Gamma$, the moduli space of $J_0$-holomorphic disks representing the class $\lb \beta \rb$ is denoted by ${\mc O} (\lb \beta \rb)$. A key hypothesis in the following construction is

\vspace{0.3cm}

{\bf (H5)} If ${\mc O} (\lb \beta \rb) \neq \emptyset$, then either $\lb \beta \rb = 0$ or ${\sf ml}\lb \beta \rb \geq 2$. Moreover, every stable $J_0$-holomorphic disk is regular.

\vspace{0.3cm}

Notice that this is not the {\it topological} monotonicity condition, but a condition on the special $J_0$.

Choose a $G$-brane structure on $L$ (see Appendix \ref{appendixa}, notation simplified to a choice of $b\in H_1(\ov{L}, \Lambda_0 )$). Choose a Morse-Smale pair $(F, B)$ on $\ov{L}$ with $F$ having a single maximum $y_{\max}$ and choose a generic universal consistent perturbation datum. Then Woodward \cite{Woodward_toric} constructed the {\bf quasimap $A_\infty$ algebra}
\begin{align*}
QA_\infty(L^b) = \big( CQF(L, \Lambda), \{{\mf m}_k^b \}_{k\geq 0}\big)
\end{align*}
of the Lagrangian brane $L^b$, and $\big( CQF(L, \Lambda), {\mf m}_1^b\big)$ is a chain complex of ${\mb Z}_2$-graded $\Lambda$-modules with homology $HQF(L^b, \Lambda)$.

\subsection{Vortex equation and the closed-open moduli spaces}

Now we define a chain level map from the vortex Floer complex $VCF(M; H)$ of a Hamiltonian $H\in {\mc H}_G(M)$ to the quasimap chain complex of a $G$-Lagrangian brane.  We need to consider the symplectic vortex equation over the half cylinder $\Theta_-$ perturbed by a ${\rm Ham}^G(M)$-connection, which interpolating the flow line equation \eqref{equation23} and the $J_0$-holomorphic curve equation. In order to study heaviness of Lagrangians, we need to choose a specific ${\rm Ham}^G(M)$-connection which interpolates between $H$ and a suitable constant, as did in \cite{FOOO_spectral}.

For any $\ov{H} \in {\mc H}(\ov{M})$, denote
\begin{align*}
R_{\ov{H}} = \sup \big\{ \ov{H}(t, x)\ |\ (t, x) \in [0,1]\times \ov{L} \big\}.
\end{align*}
\begin{lemma}{\rm (}\cite[Lemma 18.10]{FOOO_spectral}{\rm )} For any $\ov{H} \in {\mc H}(\ov{M})$ and $\epsilon>0$, there exist an open neighborhood $\ov{U}$ of $\ov{L}$ and a smooth function $\ov{F}_\epsilon = \ov{F}_\epsilon: \Theta_- \times \ov{M} \to {\mb R}$ such that
\begin{align*}
\begin{array}{cll}
\ov{F}_\epsilon (s, t, x) & = \ov{H}_t(x), & \ \forall s< -10,\ t\in S^1,\\[0.2cm]
\ov{F}_\epsilon (s, t, x) & = R_{\ov{H}} + \epsilon, & \ \forall s> -1,\ t\in S^1,\  x \in \ov{U},\\
\displaystyle \frac{\partial \ov{F}_\epsilon}{\partial s}(s, t, x) & \geq 0, &\ \forall s\in {\mb R},\ t\in S^1,\ x\in \ov{M}.
\end{array}
\end{align*}
\end{lemma}
Then we can lift $\ov{F}_\epsilon$ to a $G$-invariant smooth function $F_\epsilon: \Theta_- \times M \to {\mb R}$ (which induces a $G$-invariant lift $H: [0,1]\times M\to {\mb R}$ of $\ov{H}$) such that
\begin{align*}
\begin{array}{cll}
F_\epsilon (s, t, x) & = H_t(x), & \ \forall s< -10,\\[0.2cm]
F_\epsilon (s, t, x) & = R_{\ov{H}} + \epsilon, & \ \forall s> -1, x \in U,\\
\displaystyle \frac{\partial F_\epsilon}{\partial s}(s, t, x) & \geq 0, & \ \forall s\in {\mb R},\ t\in S^1, x\in M.
\end{array}
\end{align*}
Here $U$ is a $G$-invariant neighborhood of $L$.

If $\ov{H}\in {\mc H}^*(\ov{M})$, then we can choose the lift $F_\epsilon$ such that the induced lift $H$ of $\ov{H}$ belongs to a regular pair $(H, J) \in {\mc P}^{reg}_{\ov{H}, \lambda =1}$ (cf. Theorem \ref{thm23} and \cite[Definition 6.5]{Xu_VHF}). Then choose a generic family of almost complex structure $\wt{J} = (J_z) = (J_{s, t})$ parametrized by $z \in \Theta_-$, satisfying:

\begin{itemize}
\item on $(-\infty, -2]\times S^1$, $\wt{J}$ is exponentially close to $J$;

\item for $(s, t) \in \partial \Theta_-$, $J_{s, t} \equiv J_0$.
\end{itemize}

Then consider the $F_\epsilon$-perturbed symplectic vortex equation over $\Theta_-$ with auxiliary data $(\wt{J}, \nu = ds dt)$ and boundary condition $u(\partial \Theta_-) \subset L$. In the standard cylindrical coordinates, it reads
\begin{align}\label{equation51}
\left\{ \begin{array}{ccc} \displaystyle \frac{\partial u}{\partial s} + {\mc X}_\Phi(u) + J_{s,t} \Big( \frac{\partial u}{\partial t} + {\mc X}_\Psi - {\mc Y}_{F_\epsilon} (u) \Big) & = & 0,\\[0.25cm]
\displaystyle \frac{\partial \Psi}{\partial s} - \frac{\partial \Phi}{\partial t} + [\Phi, \Psi] + \mu(u) & = & 0.
\end{array} \right. \ u(\partial \Theta_-) \subset L.
\end{align}

For each finite energy solution ${\mf w}= (u, \Phi, \Psi)$ to \eqref{equation51}, there exists $\wt{x}\in {\rm Zero} \wt{\mc B}_H$ such that up to gauge transformation, $\displaystyle \lim_{z \to -\infty} {\mf w} = \wt{x}$. Therefore, for $\lb {\mf x}\rb\in {\rm Crit} {\mc A}_H$ and $ \lb \beta \rb \in \wt\Gamma$, we can consider the moduli space ${\mc M}( \lb {\mf x} \rb, \lb \beta \rb; F_\epsilon, \wt{J})$, consisting of gauge equivalence classes of solutions $[{\mf w}]$ to \eqref{equation51} which are asymptotic to the orbit of $\lb {\mf x} \rb$ and for which the concatenation of $\lb {\mf x} \rb$ and $[{\mf w}]$ represents $\lb \beta \rb$. A simple calculation similar to the proof of Theorem \ref{thm25} gives the following energy identity.
\begin{prop}\label{prop52}
The energy of a bounded solution ${\mf w} = (u, \Phi, \Psi)$ to \eqref{equation51} which represents an element in ${\mc M}( \lb {\mf x} \rb, \lb \beta\rb; F_\epsilon, \wt{J})$ is defined and given by the following formula.
\begin{align*}
\begin{split}
\mc{YMH}_{F_\epsilon} ({\mf w}) := &\ \Big\| \frac{\partial u}{\partial s} + {\mc X}_{\Phi}(u) \Big\|_{L^2(\Theta_-)}^2 + \Big\| \mu(u) \Big\|_{L^2(\Theta_-)}^2\\
= &\ \int_{\Theta_-} u^* \omega + \int_{\Theta_-} \Big( dF_\epsilon \cdot \frac{\partial u}{\partial s}\Big) ds dt\\
= &\ \omega \lb \beta \rb + {\mc A}_H\lb {\mf x} \rb + (R_{\ov{H}} + \epsilon) - \int_{\Theta_-} \Big( \frac{\partial F_\epsilon}{\partial s} \circ u \Big) ds dt
\end{split}
\end{align*}
\end{prop}

Since \eqref{equation51} has no domain symmetry, it is easy to achieve transversality for the above moduli space. More precisely, we have
\begin{thm}
For a generic family $\wt{J} = (J_{s, t})$, for any $\lb {\mf x} \rb \in {\rm Crit} {\mc A}_H$ and $\lb \beta \rb \in \wt\Gamma$, the moduli space ${\mc M}( \lb {\mf x} \rb, \lb \beta \rb;  F_\epsilon, \wt{J})$ is a smooth manifold of dimension
\begin{align*}
{\rm dim} {\mc M}( \lb {\mf x} \rb, \lb \beta \rb;  F_\epsilon, \wt{J})  = \ov{m} + {\sf ml} \lb \beta\rb + {\sf cz}\lb{\mf x}\rb.
\end{align*}
Moreover, a $G$-equivariant spin structure on $L$ induces an orientation on ${\mc M}( \lb {\mf x} \rb, \lb \beta \rb;  F_\epsilon, \wt{J})$.
\end{thm}

\subsubsection*{The closed-open moduli spaces}

It is a bit lengthy to spell out rigorously the definition of closed-open map, as will be elaborated momentarily.  The reader should keep in mind that intuitively it is only a holomorphic treed quasi-disk (see Appendix \ref{subsectiona2}) with one disk component replaced by a solution of \eqref{equation51}.  Throughout we use conventions and notations from Appendix \ref{appendixa}.

%Connecting the half-cylinders with treed quasidisks (see Appendix %\ref{subsectiona2}) one has the objects responsible for the closed-open map.

\begin{defn}
\begin{enumerate}

\item A {\bf stable co-map} (with auxiliary data $F_\epsilon$, $\wt{J}$, and the perturbation datum on flow lines) consists of a Stasheff tree ${\bf T} = (V, E, {\mf e})$, a treed disk (see Appendix \ref{subsectiona1})
\begin{align*}
{\mc D} = \Big( (D_\alpha)_{\alpha \in V_1},\ (N_\beta)_{\beta\in V_0},\ (I_{e})_{e\in E},\ (z_e)_{e\in E,\ e^+\in V_1} \Big),
\end{align*}
modelled on ${\bf T}$, a distinguished vertex ${\mf c}\in V_1$, and a collection
\begin{align*}
{\bf U}:=\big( {\mf w}, (u_\alpha)_{\alpha \in V_1 \setminus \{{\mf c}\}}; (v_e)_{e\in E} \big)
\end{align*}
where
\begin{itemize}
\item ${\mf w} = (u, \Phi, \Psi)$ is a solution to \eqref{equation51}, called the cylindrical component;

\item for each $\alpha \in V_1 \setminus \{{\mf c}\}$, $u_\alpha: (D_\alpha, \partial D_\alpha) \to ( X, L)$ is a $J_0$-holomorphic disk, called a disk component;

\item For each $e \in E$, $v_e: I_e \to \ov{L}$ is a solution to the perturbed negative flow line equation \eqref{equationa1}, called an edge component.
\end{itemize}
Here we regard the domain ${\mf w}$ as the punctured disk $D_{\mf c}\setminus \{0\}$ equipped with the cylindrical volume form. We require that, if replacing ${\mf w}$ by a nonconstant quasidisk with same evaluations on all special points of $D_{\mf c}$, then ${\bf U}$ becomes a stable holomorphic treed quasidisk (see Appendix \ref{subsectiona2}).

\item there is an isomorphism of their Stasheff trees (as ribbon trees), identifying the distinguished vertex.

\item A disk vertex $\alpha \in V_1 \setminus \{{\mf c}\}$ is called a {\bf ghost disk vertex} for ${\bf U}$ if $u_\alpha$ has zero energy.

\item Two stable co-maps are regarded as equivalent, if
\begin{itemize}

\item their underlying treed disks are the same;

\item the corresponding disk components are isomorphic, i.e., differ by a domain symmetry preserving the special points and a $G$-action;

\item the corresponding edge components are the same perturbed flow line;

\item the cylindrical components are isomorphic, i.e., differ by a gauge transformation on $\Theta_-$.
\end{itemize}
\end{enumerate}
\end{defn}
See Figure \ref{figure1} as an illustration of a stable co-map.

\begin{figure}
\centering
\begin{tikzpicture}[scale=0.6]
\coordinate (C1) at (0,1);

\coordinate (C2) at (0, -1);

\draw[dashed, color=gray] (0,1.5) arc (90:270:0.3 and 0.75); % dashed arc on the cylinder
\draw (0,0) arc (-90:90:0.3 and 0.75); % the un-dashed arc on the cylinder
\draw (0, 1.5) -- (-4, 1.5); % the upper boundary of the cylinder
\draw (0,0) -- (-4, 0); % the lower boundary of the cylinder
\draw (0,0) -- (0, -2); % the vertical edge

\fill[gray] (0, -2.6) circle (0.6);

\draw (0.3, 0.75) -- (2.3, 0.75); % the horizontal edge

\draw[dotted] (2.9, 0.75) circle (0.6); % the ghost disk

\draw (2.9, 1.35) -- (2.9, 2.35);

\fill[gray] (2.9, 2.95) circle (0.6); % the above disk

\draw (0.7, 2.95) -- (2.3, 2.95);
\draw[dotted] (0.7, 2.95) node[fill, circle, inner sep = 2pt]{};

\draw (3.5, 0.75) -- (4.5, 0.75); % the very right edge
\draw[dotted] (4.5, 0.75) node[fill, circle, inner sep = 2pt]{};

\draw[dotted] (1.3, 0.75) node[fill, circle, inner sep = 2pt]{};
\draw[dotted] (0, -1) node[fill, circle, inner sep = 2pt]{};

\end{tikzpicture}
\caption{A stable co-map with one out-going vertex (the node on the very right), one boundary in-coming vertex (the node on the top). The other two nodes are breakings of perturbed gradient lines. The dashed circle indicates a ghost disk component.}
\label{figure1}
\end{figure}
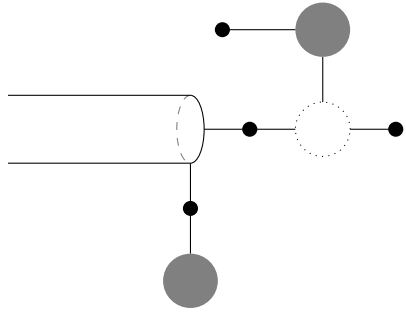

We now restrict our attention to special types of stable co-maps whose underlying Stasheff trees have at most two in-coming semi-infinite edges. In this section we simply call those in-coming semi-infinite edges {\bf boundary inputs} and call the cylindrical components the {\bf cylindrical inputs}; we call the unique out-going semi-infinite edge the {\bf output}.

Here and in the following we will omit the dependence on $F_\epsilon$ and $\wt{J}$ in the notations.

\bdf For each $\lb {\mf x}\rb \in {\rm Crit} {\mc A}_H$, $\lb \beta \rb \in \wt\Gamma$ and $y \in {\rm Crit} F$,  $\mc{CO} \big( \lb {\mf x} \rb, \lb \beta \rb, y\big)$ denotes the moduli space of equivalence classes of stable co-maps ${\bf U}$ of which the underlying isomorphism classes of Stasheff trees have no boundary inputs, such that
\begin{itemize}
\item The cylindrical component of ${\bf U}$ represents an element in ${\mc M}(\lb {\mf x} \rb, \lb \beta' \rb; F_\epsilon, \wt{J})$ for some $\beta' \in \wt\Gamma$ and the disk components have homology class $\beta''\in \wt\Gamma$ such that $\beta = \beta' + \beta''$.

\item The evaluation at the $+\infty$ of the output of ${\bf U}$ is $y$.
\end{itemize}
\edf

An important special case is to consider the \textbf{linear} stable co-maps, whose underlying Stasheff tree has maximal valence equal two. To describe boundaries of 1-dimensional moduli spaces of the form $\mc{CO} ( \lb {\mf x} \rb, \lb \beta \rb , y )$, we also consider

\bdf In the situation as above, denote by $\mc{CO}'( \lb {\mf x} \rb, \lb \beta \rb, y )$ the moduli space of stable co-maps with one boundary input. For $y' \in {\rm Crit} F$, denote by $\mc{CO}'( \lb {\mf x} \rb, \lb \beta \rb, y; y' )$ the subset of elements whose evaluations at the $-\infty$ of the boundary input is $y'$.

\edf

Examples of elements in $\mc{CO} \big( \lb {\mf x} \rb, \lb \beta \rb, y\big)$ are shown in Figure 2.  Removing the shaded disk and the edge connecting to it yields an example of elements in $\mc{CO}'( \lb {\mf x} \rb, \lb \beta \rb, y )$.

\begin{rem}
The topology of $\mc{CO}( \lb {\mf x} \rb, \lb \beta\rb, y)$ or $\mc{CO}'(\lb {\mf x} \rb, \lb \beta\rb, y; y')$ is defined in a similar way as the moduli spaces of perturbed treed holomorphic quasidisks given in \cite{Woodward_toric}. The only difference is the  bubbling of quasidisks on the cylindrical component, which still falls into the realm of standard techniques: on the one hand, no sphere bubbles can appear; on the other hand, by our choice of $\wt{J}$, whose value on the boundary of $\Theta_-$ is equal to $J_0$, the bubbles on the boundary will always be $J_0$-quasidisks.
\end{rem}

By standard argument similar to that of \cite{Woodward_toric_corrected}, with one of the disk components replaced by a cylindrical component, we have the following theorem.
\begin{thm}\label{thm58}
For $\lb {\mf x} \rb \in {\rm Crit} {\mc A}_H$, $y \in {\rm Crit} F$, $\lb \beta \rb \in \wt\Gamma$, for a generic, admissible family of almost complex structures $\wt{J}$ and a generic perturbation datum, the moduli space $\mc{CO}\big( \lb {\mf x} \rb, \lb \beta \rb, y \big)$ is smooth of dimension
\begin{align}\label{equation52}
{\rm dim} \mc{CO} ( \lb {\mf x} \rb, \lb \beta \rb, y \big) = {\sf cz}\lb {\mf x} \rb + {\sf ml}\lb \beta \rb - {\sf ind} y.
\end{align}
The $G$-equivariant spin structure endows consistent orientations on these moduli spaces.

Moreover, when \eqref{equation52} is equal to one, the boundary of the moduli space is the disjoint union of the following three sets:

{\bf (I)} The set of objects broken in the cylindrical part:
\begin{align}\label{equation53}
\partial_+ \mc{CO} \big( \lb {\mf x} \rb, \lb \beta \rb, y\big) = \bigcup_{{\sf cz} \lb {\mf x} \rb - {\sf cz} \lb {\mf x}' \rb = 1}  {\mc M} \big( \lb {\mf x} \rb, \lb {\mf x}'\rb \big) \times \mc{CO} \big( \lb {\mf x}' \rb, \lb \beta \rb, y \big)
\end{align}

{\bf (II)} The set of objects broken in the quasidisk part:
\begin{align}\label{equation54}
\partial_- \mc{CO}\big( \lb {\mf x} \rb, \lb \beta \rb, y\big) = \bigcup_{{\sf ml}\lb \beta_2\rb + {\sf ind} y' - {\sf ind} y = 1 \atop \beta = \beta_1 + \beta_2 } \mc{CO} \big( \lb {\mf x} \rb, \lb \beta_1 \rb, y' \big)  \times {\mc M}\big( y', y; \lb \beta_2  \rb \big)
\end{align}

{\bf (III)} The set of objects where side-bubbling happens (see the middle and right in Figure 2):
\begin{align}\label{equation55}
\partial_0 \mc{CO} \big( \lb {\mf x} \rb, \lb \beta \rb, y\big) = \bigcup_{y' \in {\rm Crit} F} \bigcup_{\beta = \beta_0 + \beta_1} \mc{CO}'(\lb {\mf x} \rb, \lb \beta_0 \rb, y; y') \times \mc{O}( \lb \beta_1\rb, y').
\end{align}
Here ${\mc O}(\lb \beta_1 \rb, y')$ is the subset of $MW_0(L,\beta_1)$ with $y'$ as the output, and ${\mc M}\big( y', y; \lb \beta_2  \rb \big)$ is the subset of $MW_1(L,\beta_2)$ with boundary input equal $y'$ and output $y$ (see Appendix \ref{appendixa} for relevant definitions).
\end{thm}

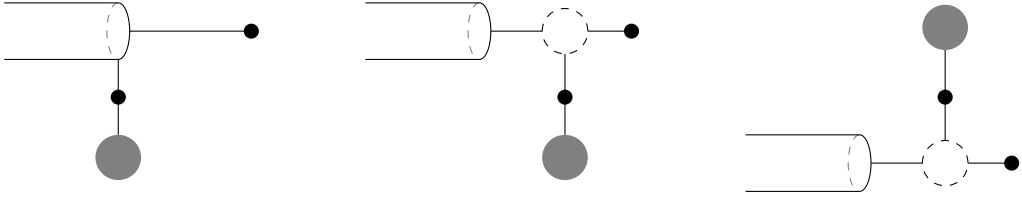
\begin{figure}
\centering
\begin{tikzpicture}[scale=0.5]

% the left
\draw[dashed, color=gray] (0.5,1.5) arc (90:270:0.3 and 0.75); % dashed arc on the cylinder
\draw (0.5,0) arc (-90:90:0.3 and 0.75); % the un-dashed arc on the cylinder
\draw (0.5, 1.5) -- (-2.5, 1.5); % the upper boundary of the cylinder
\draw (0.5,0) -- (-2.5, 0); % the lower boundary of the cylinder
\draw (0.5,0) -- (0.5, -2); % the vertical edge
\fill[gray] (0.5, -2.6) circle (0.6);
\draw (0.8, 0.75) -- (4, 0.75); % the horizontal edge
\draw[dotted] (4, 0.75) node[fill, circle, inner sep = 2pt]{};
\draw[dotted] (0.5, -1) node[fill, circle, inner sep = 2pt]{};

%the middle
\draw[dashed, color=gray] (10,1.5) arc (90:270:0.3 and 0.75); % dashed arc on the cylinder
\draw (10,0) arc (-90:90:0.3 and 0.75); % the un-dashed arc on the cylinder
\draw (10, 1.5) -- (7, 1.5); % the upper boundary of the cylinder
\draw (10,0) -- (7, 0); % the lower boundary of the cylinder
\draw (10.3, 0.75) -- (11.65, 0.75); % the horizontal edge
\draw[dashed] (12.25, 0.75) circle (0.6);
\draw (12.25, 0.15) -- (12.25, -2); % the vertical edge
\draw[dotted] (12.25, -1) node[fill, circle, inner sep = 2pt]{};
\fill[gray] (12.25, -2.6) circle (0.6);
\draw (12.85, 0.75)--(14, 0.75);
\draw[dotted] (14, 0.75) node[fill, circle, inner sep = 2pt]{};

% the right
\draw[dashed, color=gray] (20,-2) arc (90:270:0.3 and 0.75); % dashed arc on the cylinder
\draw (20,-3.5) arc (-90:90:0.3 and 0.75); % the un-dashed arc on the cylinder
\draw (20, -2) -- (17, -2); % the upper boundary of the cylinder
\draw (20,-3.5) -- (17, -3.5); % the lower boundary of the cylinder

\draw (20.3, -2.75) -- (21.65, -2.75); % the horizontal edge
\draw[dashed] (22.25, -2.75) circle (0.6);
\draw (22.25, -2.15) -- (22.25, 0.25);

\fill[gray] (22.25, 0.85) circle (0.6);
\draw (22.85, -2.75)--(24, -2.75);
\draw[dotted] (24, -2.75) node[fill, circle, inner sep = 2pt]{};

\draw[dotted] (22.25, -1) node[fill, circle, inner sep = 2pt]{};

\end{tikzpicture}
\caption{Stable co-maps with side bubbles (from the left to right, attached to the cylinder, attached to a ghost disk from the right, and attached to a ghost disk from the left, respectively). When counting defines a chain map, the configuration on the left doesn't exist and the configurations in the middle and on the right cancel in pairs.}
\label{figure2}
\end{figure}

We have more precise descriptions of the moduli spaces when assuming {\bf (H5)}.
\begin{lemma}\label{lemma59}
Assume $(L, J_0)$ satisfies {\bf (H5)}. Then when $\mc{CO}\big( \lb {\mf x} \rb, \lb \beta \rb, y\big)$ or ${\mc M}(y', y; \lb \beta \rb)$ is zero-dimensional, for each of its element, its underlying treed disk has neither a disk vertex with only one special point nor a nodal vertex. In particular, elements in zero-dimensional $\mc{CO}\big( \lb {\mf x} \rb, \lb \beta \rb, y\big)$ or zero-dimensional ${\mc M}(y', y; \lb \beta \rb)$ must be linear stable co-maps.
\end{lemma}

\begin{proof}
Let ${\bf U}$ be a stable co-map representing an element in $\mc{CO}( \lb {\mf x} \rb, \lb \beta \rb, y)$ (or ${\mc M}(y', y; \lb \beta \rb)$, which has no difference in this proof) whose underlying treed disk is ${\mc D}$. It is easy to see that ${\mc D}$ doesn't contain any interior nodal vertex, otherwise by gluing broken gradient lines, one can construct a 1-parameter family of holomorphic treed quasidisks.

Now we prove that ${\mc D}$ has no disk vertex which has only one special point. Suppose $\alpha \in V_1$ is a disk vertex with only one edge $\alpha \succ \alpha'$ that connects to $\alpha$. Let ${\mc D}'$ be the treed disk obtained by forgetting the vertex $\alpha$ and the edge $\alpha \succ \alpha'$ and ${\bf U}'$ the object by forgetting $u_\alpha$ and $v_{\alpha \succ \alpha'}$. $u_\alpha$ is nonconstant by stability and $\alpha'$ is not nodal as we have just seen. Then $\alpha'$ is a disk vertex. If $u_{\alpha'}$ is nonconstant, then ${\bf U'}$ is a stable co-map of index $- {\sf ml}(u_{\alpha})$ with underlying treed disk ${\mc D}'$, although the perturbation data may be different. Since there is no nonconstant quasidisk with nonpositive Maslov index, the index must be negative. The transversality supported by Theorem \ref{thm58} implies that there is no such object.

If $u_{\alpha'}$ is constant, then removing the vertex $\alpha$ and the edge $\alpha \succ \alpha'$, the remaining object is a (possibly) unstable co-map, of index at most $1 - {\sf ml}(u_{\alpha})<0$, since minimal Maslov index is two. For the same reason as above, there cannot be such object.
\end{proof}

\subsection{The closed-open map}

The counting of zero-dimensional moduli spaces defines a $\Lambda$-linear map ${\mf i}_{F_\epsilon}: VCF (M; H) \to CQF(L)$ given by
\begin{align}\label{equation56}
{\mf i}_{F_\epsilon} \lb {\mf x} \rb = \sum_{\beta \in \wt\Gamma\atop y \in {\rm Crit} F} \# \mc{CO} ( \lb {\mf x} \rb, \lb \beta \rb, y ) q^{- A(\lb \beta \rb)} e^{\langle b, \partial\lb \beta \rb\rangle} \langle y\rangle.
\end{align}
Here $\# \mc{CO}( \lb {\mf x} \rb, \lb \beta \rb, y )$ is the algebraic counting. By Lemma \ref{lemma59}, under the assumption {\bf (H5)}, one only counts linear stable co-maps.

\begin{prop}\label{prop510}
${\mf i}_{F_\epsilon}$ is a chain map, i.e., ${\mf i}_{F_\epsilon} \circ \delta_{J} = {\mf m}_1^b \circ {\mf i}_{F_\epsilon}$.
\end{prop}

By Theorem \ref{thm58}, the proposition will be proved if we can show that the contribution of \eqref{equation55} is zero. However, it is not easy to evaluate this contribution due to our perturbation scheme. So we bypass it by considering one-dimensional moduli space of linear stable co-maps. Indeed, fix $\lb {\mf x} \rb \in {\rm Crit} {\mc A}_H$, $\lb \beta \rb \in \wt\Gamma$ and $y\in {\rm Crit} F$. Consider the subset
\begin{align*}
\mc{BO}\big( \lb {\mf x} \rb, \lb \beta \rb, y \big) \subset \mc{CO} \big( \lb {\mf x} \rb, \lb \beta \rb, y \big)
\end{align*}
consisting of linear stable co-maps. It has the induced topology from $\mc{CO}\big(\lb {\mf x} \rb, \lb \beta \rb, y \big)$. By Lemma \ref{lemma59}, when the moduli spaces are zero-dimensional, $\mc{BO}$ is the same as $\mc{CO}$. Moreover, we have the following description of boundaries of 1-dimensional $\mc{BO}\big( \lb {\mf x} \rb, \lb \beta \rb, y \big)$.

\begin{figure}
\centering
\begin{tikzpicture}[scale=0.5]

%the middle
\draw[dashed, color=gray] (10,1.5) arc (90:270:0.3 and 0.75); % dashed arc on the cylinder
\draw (10,0) arc (-90:90:0.3 and 0.75); % the un-dashed arc on the cylinder
\draw (10, 1.5) -- (7, 1.5); % the upper boundary of the cylinder
\draw (10,0) -- (7, 0); % the lower boundary of the cylinder
\draw (10.3, 0.75) -- (11.65, 0.75); % the horizontal edge
\draw[dashed] (12.25, 0.75) circle (0.6);
\fill[gray] (12.25, -0.45) circle (0.6);
\draw (12.85, 0.75)--(14, 0.75);
\draw[dotted] (14, 0.75) node[fill, circle, inner sep = 2pt]{};

% the right
\draw[dashed, color=gray] (20,0) arc (90:270:0.3 and 0.75); % dashed arc on the cylinder
\draw (20,-1.5) arc (-90:90:0.3 and 0.75); % the un-dashed arc on the cylinder
\draw (20, 0) -- (17, 0); % the upper boundary of the cylinder
\draw (20,-1.5) -- (17, -1.5); % the lower boundary of the cylinder

\draw (20.3, -0.75) -- (21.65, -0.75); % the horizontal edge
\draw[dashed] (22.25, -0.75) circle (0.6);

\fill[gray] (22.25, 0.45) circle (0.6);
\draw (22.85, -0.75)--(24, -0.75);
\draw[dotted] (24, -0.75) node[fill, circle, inner sep = 2pt]{};
\end{tikzpicture}
\caption{}
\label{figure3}
\end{figure}

\begin{lemma}\label{lemma511}
Suppose {\bf (H5)} holds. When \eqref{equation52} is equal to one, $\mc{BO}\big( \lb {\mf x} \rb, \lb \beta \rb, y \big)$ is a one-dimensional manifold with boundary and its boundary is the disjoint union of \eqref{equation53}, \eqref{equation54} and a finite set $\partial_0 \mc{BO}\big( \lb {\mf x} \rb, \lb \beta \rb, y \big)\subset \mc{CO}\big( \lb {\mf x}\rb, \lb \beta \rb, y\big)$ consisting of stable co-maps with exactly one side bubbles. Moreover, the side bubbles are holomorphic quasidisks of Maslov index two and are connected via a length-zero edge to a ghost disk vertex (see Figure \ref{figure3}).
\end{lemma}

\begin{proof}
By index calculation and {\bf (H5)}, it is easy to see that in each stable co-map in $\partial_0 \mc{BO}\big( \lb {\mf x}\rb, \lb \beta \rb, y \big)$, there can only be one side bubble which is of Maslov index two. Moreover, if it is not adjacent to a ghost disk vertex, then by removing the bubble component, we obtain another stable co-map whose index is negative, which cannot exist.
\end{proof}

\begin{proof}[Proof of Proposition \ref{prop510}] By the above lemma, the difference ${\mf i}_{F_\epsilon} \circ \delta_{J} - {\mf m}_1^b \circ {\mf i}_{F_\epsilon}$ is given by the counting of $\partial_0 \mc{BO}\big( \lb {\mf x} \rb, \lb \beta \rb, y\big)$. Indeed, there are two types of such configurations, those with bubbles on the ``left'' and those with bubbles on the ``right'' (see Figure \ref{figure3} again). The proposition will be proved if we show that these configurations cancel. Indeed, if we flip the bubble from top to bottom (or backward), and precompose a proper automorphism of the domain of the bubble, it doesn't change the perturbation data on the edges of positive lengths (because then we obtain a Stasheff tree isomorphic to the one before flipping, and the perturbation datum only depends on the isomorphism class of the underlying trees). Therefore, there is a natural one-to-one correspondence between the two types of configurations in the boundary of $\mc{BO}\big( \lb {\mf x} \rb, \lb \beta \rb, y \big)$. Lastly, the contribution from configurations with side bubbles is indeed zero due to the opposite orientations.
\end{proof}

Hence, the map ${\mf i}_{F_\epsilon}$ induces a map in homology
\begin{align*}
\mbms{co}: VHF(M) \to HQF(L^b).
\end{align*}
We call it the {\bf closed-open} map from vortex Floer homology to the quasi-map Lagrangian Floer homology. It is independent of the choice of the auxiliary data $F_\epsilon$ (in particular, independent of $\epsilon$) and $\wt{J}$. This independence can be proved similarly as Proposition \ref{prop510} as well as Theorem \ref{thm512} below, where a homotopy between two sets of auxiliary data will produce a chain homotopy.

We prove the following crucial property of the map $\mbms{co}$ in the next subsection.
\begin{thm}\label{thm512}
Let $L \subset M$ be a $G$-Lagrangian satisfying {\bf (H5)} with a $G$-equivariant brane structure $L^b$. Then $\mbms{co}(\mbms{1}_{\mbms{H}}) = \mbms{1}_{L^b} \in HQF(L^b, \Lambda)$, where ${\mbms 1}_{L^b}$ is the class represented by the chain $\langle y_{\max} \rangle$.
\end{thm}

\subsection{Proof of Theorem \ref{thm512}}\label{subsection54}

The proof is based on a gluing construction.  The actual gluing follows the standard procedure, hence the technical part described in the following is the appropriate family of perturbation data used for the gluing.

\vspace{0.5cm}
\noindent$\bullet$ \textbf{The choice of parametrization and the volume form}
\vspace{0.5cm}

Consider a smooth 1-parameter family of volume forms $\{\nu_\chi\}_{\chi\in [-1,+\infty)}$ on the unit disk ${\mb D}\subset {\mb C}$ such that
\begin{itemize}
\item $\nu_0$ is the standard volume form, $\nu_{-1}=0$;

\item For $\chi\geq 0$, $\nu_\chi$ stretches the length of the cylindrical part;

\item For $\chi \leq 0$, $\nu_\chi$ is a rescale of $\nu_0$ by a factor $<1$;

\item For $\chi \geq 0$, the metric $g_\chi$ induced from $\nu_\chi$ on ${\mb D} \setminus B( e^{-2 \chi})$ is isometric via the exponential map to the cylinder $[- 2 \chi, 0] \times S^1$ with the standard metric,

\item The area of $B(e^{-2\chi})$ is uniformly bounded.
\end{itemize}
The family of volume forms can be illustrated in Figure \ref{figure4}.

For each finite $\chi$, let ${\mb D}_\chi$ be the corresponding disk with the volume form $\nu_\chi$. Denote $C_\chi = B(e^{-\chi})$ and $\Theta_\chi = {\mb D}_\chi \setminus C_\chi$. Then as $\chi$ approaches to infinity, $(C_\chi, \nu_\chi|_{C_\chi})$ converges to ${\mb C}$ equipped with a cylindrical volume form, meaning there is a family of isometry $\vp_\chi:(C_\chi, \nu_\chi|_{C_\chi})\rightarrow B(r_\chi)\subset {\mb C}$ with $r_\chi\rightarrow\infty$, respecting the cylindrical ends.  $(\Theta_\chi, \nu_\chi|_{\Theta_\chi})$ approaches to $\Theta_-$ similarly. We regard ${\mb D}_\infty$ as the disjoint union of a copy of ${\mb C}$ and $\Theta_-$.

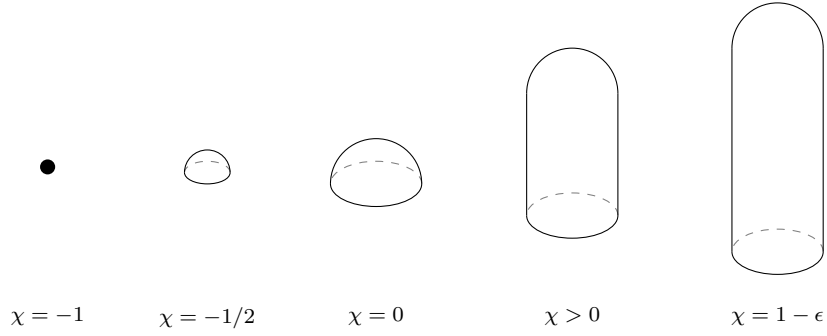
\begin{figure}
\centering

\begin{tikzpicture}[scale=0.6]

\draw[dotted] (-6, -0.125) node[fill, circle, inner sep = 2pt]{}; % \chi = -1

\node[below] at (-6, -3) {\scriptsize $\chi = -1$};

\draw[dashed, color=gray] (-2, -0.25) arc (0:180:0.5 and 0.25); % \chi = -0.5
\draw (-3, -0.25) arc (180:360:0.5 and 0.25);
\draw (-2, -0.25) arc (0: 180: 0.5 and 0.5);

\node[below] at (-2.5, -3) {\scriptsize $\chi = -1/2$};

\draw[dashed, color=gray] (2.2, -0.5) arc (0:180:1 and 0.5); % \chi = 0
\draw (0.2, -0.5) arc (180:360:1 and 0.5);
\draw (2.2, -0.5) arc (0: 180:1 and 1);

\node[below] at (1.2, -3) {\scriptsize $\chi = 0$};

\draw[dashed, color=gray] (6.5, -1.2) arc (0:180:1 and 0.5); % \chi = 0.5
\draw (4.5, -1.2) arc (180:360:1 and 0.5);
\draw (6.5, -1.2)--(6.5, 1.5);
\draw (4.5, -1.2)--(4.5, 1.5);
\draw (6.5, 1.5) arc (0:180:1 and 1);

\node[below] at (5.5, -3) {\scriptsize $\chi > 0$};

\draw[dashed, color=gray] (11, -2) arc (0:180:1 and 0.5); % \chi
\draw (9, -2) arc (180:360:1 and 0.5);
\draw (11, -2)--(11, 2.5);
\draw (9, -2)--(9, 2.5);
\draw (11, 2.5) arc (0:180:1 and 1);

\node[below] at (10, -3) {\scriptsize $\chi = 1-\epsilon$};

\end{tikzpicture}
\caption{The family of volume forms $\nu_\chi$ on the unit disk.}
\label{figure4}
\end{figure}

Choose a cut-off function $\delta(\chi)$ such that
\begin{align*}
\delta(\chi) = \left\{ \begin{array}{cc} 0,&\ -1 \leq \chi \leq 1,\\
                                               1,&\ 2 \leq \chi < +\infty.
\end{array} \right.
\end{align*}
Also choose a family of cut-off functions $\{\rho_\chi\}_{\chi \geq \chi_0}$ on the disk such that
\begin{align*}
\rho_\chi(s, t) = \left\{ \begin{array}{cc} 0,&\ - \infty < s \leq - 2 \chi,\\
                                             1,&\ - 2\chi+1 \leq s \leq 0.
																						 \end{array}
																						 \right.
																						 \end{align*}
where $(s, t)$ is the cylindrical coordinate away from $0\in {\mb D}_\chi$. We assume that as $\chi \to +\infty$, $\rho_\chi\circ\vp_{\chi}^{-1}$ converges uniformly with all derivatives on compact subsets of ${\mb C}$, to a cut-off function $\rho_\infty$.

\vspace{0.5cm}
\noindent$\bullet$ {\bf The choice of Hamiltonian perturbation}
\vspace{0.5cm}

%We will consider the perturbed symplectic vortex equation over the unit disk %with respect to the volume form $\nu_\chi$.
The Hamiltonian perturbation is given by a smooth family of domain-dependent Hamiltonians $H^\chi_\epsilon = \delta(\chi) \rho_\chi(s) F_\epsilon$, where $F_\epsilon$ is the function used in \eqref{equation51}, which interpolates between $H\in {\mc H}_G(M)$ and a constant. Simply speaking, we turn on the perturbation after the domain is long enough, on a longer and longer cylindrical region. Denote by $\nabla' = - \rho_\infty(s, t) H_t dt $ which is a ${\rm Ham}^G(M)$-connection on ${\mb C}$.

\vspace{0.5cm}
\noindent$\bullet$ {\bf The choice of almost complex structures}
\vspace{0.5cm}

We need to choose a nice family of almost complex structures. First choose $J = (J_t)_{t\in S^1}$ such that $(H, J)$ is regular pair; so all vortex Floer trajectories are assumed regular. Next, choose $J' = (J_z')_{z\in {\mb C}}$ which is exponentially close to $J$ on the cylindrical end of ${\mb C}$; so all moduli spaces ${\mc C}(\lb {\mf x} \rb)$ of $\nabla'$-perturbed vortex caps (i.e., solutions to \eqref{equation218} with auxiliary data $\nabla$, $J'$ and the cylindrical volume form on ${\mb C}$) are regular. Then we can choose a family $\wt{J}^\chi = (J^\chi_z)$ parametrized by $\chi \in [-1, +\infty)$ and $z \in {\mb D}_\chi$ such that
\begin{itemize}
\item $J^{-1}_z \equiv J_0$;

\item For each $\chi$, $J^\chi_z = J_0$ for $z \in \partial {\mb D}_\chi$;

\item For some large $\chi_0$ (which is allowed to vary), for each $\chi \geq \chi_0$, $J^\chi$ is obtained by a gluing construction from $J$ and $J'$.
\end{itemize}

\vspace{0.5cm}
\noindent$\bullet$ {\bf The equation and the moduli spaces}
\vspace{0.5cm}

Now consider the family of perturbed symplectic vortex equation reads
\begin{align}\label{equation57}
\ov\partial_A u -  (Y_{H^\chi_\epsilon} dt )^{0,1}(u) = 0,\ F_A + \mu(u) \nu_\chi = 0,\ u(\partial {\mb D}_\chi) \subset L.
\end{align}
Here $\ov\partial_A u$ and $(Y_{H^\chi_\epsilon} dt)^{0,1}$ are both defined with respect to $J_z^\chi$.

Each solution to \eqref{equation57} represents a class $\lb \beta \rb \in \wt\Gamma$. Let $y \in {\rm Crit} F\in \ov{L}$. In the definition of stable co-maps, if we replace the equation on the cylindrical component $D_{\mf c}$ by \eqref{equation57}, then we obtain a family of moduli spaces $\mc{CO}_\chi( \lb \beta \rb, y)$ for each disk class $\lb \beta \rb \in \wt\Gamma$. In particular, each element of $\mc{CO}_{-1}( \lb \beta \rb, y)$ is actually an isomorphism class of holomorphic treed quasidisks whose underlying Stasheff tree lies in ${\mc T}_0$, and one only quotients out the $G$-action on $D_{\mf c}$ but not domain automorphisms.

Moreover, one can define a natural topology on the union
\begin{align*}
\bigcup_{\chi\geq -1} \mc{CO}_\chi(\lb \beta \rb, y).
\end{align*}
Since the length of the cylinder tends to infinity as $\chi\to +\infty$, we can compactify the union by adding broken objects. The compactification is denoted by $\wt{\mc{CO}}( \lb \beta \rb, y )$.

Similar to the situation of defining the closed-open map, one need to consider moduli spaces of ``linear'' objects with notations given by replacing $\mc{CO}$ by $\mc{BO}$. By choosing generic perturbation datum, one can make the compactified space a smooth oriented manifold with boundary when the virtual dimension is one or zero. Then one has the following theorem (one can compare with Theorem \ref{thm58}, Lemma \ref{lemma59} and Lemma \ref{lemma511}).

\begin{prop}\label{p:CO}
For a generic perturbation datum and a generic family of almost complex structure $\wt{J}^\chi$, for every disk class $\lb \beta \rb \in \wt\Gamma$ and $y \in {\rm Crit} F$ such that $\ov{m}+ {\sf ml} \lb \beta \rb - {\sf ind} y = 0$, $\wt{\mc{BO}}( \lb \beta \rb, y)$ is a smooth, compact, oriented, 1-dimensional manifold with boundary. Moreover, its boundary is the disjoint union of the following components.

{\bf (I)} The $\chi = +\infty$ side boundary, which is
\begin{align}\label{equation58}
\partial_1 \wt{\mc{BO}}(\lb \beta \rb, y) = \bigcup_{\lb {\mf x} \rb} \mc{CO}(\lb {\mf x} \rb, \lb \beta \rb, y; F_\epsilon, J) \times \mc{C}(\lb {\mf x} \rb; \nabla', J').
\end{align}

{\bf (II)} The $\chi = -1$ side boundary, which is
\begin{align}\label{equation59}
\partial_{-1} \wt{\mc{BO}} ( \lb \beta \rb, y) = \left\{ \begin{array}{cc} \emptyset,\ & y \neq y_{\max}\ {\rm or}\ \lb \beta \rb \neq 0,\\
                                                                            \{[{\bf V}_0]\},\ & y = y_{\max}\ {\rm and}\ \lb \beta \rb = 0.
\end{array} \right.
\end{align}
Here $[{\bf V}_0]$ is the isomorphism class of the constant quasidisk with a perturbed negative gradient flow line connecting the constant disk to $y_{\max}$.

{\bf (III)} The interior boundary component $\partial_{\mf{int}} \wt{\mc{BO}}( \lb \beta \rb, y)$ consisting of two types of configurations. The first type of configurations are those containing broken graident lines, i.e.,
\begin{align*}
\bigcup_{\chi > -1} \bigcup_{y', \lb \beta'\rb} {\mc{BO}}_\chi( \lb \beta - \beta' \rb, y')\times {\mc M}( y', y ; \lb \beta' \rb);
\end{align*}
the other type of configurations are those with exactly one side bubbles which are attached to a ghost disk vertex via a length-zero edge.
\end{prop}

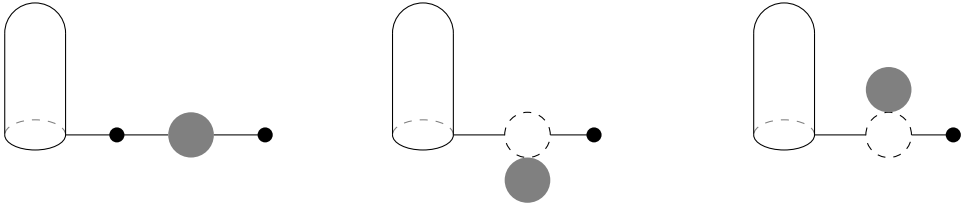
\begin{figure}\label{figure5}
\centering
\begin{tikzpicture}[scale= 0.5]

%the left

\draw[dashed, color=gray] (0.3, 0.75) arc (0:180:0.8 and 0.4);
\draw (-1.3, 0.75) arc (180:360:0.8 and 0.4);
\draw (0.3, 0.75)--(0.3, 3.45);
\draw (-1.3, 0.75)--(-1.3, 3.45);
\draw (0.3, 3.45) arc (0:180:0.8 and 0.8);

\draw (0.3, 0.75) -- (1.65, 0.75) node[fill, circle, inner sep = 2pt]{}; % the horizontal edge
\draw (1.65, 0.75) -- (3, 0.75);
\fill[gray] (3.6, 0.75) circle (0.6);
\draw (4.2, 0.75)--(5.55, 0.75) node[fill, circle, inner sep = 2pt]{};

%the middle

\draw[dashed, color=gray] (10.5, 0.75) arc (0:180:0.8 and 0.4);
\draw (8.9, 0.75) arc (180:360:0.8 and 0.4);
\draw (10.5, 0.75)--(10.5, 3.45);
\draw (8.9, 0.75)--(8.9, 3.45);
\draw (10.5, 3.45) arc (0:180:0.8 and 0.8);

\draw (10.5, 0.75) -- (11.85, 0.75); % the horizontal edge
\draw[dashed] (12.45, 0.75) circle (0.6);
\fill[gray] (12.45, -0.45) circle (0.6);
\draw (13.05, 0.75)--(14.2, 0.75);
\draw[dotted] (14.2, 0.75) node[fill, circle, inner sep = 2pt]{};

% the right
\draw[dashed, color=gray] (20, 0.75) arc (0:180:0.8 and 0.4);
\draw (18.4, 0.75) arc (180:360:0.8 and 0.4);
\draw (20, 0.75)--(20, 3.45);
\draw (18.4, 0.75)--(18.4, 3.45);
\draw (20, 3.45) arc (0:180:0.8 and 0.8);

\draw (20, 0.75) -- (21.35, 0.75); % the horizontal edge
\draw[dashed] (21.95, 0.75) circle (0.6);
\fill[gray] (21.95, 1.95) circle (0.6);

\draw (22.55, 0.75)--(23.65, 0.75) node[fill, circle, inner sep = 2pt]{};

\end{tikzpicture}
\caption{The boundary component $\partial_{\mf int} \wt{\mc{BO}}( \lb \beta \rb; y)$.}
\end{figure}

\begin{proof}
The proof is similar to that of Theorem \ref{thm58}, which essentially only differs in the description of the $\chi=-1$ side boundary.

Since there cannot be sphere bubbles, compactification is given by adding broken objects as well as side bubbles. On the $\chi = +\infty$ side, we only have breakings on the cylindrical component. A gluing construction based on the careful choices of the auxiliary data we made gives \eqref{equation58}. For $\chi \in (-1, +\infty)$, there is no breaking on the cylindrical component. Instead, the perturbed flow line in $\ov{L}$ may break, and there may be $J_0$-holomorphic disk bubbles. These objects give the boundary components in {\bf (III)}.

Finally, suppose we have a sequence $\chi_i \to -1^+$ and a sequence of elements ${\bf U}^{(i)}$ with bounded energy, then a subsequence (still indexed by $i$) will converge in the following sense: the cylindrical component will converge (up to gauge transformation) to a flat connection (because the volume form shrinks to zero) and a quasidisk without reparametrizing the domain; all the other components or trajectories converge as perturbed holomorphic treed quasidisks. However, by dimensional counting and the hypothesis {\bf (H5)}, there can only be the trivial quasi-disk when $y = y_{\max}$. Therefore, the $\chi = -1$ side boundary is of the form as described in {\bf (II)}.

To achieve transversality, notice that for $\chi = -1$, transversality was proved in \cite{Woodward_toric}; once $\chi$ is greater than $-1$, we are allowed to change $J$ to be dependent on $z$, which is enough to guarantee the transversality near and away from $\chi = -1$.
\end{proof}

\begin{proof}[Proof of Theorem \ref{thm512}]
It is similar to the proof of Proposition \ref{prop510}. First note that, by the index formula and transversality, each component in \eqref{equation58} on the right-hand-side is nonempty only when
\begin{align*}
{\sf cz} \lb {\mf x} \rb + {\sf ml}\lb \beta \rb - {\sf ind} y \geq 0,\ \ov{m} - {\sf cz} \lb {\mf x} \rb \geq 0
\end{align*}
which implies that ${\sf cz} \lb {\mf x} \rb = \ov{m}$. Remember that the representative of $\mbms{1}_{\mbms H}$ in $VCF(M; H)$ constructed in Subsection \ref{subsection27} is a linear combination of $\lb {\mf x} \rb$'s with ${\sf cz} \lb {\mf x} \rb = \ov{m}$. So the algebraic count of $\partial_1 \wt{\mc{BO}}(\lb \beta \rb, y)$ for all $\lb \beta \rb$ and all $y$ gives the chain level representative of $\mbms{co}(\mbms{1}_{\mbms H})$. More precisely, the image of $\mbms{1}_{\mbms H}$ is represented by
\begin{align*}
\sum_{\lb \beta \rb \in \wt\Gamma \atop y\in {\rm Crit} F} \Big( \# \partial_1 \wt{\mc{BO}}(\lb \beta \rb, y)\Big) q^{-\omega\lb \beta \rb} \langle y \rangle\in CQF(L).
\end{align*}
Hence by Proposition \ref{p:CO}, this chain is equal to the contribution of the other two parts of each of these moduli spaces.

The contribution from $\partial_{\mf{int}}\wt{\mc{BO}}(\lb \beta\rb, y)$ consists of two parts, as shown in Figure \ref{figure5}. The first part gives rise to a boundary in $CQF(L)$, hence we can ignore it; the second part contributes zero by the same cancellation argument as in the proof of Proposition \ref{prop510}. Therefore, the contribution of $\partial_1\wt{\mc{BO}}(\lb\beta\rb, y)$ is cobordant to $\partial_{-1} \wt{\mc{BO}}(\lb \beta\rb, y)$, which implies
\begin{align*}
{\mf i}_{F_\epsilon}({\bm 1}_H) = \sum_{\lb \beta \rb \in \wt\Gamma \atop y\in {\rm Crit} F} \Big( \# \partial_1 \wt{\mc{BO}}(\lb \beta \rb, y)\Big) q^{-\omega\lb \beta \rb} \langle y \rangle = \langle y_{\max}\rangle.
\end{align*}
By passing to homology this completes the proof of Theorem \ref{thm512}.
\end{proof}

\section{Application of the spectral invariants and the closed-open map}\label{s:app}

\subsection{A weak Arnold conjecture}

As the first application, we give the proof of Corollary \ref{c:weakArnold} by the basic properties of closed-open maps proved in Section \ref{s:COMap}.
% implies the following version of weak Arnold conjecture.

\begin{cor}{\rm (}{\bf Weak Arnold Conjecture}{\rm )} For any Hamiltonian $G$-manifold $(X, \omega, \mu)$ satisfying {\bf (H1)}--{\bf (H3)} with at least one $G$-Lagrangian whose quasimap Floer homology is defined and nonzero, any generic 1-periodic Hamiltonian $\ov{H}$ on $\ov{X}$ has at least one 1-periodic orbit of $\ov{H}$.
\end{cor}

\bpf An immediate implication of Theorem \ref{thm512} is that the nontriviality of quasimap Floer homology of any $G$-Lagrangian implies the nontriviality of $VHF(X)$.
\epf

Combining with Woodward's result \cite[Theorem 6.6]{Woodward_toric_corrected} and the fact that the Hori-Vafa superpotential of any toric moment polytope has at least one critical point (see Appendix \ref{subsectiona6}), Corollary \ref{c:weakArnold} follows.

\subsection{Heavy Lagrangians in symplectic quotients}

In this subsection, we obtain the following analogue to \cite[Theorem 18.8]{FOOO_spectral}.

\begin{thm}\label{thm61}
Let $(M, \omega, \mu)$ be a Hamiltonian $G$-manifold satisfying {\bf (H1)}--{\bf (H3)} and $L \subset M$ be a $G$-Lagrangian satisfying {\bf (H5)} such that ${\mf J}$ of {\bf (H3)} and $J_0$ of {\bf (H5)} coincide outside a compact subset of $M$. Let $b\in H^1(\ov{L}, \Lambda)$ be an equivariant brane structure on $L$. Let $\mbms{e} \in VHF(M)$ be an idempotent element. If $\mbms{co} (\mbms{e}) \neq 0\in HQF(L^b)$, then $\ov{L}$ is $\upzeta_{\mbms e}$-heavy and $\upmu_{\mbms e}$-heavy.
\end{thm}

A particularly interesting instance of Theorem \ref{thm61} is given by toric Lagrangian fibers. We again refer the readers to Appendix \ref{subsectiona6} for relevant notations.

\begin{thm}{\rm (}{\bf Heaviness of Hori-Vafa critical fibres}{\rm )}
If $W_\lambda^G$ has a critical point $b\in H^1(\ov{L}_\lambda, \Lambda_0)$, then $\mbms{1}_{\mbms{H}} \neq 0 \in VHF(M)$ and $\ov{L}_\lambda$ is $\upzeta$-heavy and $\upmu$-heavy.
\end{thm}

\begin{proof}
By Theorem \ref{thma8}, the condition that $W_\lambda^G$ has a critical point implies that $HQF(L_\lambda^b) \neq 0$, hence ${\bm 1}_{L_\lambda^b} \neq 0$ since it is the identity. By Theorem \ref{thm512}, $\mbms{1}_{\mbms{H}} \neq 0$ and hence induces a prequasimorphism $\upmu$ and a partial quasi-state $\upzeta$. By Theorem \ref{thm61}, $\ov{L}_\lambda$ is $\upzeta$-heavy in $\ov{M}$.
\end{proof}

Now we prove Theorem \ref{thm61} following similar procedure as \cite{FOOO_spectral}. The valuation ${\mf v}_q: \Lambda \to {\mb R}$ (see \eqref{equation30}) naturally induces a filtration on $CQF(L)$, making it a filtered Floer-Novikov chain complex. More precisely, for $\tau\in {\mb R}$, denote
\begin{align*}
CQF(L)_\tau = \big\{ {\mf Y} \in CQF(L)\ |\ {\mf v}_q({\mf Y}) < \tau \big\}.
\end{align*}
Then by the definition of ${\mf m}_1^b$ \eqref{equationa2}, ${\mf m}_1^b ( CQF(L)_\tau ) \subset CQF(L)_\tau$. Therefore, we have the filtered quasimap Floer homology $HQF(L^b)_\tau$ with maps
\begin{align*}
\iota: HQF(L^b)_\tau \to HQF(L^b)_{\tau'},\ \forall \tau< \tau'.
\end{align*}
It allows us to define a similar quantity as the spectral numbers $c(a, H)$: for every $x\in HQF(L^b)$, define
\begin{align*}
c_{L^b} (x) = - \sup \Big\{ \tau\ |\ x \in {\rm Im} \big(\iota: HQF(L^b)_\tau \to HQF(L^b) \big)\Big\}
\end{align*}
By Usher's abstract formalism (see Subsection \ref{subsection31} and \cite[Theorem 1.3, Theorem 1.4]{Ush08}), $c_{L^b}(x)$ is finite for all nonzero $x$.

The closed and open spectral numbers are related by the closed-open map (recall \eqref{equation56}).
\begin{prop}\label{prop64}{\rm (}cf. \cite[Proposition 18.19]{FOOO_spectral}{\rm )}
For every $H \in {\mc H}_G^*(M)$, every $\tau\in {\mb R}$ and every $\epsilon>0$, we have
\begin{align}\label{equation63}
{\mf i}_{F_\epsilon} \big( VCF(M; H)_\tau \big) \subset CQF(L)_{\tau + R_{\ov{H}} + \epsilon}
\end{align}
\end{prop}

\begin{proof}
Since ${\mf i}_{F_\epsilon}$ is $\Lambda$-linear, we only have to prove that for each $\lb {\mf x} \rb \in {\rm Crit} {\mc A}_H$ with ${\mc A}_H\lb {\mb x} \rb \leq \tau$, ${\mf i}_{F_\epsilon}(\lb {\mf x} \rb) \in CQF(L)_{\tau + \ov{R}_H + \epsilon}$. If ${\mf i}_{F_\epsilon}(\lb {\mf x} \rb) = 0$, then \eqref{equation63} is automatically true. Suppose it is not the case. Then by definition
\begin{align*}
{\mf i}_{F_\epsilon} (\lb {\mf x} \rb) = \sum_{y\in {\rm Crit} F\atop \lb \beta \rb \in \wt\Gamma} \# \Big(\mc{CO}( \lb {\mf x} \rb, \lb \beta \rb, y_i; F_\epsilon, \wt{J})\Big) q^{- A\lb \beta \rb} \langle y \rangle\neq 0.
\end{align*}
Let ${\bf U}$ be a stable co-map representing an element in $\mc{CO}( \lb {\mf x} \rb, \lb \beta \rb, y_i; F_\epsilon, \wt{J})$. Suppose its cylindrical component is ${\mf w} = (u, \Phi, \Psi)$, which represents an element in ${\mc M}(\lb {\mf x} \rb, \lb \beta'\rb; F_\epsilon, \wt{J})$. Then $\omega\lb \beta'\rb \leq \omega \lb \beta\rb$. By Proposition \ref{prop52} and the special property of $F_\epsilon$, we have
\begin{align*}
{\mc A}_H \lb {\mf x} \rb \geq - R_{\ov{H}} - \epsilon - A\lb \beta' \rb \geq -R_{\ov{H}} - \epsilon - A \lb \beta \rb.
\end{align*}
Then by the definition of the filtration on $CQF(L)$ we see \eqref{equation63} is true.
\end{proof}

The following is an immediate consequence of Proposition \ref{prop64}.
\begin{cor}\label{cor65}{\rm (}cf. \cite[Proposition 18.19]{FOOO_spectral}{\rm )} For each $a \in VHF(M)$, any Hamiltonian $\ov{H} \in {\mc H}(\ov{M})$ and any $\wt\phi \in \wt{\rm Ham}(\ov{M})$, we have
\begin{align*}
\begin{split}
c(a, \ov{H}) \geq &\ - E_\infty^+(\ov{H}; \ov{L}) + c_{L^b}( \mbms{co}(a)),\\
c(a, \wt\phi) \geq &\ - e_\infty^+(\wt\phi; \ov{L}) + c_{L^b}( \mbms{co}(a)).
\end{split}
\end{align*}
\end{cor}

\begin{proof}[Proof of Theorem \ref{thm61}] Similar to the proof of \cite[Theorem 18.8]{FOOO_spectral}, we only need to prove the $\upmu_{\mbms{e}}$-heaviness. Choose a normalized Hamiltonian $\ov{H}\in {\mc H}(\ov{M})_0$ and take
\begin{align*}
{\ov{H}}_{(n)}(t, x) = n \ov{H}\Big( 2\pi \big( \frac{nt}{2\pi} - \big[ \frac{nt}{2\pi} \big]\big), x\Big).
\end{align*}
Here $[c]$ is the greatest integer which is no greater than $c$. Then $\wt\phi_{\ov{H}_{(n)}} = ( \wt\phi_{\ov{H}})^n$. Apply Corollary \ref{cor65} to $\ov{H}_{(n)}$, we obtain
\begin{align*}
c(\mbms{e}, (\wt\phi_{\ov{H}})^n) \geq - n \sup\big\{ \ov{H}(t, x)\ |\ t\in S^1,\ x\in \ov{L} \big\} + c_L^b( \mbms{co}(\mbm{e})).
\end{align*}
Then by the definition of $\upmu_{\mbms{e}}$ and $E_\infty^+$, we obtain $\upmu_{\mbms{e}} ( \wt\phi_{\ov{H}}) \geq - {\rm vol}(\ov{M}) E_\infty^+( \ov{H}, \ov{L})$. Take supremum among all normalized $\ov{H}$ generating the same $\wt\phi_{\ov{H}}$, we have
\begin{align*}
\upmu_{\mbms{e}} ( \wt\phi_{\ov{H}}) \geq - {\rm vol}(\ov{M}) e_\infty^+( \wt\phi_{\ov{H}}, \ov{L}),
\end{align*}
i.e., the $\upmu_{\mbms{e}}$-heaviness of $L$.
\end{proof}

%\subsection{Heavy toric fibres}

%We use the same notations as in Appendix \ref{subsectiona6}. Let $\ov{M}$ be a %compact toric manifold obtained by torus reduction from $M = {\mb C}^m$. Let %$\ov{L}_\lambda\subset \ov{M}$ be a toric moment fibre whose invariant lift to %$M$ is the torus $L_\lambda$ given by \eqref{equationa4}. Then $L_\lambda$ has %a natural $G$-equivariant spin structure and each $b \in H^1( \ov{L}_\lambda, %\Lambda_0)$ makes $L_\lambda$ a $G$-Lagrangian brane $L_\lambda^b$. Each %$L_\lambda$ together with the standard complex structure $J_0$ on ${\mb C}^n$ %satisfies ({\bf H5}). Then we have the quasimap $A_\infty$ algebra %$QA_\infty(L^b_\lambda)$ and by Proposition \ref{propa7}, it is weakly %unobstructed, with central charge ${\mf m}_0^b(1) = W_\lambda^G(b) {\bm %1}_{L_\lambda^b}$, where $W_\lambda^G(b)$ is the Hori-Vafa potential.

\subsection{Products of heavy sets}\label{subsection63}

Our last application in concern is a counterpart of \cite[Theorem 5.2]{Entov_Polterovich_3} in the vortex case.  Let $(M^i, \omega^i, \mu^i)$ be a Hamiltonian $G^i$-manifold satisfying {\bf (H1)}--{\bf (H3)} for $i=0,1$. Denote $M = M^0 \times M^1$ and let $G^0\times G^1$ act factor-wisely. Let $H^i \in {\mc H}_G(M^i)$ which descends to $\ov{H}^i \in {\mc H}(\ov{M}^i)$ and $H = (H^0, H^1) \in {\mc H}_G(M)$. Denote $\ov{M} = \ov{M}^0\times \ov M^1$ and $\ov{H} = (\ov{H}^0, \ov H^1)$.

\subsubsection*{The chain complex of the product}

Apply the construction of Section \ref{section2} to $(M, \omega, \mu)$, we have the action functional ${\mc A}_H: {\mc L} \to {\mb R}$. It is easy to see that there is a canonical identification ${\mf L} \simeq {\mf L}^0\times {\mf L}^1$. Denote an element of ${\mf L}$ by $\lb {\mf x}^0, {\mf x}^1 \rb$ for $\lb {\mf x}^0 \rb \in {\mf L}^0$, $\lb {\mf x}^1 \rb \in {\mf L}^1$. Then it is easy to see
\begin{align}
{\mc A}_H \lb {\mf x}^0 , {\mf x}^1 \rb = {\mc A}_{H^0} \lb {\mf x}^0 \rb + {\mc A}_{H^1} \lb {\mf x}^1 \rb.
\end{align}

\begin{thm}\label{thm66} Let $(M^i, \omega^i, \mu^i)$ be a Hamiltonian $G^i$-manifold satisfying {\bf (H1)}--{\bf (H3)} for $i=0,1$. Then $(M, \omega, \mu)$ satisfies {\bf (H1)}--{\bf (H3)} with smooth symplectic quotient $\ov{M}$. Moreover,
\begin{enumerate}
\item If $H^i \in {\mc H}_{G^i}^*(M^i)$, then $\ov{H}$ satisfies {\bf (H4)}, ${\rm Crit} {\mc A}_{H}\simeq {\rm Crit} {\mc A}_{H^0} \times {\rm Crit} {\mc A}_{H^1}$ and ${\sf cz}\lb {\mf x}^0, {\mf x}^1 \rb = {\sf cz}\lb {\mf x}^0 \rb + {\sf cz} \lb {\mf x}^1 \rb$.

\item Let $\lambda > 0$. Suppose $(H^i, J^i)$ is a regular pair relative to $\ov{H}^i$ and $\lambda$,
    %and $J^1 \in \wt{\mc J}^{reg}_{H^1}$.
    Denote $J = (J^0, J^1)$. Then $(H, J)$ is a regular pair relative to $\ov{H}$ and $\lambda$, and for $\lb {\mf x}_\pm \rb = \lb {\mf x}_\pm^0, {\mf x}_\pm^1 \rb \in {\rm Crit} {\mc A}_H$, ${\mc M}_\lambda ( \lb {\mf x}_- \rb, \lb {\mf x}_+ \rb; H, J)$ is smooth, oriented and is diffeomorphic to
\begin{align}\label{equationb7}
{\mc M}_\lambda( \lb {\mf x}_-^0 \rb, \lb {\mf x}_+^0 \rb; H^0, J^0) \times {\mc M} ( \lb {\mf x}_-^1 \rb, \lb{\mf x}_+^1\rb; H^1, J^1).
\end{align}

\item The counting of isolated gauge equivalence classes of connecting orbits modulo translation defines a chain complex $\big( VCF (M; H), \delta_{J, \lambda} \big) $, and there is an isomorphism of ${\mb Z}_2$-graded chain complexes
\begin{align*}
K_{PG}: \big( VCF(M^0; H^0), \delta_{J^0, \lambda}\big) \otimes \big( VCF(M^1; H^1), \delta_{J^1, \lambda} \big) \simeq \big( VCF (M; H), \delta_{J, \lambda} \big).
\end{align*}
$K_{PG}$ induces an isomorphism on homology
\begin{align}\label{equation66}
\mbms{K}_{\mbms{PG}}: VHF(M^0) \otimes VHF(M^1) \simeq VHF(M).
\end{align}

\item $\mbms{1}_{\mbms H}^0 \otimes \mbms{1}_{\mbms H}^1$ is mapped via $\mbms{K}_{\mbms{PG}}$ to the multiplicative identity $\mbms{1}_{{\mbms H}} \in VHF(M)$.
\end{enumerate}
\end{thm}

\begin{proof}
The fact that $(M, \omega, \mu)$ satisfies {\bf (H1)}--{\bf (H3)} and the first item are straightforward. For (2), let ${\mf w} = ({\mf w}^0, {\mf w}^1)$, where ${\mf w}^i = (u^i, \Phi^i, \Psi^i): \Theta \to M^i \times {\mf g}^i \times {\mf g}^i$.% and $u^1: \Theta \to M^1$.
Then because the action of $G$ splits, (\ref{equation23}) for ${\mf w}$ splits into the equation (\ref{equation23}) for ${\mf w}^i$ in $M^i$
% and (\ref{equationb3}) for $u^1$ in $M^1$.
Therefore, it is easy to see that the moduli space ${\mc M}_\lambda( \lb {\mf x}_- \rb, \lb {\mf x}_+ \rb; H, J)$ is naturally identified with the product (\ref{equationb7}). The linearization $D_{\mf w}$ of (\ref{equation23}) along ${\mf w}$ is the direct sum of the linearization $D_{{\mf w}^i}$ of (\ref{equation23}) along ${\mf w}^i$ for $(H^i, J^i)$. %and the linearization $D_{u^1}$ of (\ref{equationb3}) along $u^1$ for %$(H^1, J^1)$.
Since both $(H^i, J^i)$
 %and $(H^1, J^1)$
 are regular, so is $(H, J)$. So the identification of moduli spaces is actually a diffeomorphism, which respects the orientations in a natural way.

For (3), define the map
\begin{align}\label{equation67}
K_{PG}: VCF(M^0; H^0) \otimes VCF(M^1; H^1)\to VCF(M; H)
\end{align}
induced from $\lb {\mf x}^0 \rb \otimes \lb {\mf x}^1 \rb \mapsto \lb {\mf x}^0, {\mf x}^1 \rb$. This is an isomorphism of $\Lambda$-modules. Note that by transversality, the moduli space ${\mc M}_\lambda( \lb {\mf x}_- \rb, \lb {\mf x}_+ \rb; H, J)$ is one-dimensional only when
\begin{align*}
\lb {\mf x}_-^0\rb = \lb {\mf x}_+^0\rb,\ {\sf cz}\lb {\mf x}_-^1 \rb - {\sf cz} \lb {\mf x}_+^0 \rb=1\ {\rm or}\ {\sf cz}\lb {\mf x}_-^0\rb - {\sf cz}\lb {\mf x}_+^0\rb = 1,\ \lb {\mf x}_-^1 \rb = \lb {\mf x}_+^0 \rb.
\end{align*}
Therefore $K_{PG}$ is a chain map where the differential on the tensor product is
\begin{align*}
\delta_P \lb {\mf x}^0, {\mf x}^1 \rb = K_{PG} \Big(\big( \delta_{J^0, \lambda} \lb {\mf x}^0\rb\big) \otimes \lb {\mf x}^1 \rb	 + (-1)^{{\sf cz} \lb {\mf x}^0 \rb} \lb {\mf x}^0 \rb \otimes \big( \delta_{J^1,\lambda} \lb {\mf x}^1 \rb \big)\Big).
\end{align*}
Therefore (\ref{equation67}) is an isomorphism of chain complexes inducing the isomorphism \eqref{equation66}.

Lastly, we consider the chain level representative of $\mbms{1}_{\mbms H}$. The equation on cappings (\ref{equation218}) also splits into two parts and the transversality follows from the transversality of each component. The moduli space ${\mc C}( \lb {\mf x}^0, {\mf x^1} \rb; \wt{H}, \wt{J})$ is diffeomorphic to the product ${\mc C}(\lb {\mf x}^0 \rb; \wt{H}^0, \wt{J}^0)\times {\mc C}( \lb {\mf x}^1 \rb; \wt{H}^1,\wt{J}^1)$, which is zero-dimensional only when ${\sf cz} \lb {\mf x}^0 \rb = \ov{m}^0$ and ${\sf cz} \lb {\mf x}^1 \rb = \ov{m}^1$, where $2\ov{m}^i$ is the dimension of $\ov{M}^i$. Hence
\begin{align*}
{\bm 1}_H = \sum_{{\sf cz} \lb {\mf x}^0 \rb = \ov{m}^0 \atop {\sf cz} \lb {\mf x}^1 \rb = \ov{m}^1} \# \Big( {\mc C}(\lb {\mf x}^0 \rb; \wt{H}^0, \wt{J}^0)\times {\mc C}( \lb {\mf x}^1 \rb; \wt{H}^1,\wt{J}^1)\Big) \lb {\mf x}^0, {\mf x}^1 \rb.
\end{align*}
Via the inverse of the chain level map (\ref{equation67}), it is identified with ${\bm 1}_{H^0}^0 \otimes {\bm 1}_{H^1}^1$. Passing to homology, we see that (4) holds.\end{proof}

\subsubsection*{Spectral invariants and quasi-states of products}

We prove
\begin{prop}\label{prop67}{\rm (}cf. \cite[Theorem 5.2]{Entov_Polterovich_3}{\rm )}
For $a^i \in VHF(M^i)$ and $H^i \in {\mc H}_{G^i}(M^i)$, we have %$H^1 \in {\mc H}(M^1)$,
\begin{align}\label{equation68}
c(\mbms{K}_{\mbms{PG}}(a^0\otimes a^1), (H^0, H^1)) = c(a^0,H^0) + c(a^1, H^1).
\end{align}
\end{prop}

\begin{proof}
It is essentially the algebraic result \cite[Theorem 5.2]{Entov_Polterovich_3} and we only remark on some minor difference in notations and conventions. The Novikov ring used in \cite{Entov_Polterovich_3} is the downward completion of the group ring ${\mc K}_\Gamma$ where ${\mc K}$ is either ${\mb Z}_2$ or ${\mb C}$, and $\Gamma\subset {\mb R}$ is the action spectrum, a countable subgroup of ${\mb R}$, similar as the role of $G(M, \mu)$ in (\ref{equation31}). It is slightly smaller than the universal Novikov ring $\Lambda$. However our theory can be easily restricted to one with coefficient ring ${\mc K}_\Gamma$, since the spectrum is countable for both the vortex Floer theories.
 %of $M$, $M^0$ and the ordinary Hamiltonian Floer theory on $M^1$.
 Therefore, for regular pairs $(H^i, J^i)$ on $M^i$,
  %and $(H^1, J^1)$ on $M^1$,
  the filtered chain complexes $(VCF(M^i; H^i), \delta_{J^i, \lambda})$
  % and $(CF(M^1; H^1), \delta_{J^1,\lambda})$
  are decorated ${\mb Z}_2$-graded complexes in the sense of \cite[Subsection 5.2]{Entov_Polterovich_3}. Therefore by \cite[Theorem 5.2]{Entov_Polterovich_3} and the properties of the spectral numbers, (\ref{equation68}) holds for nondegenerate $H^i$. Since both the left-hand-side and the right-hand-side are Lipschitz continuous, (\ref{equation68}) holds for all $H^i \in {\mc H}_{G^i}(M^i)$.
   %and $H^1 \in {\mc H}(M^1)$.
\end{proof}

The identity elements ${\mbms 1}_{\mbms H}^0$, ${\mbms 1}_{\mbms H}^1$, ${\mbms 1}_{\mbms H}$ induce the quasi-states $\zeta^0$ on $\ov{M}^0$, $\zeta^1$ on $\ov{M}^1$ and $\zeta$ on $\ov{M}$ respectively. By (4) of Theorem \ref{thm66} and Proposition \ref{prop67}, we have
\begin{cor}
For $(\ov{H}^0, \ov H^1) \in C^0(\ov{M}^0) \times C^0(\ov M^1)$, $\zeta(\ov{H}^0,\ov H^1) = \zeta^0(\ov H^0) + \zeta^1(\ov H^1)$.
\end{cor}

Then we have the main theorem of this subsection.
\begin{thm}\label{thm69}
If $\ov{X}^0 \subset \ov{M}^0$ is $\zeta^0$-heavy and $\ov X^1\subset \ov M^1$ is $\zeta^1$-heavy, then $\ov{X} = \ov{X}^0 \times\ov X^1$ is $\zeta$-heavy.
\end{thm}

\begin{proof}
The proof is the same as that of \cite[Theorem 1.7]{Entov_Polterovich_3}. Note that one only needs the Lipschitz continuity and monotonicity of $\zeta$, which were established in our paper.
\end{proof}

\brmk Note that when $G_1$ is trivial and $M^1 = \ov{M}^1$ is compact aspherical, one obtains a ``hybrid case".  An instance is when $\ov M=\ov M^0\times \mathbb{T}^2$ and $\ov X^1$ being a meridian.  Theorem \ref{thm69} implies the product of a heavy set $\ov X^0\subset\ov M^0$ with the meridian is a heavy set hence non-displaceable.  This in turn implies the stably non-displaceability property of a $\zeta$-heavy set as defined in \cite{Entov_Polterovich_3}.

\ermk

\appendix

\section{Woodward's quasimap $A_\infty$ algebra}\label{appendixa}

In this appendix we review the quasimap $A_\infty$ algebra and quasimap Floer homology constructed by Woodward in \cite{Woodward_toric}, and his computation in the toric case.

\subsection{Treed disks}\label{subsectiona1}

A {\it Stasheff tree} is a ribbon tree ${\mb T} = (V, E, {\mf e})$ with $V$ the set of vertices, $E$ the set of edges and ${\mf e}\in E$ a distinguished out-going semi-infinite edge, and a metric $l: E \to [0, +\infty]$. The semi-infinite edges are required to have infinite length. Two Stasheff trees are isomorphic if they become isomorphic after collapsing all edges of length zero. For a Stasheff tree ${\mb T} = (V, E, {\mf e})$, we use ${\mf m}$ to denote the unique vertex adjacent to ${\mf e}$. All other semi-infinite edges ${\mf e}_1, \ldots, {\mf e}_n$ are regarded as in-coming and has a canonical order. Every edge then has an induced orientation; as a convention, $E$ consists of all {\it oriented} edges. For each $\alpha\in V \setminus \{\mf m\}$, we denote by $\alpha'$ the unique vertex in $V$ such that there is an oriented edge from $\alpha$ to $\alpha'$, denoted as $\alpha \succ\alpha'$. For each $e\in E$, denote by $e^- \in V$ ($e^+$ resp.) the source (target resp.) of $e$ so that $e= e^- \succ e^+$. $e^- (e^+, \text{ resp.})$ does not exist if $e$ is an incoming (out-going, resp.) edge.

For $k \geq 0$, let ${\mc T}_d$ be the moduli space of Stasheff trees with $k$ incoming semi-infinite edges. There are various inclusion maps such as
\begin{align}\label{equationa0}
{\bm i}_{k_1, k_2}^j: {\mc T}_{k_1} \times {\mc T}_{k_2} \to {\mc T}_{k_1 + k_2 - 1}
\end{align}
(i.e., composing the $j$-th
%, $j+1$-st, $\cdots$, $j+k_1 - 1$-st
in-coming semi-infinite edges on $\cT_{k_2}$ and the out-going semi-infinite edge of ${\cT_{k_1}}$, for any $1\le j \le k_2$).

A {\it treed disk} is modelled on a Stasheff tree  ${\mb T} = (V, E, {\mf e})$ along with a partition of the vertices $V=V_0\cup V_1$ and consists of the following collection of data
\begin{align*}
{\mc D} = \Big( (D_\alpha)_{\alpha \in V_1},\ (N_\beta)_{\beta\in V_0},\ (I_{e})_{e\in E},\ (z_e)_{e\in E,\ e^+\in V_1} \Big),
\end{align*}
where
\begin{enumerate}
\item $V = V_1 \cup V_0$ is a disjoint partition of $V$ into ``disk vertices'' and ``nodal vertices'', respectively (the disk and black dots in Figure 1, respectively).  Nodal vertices are required to have valence 2. For each $\alpha \in V_1$, $D_\alpha$ is a copy of the unit disk in ${\mb C}$; for each $\beta \in V_0$, $N_\beta$ is a point.

\item For each edge $e\in E$, $I_e = [\epsilon_e^-, \epsilon_e^+]$ is a closed interval (where $\epsilon_e^\pm$ are allowed to be $\pm \infty$), whose length $\epsilon_e^+ - \epsilon_e^-$ is induced from the Stasheff tree structure. They are subject to the following condition: $e^\pm \in V_1$ (resp. $e^\pm \in V_0$) if and only if $\epsilon_e^\pm$ is finite (resp. infinite).

\item For each $\alpha\in V_1$, a cyclic order for all edges adjacent to $\alpha$ is assigned.  The order starts from the unique edge $e_\alpha$ such that $e_\alpha^-=\alpha$.

\item For each edge $e\in E$ with $e^+ \in V_1$, $z_e\in \partial D_{e^+} \setminus \{1\}$ is a point. They are subject to the following condition: for each $\alpha \in V$, the collection of $z_e$'s for all $e\in E$ with $e^+ = \alpha$ are pairwise distinct and respect the cyclic order induced from (3).
\end{enumerate}

One then forms a topological space $\mathcal{D}=\cup D_\alpha\bigcup\cup N_\beta\bigcup I_e/\sim$, where we require $\epsilon_e^+\sim z_e$ if $e^+\in V_1$, $\epsilon_e^-\sim\{1\}\in \partial D_{e^-}$ if $e^-\in V_0$, and $\epsilon_e^\pm\sim N_\beta$ if $e^\pm\in V_0$.

A treed disk is called {\bf stable} if it has no infinitesimal automorphism. The forgetful map which shrinks the disks to vertices assigns to each treed disk its underlying Stasheff tree.

\begin{rem}
This definition differs slightly from Woodward's definition. Namely, we require that finite ends of edges must attach to disks, which is a condition only required for stable treed disks in \cite{Woodward_toric}.
\end{rem}

\subsection{Holomorphic treed quasidisks}\label{subsectiona2}

Let $(M, \omega, \mu)$ be a Hamiltonian $G$-manifold with a $G$-invariant almost complex structure $J_0$. Let $L \subset \mu^{-1}(0)$ be a $G$-Lagrangian. Let $(F, B)$ be a Morse-Smale pair on $\ov{L}$. That means, $F: \ov{L} \to {\mb R}$ is a Morse function, $B$ is a Riemannian metric on $\ov{L}$, such that the intersection of any unstable manifold and any stable manifold of the negative gradient flow line of $F$ with respect to $B$ is transverse. (Compare with \cite[Definition 3.1]{Woodward_toric}, which only differ in wording and the direction of the flow) a {\bf holomorphic treed quasidisk} is a treed disk
\begin{align*}
{\mc D} = \Big( (D_\alpha)_{\alpha \in V_1},\ (N_\beta)_{\beta\in V_0},\ (I_{e})_{e\in E},\ (z_e)_{e\in E,\ e^+\in V_1} \Big),
\end{align*}
together with a collection
\begin{align*}
{\bf u} = \Big( (u_\alpha)_{\alpha \in V_1},\ (v_e)_{e\in E} \Big)
\end{align*}
where for each $\alpha \in V_1$, $u_\alpha: (D_\alpha, \partial D_\alpha) \to (M, L)$ is a $J_0$-holomorphic disk; for each $e\in E$, $v_e: I_e \to \ov{L}$ is a negative gradient flow line of $F$. They are subject to the following continuity condition:
\begin{enumerate}
\item if $e^- \in V_1$, then $\ov{u_{e^-}(1)} = v_e(\epsilon_e^-)$; if $e^+ \in V_1$, then $\ov{u_{e^+}(z_e)} = v_e(\epsilon_e^+)$. Here $\ov{x}\in \ov{L}$ is the $G$-orbit of $x \in L$.

\item if $e^\pm = \beta \in V_0$, then there is another edge $e'$ with $(e')^\mp = \beta$. We require $v_e(\epsilon_e^{\pm}) = v_{e'}(\epsilon_{e'}^{\mp})\in {\rm Crit} F$.
\end{enumerate}

A holomorphic treed quasidisk is {\bf stable} if for every $\alpha \in V_1$ with $u_\alpha$ being constant, the valence of $\alpha$ is at least three, and, for each edge $e$ with infinite length, $v_e$ is nonconstant.

\subsection{Perturbations}

To achieve transversality, we need to perturb the gradient flows (because we want to keep the almost complex structure $J_0$ satisfying {\bf (H5)}). In order to prove $A_\infty$ relations one needs to perturb in a consistent way for all possible trees.
\begin{defn} (see \cite[Definition 2.6, 2.8]{Abouzaid_plumbing} and \cite[Section 3.4]{Woodward_toric}) A {\bf perturbation datum} on a Stasheff tree ${\mb T} = (V, E, {\mf e})$ is, for each edge $e\in E$, a pair of elements
\begin{align*}
F_e' \in C^\infty \big( I_e \times \ov{L} \big),\ B_e' \in C^\infty \big( I_e \times {\rm Sym}^2(T^* \ov{L}) \big)
\end{align*}
which vanishes near the infinities of all edges.

A {\bf universal consistent perturbation datum} is  %Recall that we call %a vertex an \textbf{input} if it is a leave %in the Stasheff tree and %belongs to $V_0$.  Then the Stasheff trees are %naturally divided into %equivalence classes by the number of inputs, %where the totality of %$d$-input trees are denoted as $\ov\cT_d$.
a choice ${\bf X}^{\mc T}$ of a smooth family of perturbation data for elements in ${\mc T}_d$ for every $d$, which is compatible with the gluing maps of the form \eqref{equationa0} and is invariant under automorphisms of each tree.
\end{defn}

To obtain a regular universal consistent perturbation data, the general idea is to start from the lowest strata and make the regular choice component-wisely.  Then an induction argument along the gluing process to higher strata leads to such a choice. See \cite{Seidelbook}, or \cite{Woodward_toric_corrected} for a concrete treatment in our current case.

Here we recall briefly from \cite{Woodward_toric_corrected} for readers' convenience.  In the case of holomorphic treed quasidisks, one first replaces the negative gradient flow equation of the pair $(F, B)$ on each edge by the equation
\begin{align}\label{equationa1}
u'(t) + \nabla^{B + B_e'} (F + F_e')(u) = 0
\end{align}
on the corresponding edge. The corresponding objects are called {\bf perturbed holomorphic treed quasidisks}. On the other hand, each perturbed holomorphic treed quasidisk represents a class $\beta \in D_2(M, L)$ (see the notation introduced in Subsection \ref{subsection51}). The isomorphisms of (perturbed) holomorphic treed quasidisks, and the topology of the moduli space $\ov{MW}_k (L, \beta)$ of isomorphism classes of those holomorphic chained quasidisks representing $\beta \in D_2(M, L)$ (with $k$ in-coming semi-infinite edges) can be defined in a standard way so we omit them. Moreover, for any perturbed holomorphic treed quasidisk, one can linearize the defining equations on each disk or flow-line component, which is a linear Fredholm operator between suitable Banach spaces. We say that a perturbed holomorphic treed quasidisk is {\bf regular} if the linearization is surjective; a universal consistent perturbation datum is {\bf regular} if every stable perturbed holomorphic treed quasidisk is regular. By energy quantization of holomorphic disks, for fixed $\beta$, there is an upper bound of the number of disk vertices that can appear in an element of $\ov{MW}_k(L, \beta)$, and hence a bound on the number of different possible combinatorial types. This allows one to construct a regular perturbations datum by induction.

We summarize the main results of \cite[Subsection 3.4]{Woodward_toric}.

\begin{thm}\label{thma3}
Suppose $(M, \omega, \mu)$ satisfies {\bf (H1)--(H3)} and all stable $J_0$-holomorphic disks with boundaries in $L$ are regular, then the following is true.
\begin{enumerate}
\item Any sequence of stable perturbed holomorphic treed quasidisks with bounded energy and bounded number of semi-infinite edges has a convergent subsequence.

\item For each sufficiently large $v$, for generic universal consistent perturbation datum $(F_e', B_e')$, for all $\beta\in D_2(M, L)$ with $v(\beta) \leq v$, every perturbed holomorphic treed quasidisk with $k$ in-coming semi-infinite edges representing $\beta$ is regular.

\item For $k\geq 0$, $\beta \in D_2(M, L)$ and a generic universal consistent perturbation datum $(F_e', B_e')$. The one-dimensional component $\ov{MW}_k(L, \beta)_1\subset \ov{MW}_k(L,\beta)$ has the structure of a compact one-dimensional manifold with boundary
\begin{align*}
\partial \ov{MW}_k(L, \beta)_1 = \bigcup_{j+i \leq k} \bigcup_{\beta =\beta_1 + \beta_2} MW_{k-i+1}(L, \beta_1)_0 \times_{ev_j \times ev_0} MW_i(L,\beta_2)_0.
\end{align*}
\end{enumerate}
\end{thm}

\subsection{The quasimap $A_\infty$ algebra}

Now, choose $F$ such that it has a unique maximum point $y_{\max}$. Assume the hypothesis of Theorem \ref{thma3} is satisfied and we have chosen a generic universal consistent perturbation datum.

Suppose we have a chosen a brane structure on $L$, i.e., a $G$-equivariant spin structure and a flat $\Lambda$-bundle over $L$, where the flat $\Lambda$-bundle is prescribed by an element $b \in H^1(L, \Lambda_0)$. We denote the Lagrangian brane by $L^b$. Then the various moduli spaces are oriented by the relative spin structure and all zero dimensional moduli spaces of perturbed holomorphic treed quasidisks has a well-defined counting in ${\mb Z}$; or equivalently, every isolated isomorphism class $[{\bf u}]$ of perturbed holomorphic treed quasidisk has a sign $\epsilon([{\bf u}]) \in \{1, -1\}$.

Consider the Floer chain group, for which we only need the ${\mb Z}_2$-grading:
\begin{align*}
CQF(L) = \bigoplus_{j\in {\mb Z}} CQF_j(L),\ CQF_j(L) = \bigoplus_{y \in {\rm Crit}_j F} \Lambda \langle y \rangle
\end{align*}
For $k \geq 0$, the compositions ${\mf m}_k^b$ of $CQF(L)$ is defined by
\begin{align}\label{equationa2}
{\mf m}_k^b ( \langle x_1\rangle, \ldots, \langle x_k \rangle ) = \sum_{[{\bf u}] \in \ov{MW}_k(x_0, x_1, \ldots, x_n)_0} (-1)^{\heartsuit} \epsilon([{\bf u}]) q^{-A({\bf u})}	 e^{\langle b, \partial {\bf u} \rangle} \langle x_0 \rangle.
\end{align}
Here $\heartsuit = \sum_{i=1}^k i (\ov{m}- {\sf ind}x_i)$ and $A({\bf u}) = \omega \cap [{\bf u}]$ is the symplectic area.

\begin{thm}\cite[Theorem 3.6]{Woodward_toric} $\{{\mf m}_k^b\}_{k\geq 0}$ defines the structure of a ${\mb Z}_2$-graded $A_\infty$ algebra on $CQF(L)$. Moreover, for any $x \in {\rm Crit} F$, one has
\begin{align*}
(-1)^{\ov{m} - {\rm ind} x} {\mf m}_2^b( \langle x \rangle, \langle y_{\rm max} \rangle ) = {\mf m}_2^b( \langle y_{\rm max} \rangle, \langle x \rangle ) = \langle x \rangle.
\end{align*}
\end{thm}

\begin{rem}
Here we use the downward gradient flow, so the signs differ from Woodward's convention, which used the upward gradient flow.
\end{rem}

\begin{rem}
Indeed, the unique maximum $\langle y_{\max}\rangle \in CQF(L)$ of $F$ doesn't give a strict unit of the quasimap $A_\infty$ algebra, which was claimed in \cite{Woodward_toric}. The readers may see \cite[Section 3]{Woodward_toric_corrected} for clarification. However, if $\big(CQF(L), \{{\mf m}_k^b \}_{k \geq 0} \big)$ is weakly unobstructed (see below), one can still define the Floer homology.
\end{rem}

\subsection{Quasimap Floer homology}

The counting of $MW_0(L)_0$ defines the central charge ${\mf m}_0^b (1) \in \Lambda$ of the quasimap $A_\infty$ algebra, i.e.,
\begin{align}\label{equationa3}
{\mf m}_0^b(1) =\sum_{y \in {\rm Crit} F} \sum_{[{\bf u}] \in \ov{MW}_0(L)_0} \epsilon([{\bf u}]) q^{- A({\bf u})} e^{\langle b, \partial {\bf u} \rangle} \langle y \rangle.
\end{align}
We say that $QA_\infty(L^b)$ is {\it weakly unobstructed} if ${\mf m}_0^b (1)$ is a multiple of ${\bm 1}_L = \langle y_{\max}\rangle$ (though ${\bf 1}_L$ is not a strict unit); or in other words, (\ref{equationa3}) only has contributions of quasidisks of Maslov index $2$. It is easy to see the following
\begin{prop}\cite[Proposition 3.7]{Woodward_toric_corrected}
Suppose $L$ and $J_0$ satisfies {\bf (H5)}, then $QA_\infty(L^b)$ is weakly unobstructed for any brane structure $b$ on $L$.
\end{prop}

When $QA_\infty(L^b)$ is weakly unobstructed, ${\mf m}_1^b \circ {\mf m}_1^b = 0$ and $(CQF(L), {\mf m}_1^b)$ becomes a chain complex. The quasimap Floer homology $HQF(L^b)$ is defined as the homology of $(CQF(L), {\mf m}_1^b)$.

\subsection{Toric case}\label{subsectiona6}

Let $M = {\mb C}^m$ and let $G \subset U(1)^m$ be a subtorus, acting on $M$ with a proper moment map $\mu: M \to {\mf g}^*$. Suppose $G$ acts on $\mu^{-1}(0)$ freely, which results in a smooth compact toric manifold $\ov{M} = \mu^{-1}(0)/G$. Let $T = U(1)^m/ G$ which acts on $M$ with moment map $\Psi: \ov{M} \to  {\mf t}^*$. The moment polytope of $\ov{M}$ is
\begin{align*}
\Delta := \Psi( \ov{M}) = \big\{ \lambda \in {\mf t}^* \ |\ l_i(\lambda) \geq 0 \big\},\ l_i(\lambda) = 2\pi \big( \langle \lambda, v_i \rangle - c_i \big),\ i=1, \ldots, m.
\end{align*}

For each $\lambda \in {\rm Int} \Delta$, $\ov{L}_\lambda := \Psi^{-1}(\lambda) \subset \ov{M}$ is a Lagrangian torus, whose preimage $L_\lambda \subset {\mb C}^m$ is the standard torus
\begin{align}\label{equationa4}
L_\lambda = \big\{ (z_1, \ldots, z_m)\ |\ |z_j|^2/2 = l_j(\lambda)/ 2\pi,\ j=1, \ldots, m\big\} \simeq \prod_{j=1}^m \sqrt{ \frac{ l_j(\lambda)}{\pi}} S^1.
\end{align}
Let $J_0$ be the standard complex structure on ${\mb C}^m$. By \cite[Proposition 6.1]{Woodward_toric}, any holomorphic quasidisk in $M$ with boundary in $L_\lambda$ is given by the Blaschke product
\begin{align*}
u(z) = \left( \sqrt{\frac{ l_j(\lambda)}{\pi}} \prod_{k=1}^{d_j} \frac{ z- \alpha_{j, k}}{1 - \ov{\alpha_{j, k}} z } \right)_{j=1}^m
\end{align*}
By \cite[Corollary 6.2]{Woodward_toric}, every stable $J_0$-holomorphic disk is regular. Then one has the quasimap $A_\infty$ algebra $QA_\infty(L_\lambda^b)$. Moreover, the minimal Maslov index of quasidisks is two. So we have
\begin{prop}\label{propa7}{\rm (}\cite[Corollary 6.4]{Woodward_toric} {\rm )} The quasimap $A_\infty$ algebra is weakly unobstructed and its central charge is
\begin{align*}
{\mf m}_0^{\lambda,b}(1) = W_\lambda^G(b) {\bm 1}_{L_\lambda},\ W_\lambda^G(b) = \sum_{j=1}^m e^{\langle v_i, b \rangle} q^{-l_i(\lambda)}
\end{align*}
\end{prop}
The function $W_\lambda^G(b)$ is the Hori-Vafa potential \cite{Hori_Vafa}.

\begin{thm}[\cite{Woodward_toric}, Theorem 6.6]\label{thma8} If $b\in H^1(L_\lambda, \Lambda_0)$ is a critical point of $W_\lambda^G$, then $HQF(L_\lambda^b; \Lambda)\simeq H(L_\lambda; \Lambda)$. In particular, the quasimap Floer homology of $L_\lambda^b$ doesn't vanish.
\end{thm}

\bibliography{../../symplectic_ref}
	
\bibliographystyle{amsplain}

\end{document}